\documentclass[12pt,a4paper]{article}

\usepackage[utf8]{inputenc}
\usepackage[english]{babel}
 \usepackage{mathrsfs}
\usepackage[T1]{fontenc}
\usepackage{amsmath}
\usepackage{amsfonts}
\usepackage{amssymb}
\usepackage{makeidx}
\usepackage{color}
\usepackage{amsthm}
\usepackage{graphicx}
\usepackage{hyperref}

\usepackage[all]{xy}

\newcounter{counter_conj}

\newtheorem{thm}{Theorem}[section]
\newcounter{counter_thm_int}
\newtheorem{thm_int}[counter_thm_int]{Theorem}
\newtheorem{cor_int}[counter_thm_int]{Corollary}

\newtheorem{conj}[counter_conj]{Conjecture}

\newtheorem{lem}[thm]{Lemma}
\newcounter{counter_conj-lem}
\newtheorem{conj-lem}[thm]{Conjecture-Lemma}

\newtheorem{prop}[thm]{Proposition}

\newcounter{counter_conj-prop}
\newtheorem{conj-prop}[thm]{Conjecture-Proposition}

\newcounter{counter_conj-thm}
\newtheorem{conj-thm}[thm]{Conjecture-Theorem}

\newtheorem{cor}[thm]{Corollary}
\newtheorem{defi}[thm]{Definition}

\theoremstyle{remark}

\newtheorem{rem}[thm]{Remark}

  \newcommand{\numbpf}[1]{\num_{\BPF}^{#1}}
  
  \newcommand{\normbpf}[1]{\left\lVert#1\right\rVert_{\BPF}}
  \newcommand{\normsigma}[2]{\left\lVert#1\right\rVert_{\Sigma,#2}}
  \newcommand{\normbpfs}[1]{\left\lVert#1\right\rVert_{\BPF}^\vee}

\newcommand{\C}{\mathbb{C}}
\newcommand{\R}{\mathbb{R}}
\newcommand{\Q}{\mathbb{Q}}
\newcommand{\N}{\mathbb{N}}
\newcommand{\A}{\mathbb{A}}
\newcommand{\Z}{\mathbb{Z}}
\newcommand{\D}{\mathbb{D}}
\newcommand{\Pp}{\mathbb{P}}

\DeclareMathOperator{\FS}{\mathcal{F}_{\Sigma}}

\newcommand{\K}{K}

\newcommand{\cM}{ \mathcal{M}}
\newcommand{\cP}{ \mathcal{P}}
\newcommand{\Ss}{ \mathcal{S} }
\newcommand{\X}{ \mathcal{X} }

\newcommand{\Y}{ \mathcal{Y} }
\newcommand{\cL}{ \mathcal{L} }
\newcommand{\cO}{ \mathcal{O} }

\newcommand{\fA}{\mathfrak{A}}
\newcommand{\fB}{\mathfrak{B}}

\newcommand{\om}{{\omega}}

\DeclareMathOperator{\ord}{ord}

\DeclareMathOperator{\Crit}{Crit}
\DeclareMathOperator{\CV}{CV}
\DeclareMathOperator{\codim}{codim}

\DeclareMathOperator{\Id}{Id}

\DeclareMathOperator{\Aut}{Aut}

\DeclareMathOperator{\Spec}{Spec}

\DeclareMathOperator{\psef}{Psef}

\DeclareMathOperator{\BPF}{BPF}
\DeclareMathOperator{\cBPF}{c-BPF}

\DeclareMathOperator{\nef}{Nef}
\DeclareMathOperator{\cNef}{c-Nef}

\DeclareMathOperator{\proj}{Proj}

\DeclareMathOperator{\nums}{N^1_{\Sigma}}
\newcommand{\numbpfs}[1]{\num^{#1,\vee}_{\BPF}}

\DeclareMathOperator{\Vect}{Vect}

\DeclareMathOperator{\num}{N}

\DeclareMathOperator{\wnum}{w-N}

\DeclareMathOperator{\cnum}{c-N}

\DeclareMathOperator{\LL}{L}

\author{Nguyen-Bac Dang, Charles Favre}

\title{Spectral interpretations of dynamical degrees and applications}

\begin{document}

\maketitle

\begin{abstract}
We prove that dynamical degrees of rational self-maps on projective varieties can be interpreted as spectral radii of naturally defined operators
on suitable Banach spaces. Generalizing Shokurov's notion of $b$-divisors, we consider the space of $b$-classes of higher codimension cycles, and endow this space with various Banach norms. Building on these constructions, we design a natural extension to higher dimensions of the Picard-Manin space introduced by Cantat and Boucksom-Favre-Jonsson in the case of surfaces. 
 We prove a version of the Hodge index theorem, and a surprising compactness result
in this Banach space. We use these two theorems to infer a precise control of the sequence of degrees of iterates of a map under the assumption $\lambda_1^2>\lambda_2$ on the dynamical degrees. As a consequence, we obtain that the dynamical degrees of an automorphism of the affine $3$-space are all algebraic numbers.
\end{abstract}

\tableofcontents

%%%%%%%%%%%%%%%%%%%%%%%%%%%%%%%%%%%%%%%%%%%%%%%%%%%%%%%%%%%%%%%%%%%%%%%

\section*{Introduction}
\label{sec:intro}
\addcontentsline{toc}{section}{\nameref{sec:intro}}

Let $f\colon X\dashrightarrow X$ be a dominant rational self-map of a normal projective  (irreducible) variety $X$ of dimension $d$ defined over a field $\K$ of characteristic zero. 
We may define its $k$-th degree with respect to a given  ample class $\om$ 
as the intersection product $\deg_{k,\om}(f) := (f^{*}\om^k\cdot \om^{d-k})\in (0,+\infty)$ for any $0\le k\le d$. 
The sequences of degrees $\{\deg_{k,\om}(f^n)\}_n$ give a measure of dynamical complexity. 
Its growth is indeed a fundamental invariant of $f$ that governs many of its dynamical features: e.g.
its entropy~\cite{MR2026895,MR1155288,dinh_sibony_reg_currents_entropy,zbMATH05004659,zbMATH05723568,de_thelin_vigny_2010}; the presence of invariant fibrations~\cite{MR1751620,dinh_nguyen_2011,MR3277202,zbMATH06522286,zbMATH06911551,bianco}; 
or its periodic points~\cite{xie_periodic,blms.12082}. 
It also controls the behaviour of the heights of iterates of points when $f$ is defined over a number field \cite{kawaguchi_silverman,matsuzawa_sano_shibata_surfaces,matsuzawa_sano_shibata}.

When $\K=\C$ and $X = \mathbb{P}^d$, it was proven by Russakovskii and Shiffman~\cite{MR1488341} (see also~\cite{MR1704282})
that the sequence of degrees is sub-multiplicative.
The sub-multiplicativity was generalized later on  by Dinh and Sibony~\cite{dinh_sibony_reg_currents_entropy} for arbitrary $X$ (see Theorem~\ref{thm:sbmlt} below for the exact statement), and
 so  one can define the dynamical degrees
\[ \lambda_k(f) := \lim_{n\to\infty} \deg_{k,\om}(f^n)^{1/n}\ge 1, \text{ for }  k=0, \cdots, d~.\] 
These dynamical degrees do not depend on the choice of polarization $\om$ and are invariant under any birational conjugacy. 
Khovanskii-Teissier's inequalities imply the log-concavity of the
dynamical degrees, in particular we have $\lambda_2(f)\le \lambda_1(f)^2$. 

The sub-multiplicativity property was extended to rational maps defined
over an arbitrary field by Tuyen Truong~\cite{tuyen2}.  The first author recently gave another proof of this fact~\cite{dang_degree_normal} using the notion of base-point free (BPF) classes which originated in~\cite{MATH06741284} as a substitute for nef classes in higher codimensions. The BPF property turns out to be the right notion of positivity for cycles allowing one 
to generalize Siu's inequalities 
to higher codimensions and to obtain good comparison results with complete intersection classes, see~\cite{xiao,popovici} or \cite[\S 3.4]{dang_degree_normal}.

In dimension $2$, under the assumption that $\lambda_2(f)< \lambda_1(f)^2$, Boucksom, Jonsson and the second author~\cite{boucksom_favre_jonsson_deggrowth}  proved
 that $\deg_1(f^n) = C \lambda_1(f)^n + O(\lambda_1^{n/2})$ for some $C>0$. To accomplish this, they constructed a suitable
 Hilbert space in the space of divisors on an "infinite blow-up"\footnote{A similar idea was developed for the Newton map associated to pairs of quadratic maps in $\C^2$ by Hubbard and Papadopol in~\cite{MR2376103}.} of $X$. They proved that $f$ induces a bounded operator on that space
and analyzed the spectral properties of this operator using intersection theory
 on birational models over $X$. This Hilbert space was simultaneously introduced by Cantat in his work on the Cremona group~\cite{cantat_bir_surfaces} 
where it was referred to as the Picard-Manin space\footnote{Manin introduced divisors on the infinite blow-up of a rational surface and described the action of the Cremona group on them, see~\cite[\S 34]{manin_cubic}}. 
By exploiting the geometric properties of this space, Xie proved a semicontinuity property for the dynamical degrees \cite{xie_periodic}, and
Blanc and Cantat~\cite{blanc_cantat} obtained a detailed description of the set of dynamical degrees of birational surface maps.
 
---

In higher dimensions, the degree growth of various classes of examples have been computed both in the physics and mathematics literatures, see~\cite{MR1254598,MR1694727,MR1704016,MR1732080,MR2019672,MR2087743,MR2220002,MR3335696} and~ \cite{MR2111418,MR2428100,bedford_kim_degree_matrix,MR2582437,bedford_truong,favre_wulcan,lin,MR3570024,MR3725886,deserti_degree_examples,blanc_santen_automorphism,MR4030548}. However 
 understanding degree growth and
computing  dynamical degrees remain very hard in full generality. We rely in this paper on methods from functional analysis to address this problem.
 
We introduce several new graded Banach subspaces of the space of numerical classes of cycles in all birational models of $X$,
on which $f$ induces natural bounded operators whose spectral radii are given by its dynamical degrees (Theorem~\ref{thm:dey-degree-interpret}). Using the resolvent method~\cite{zbMATH03142405,yosida},  we go further and build BPF eigenclasses  whose associated eigenvalues equal
   $\lambda_k(f)$ (Theorem~\ref{thm:construct-eigenvec-pushpull}).  
 In the sequel to this paper \cite{dang_favre}, we will further explore how these Banach spaces interact with one another.

Our main contribution is to generalize the Picard-Manin space to arbitrary dimensions (see \S\ref{section_nsigma}).
We denote this space\footnote{The symbol $\Sigma$ in $\nums$ stands for surface.} by $\nums (\X)$ and endow it with a norm $\normsigma{\cdot}{\om}$ which turns it into a Banach space. Here $\X$ stands for the Riemann-Zariski space of $X$ defined as the projective limit of all birational models over $X$ (see \S\ref{sec:numerical_b_class}).
An element in $\nums (\X)$ is a collection of classes in the Néron-Severi space of each of these models, the classes of which are compatible under push-forwards.
We further assume that their restriction to any sufficiently general surface $S\subset X$ belongs to the Picard-Manin space of $S$. The precise definition of  $\nums (\X)$ is quite technical, 
and we need to take all surfaces $S$ whose numerical class is basepoint free in the sense of~\cite{dang_degree_normal}.
One can view a class in $\nums (\X)$ as one associated to a $b$-divisor as defined in~\cite{shokurov_3_fold,shokurov_prelimiting,corti_3_fold,zbMATH05541351,boucksom_fernex_favre}
and satisfying some growth condition. 

We prove that  classes in $\nums (\X)$ still satisfy a suitable version of the Hodge index theorem (Theorem~\ref{thm_hodge_index});
and that $\nums (\X)$ exhibits surprising compactness properties (Theorem~\ref{thm_compactness}) which in special situations allow one to deduce convergence
in norm from weak convergence. 
Using standard arguments from functional analysis, we are then able to describe the main spectral properties of a naturally defined bounded operator $f^*\colon\nums (\X)\to\nums(\X)$. This yields our first theorem. 

\begin{thm_int} \label{thm_int_gap}
 Let $X$ be a normal projective variety defined over a field of characteristic $0$, and let $\om$ be any ample class on $X$.
 For any dominant rational self-map $f\colon X\dashrightarrow X$ satisfying $\lambda_1(f)^2 > \lambda_2(f)$, 
there exists a constant $C> 0$ such that
\[
\deg_{1,\om}(f^n) = C \lambda_1(f)^n + O(\lambda^n)~,
\]
for any $\lambda$ satisfying  $\sqrt{\lambda_2(f)} < \lambda < \lambda_1(f)$.
\end{thm_int}

This gives in particular a positive answer to~\cite[Question~3.3]{MR2139697}.

This result was known for algebraically stable maps on $3$-folds by Tuyen Truong~\cite{MR3255693} and for toric maps~\cite{favre_wulcan}. Observe that in the latter case (and for maps
on Abelian varieties) the assumption $\lambda_1(f)^2 > \lambda_2(f)$ reflects the fact that some integer-valued matrix  has a unique eigenvalue
whose modulus equals its spectral radius.

Similarly, our theorem follows  from a spectral gap property for the operator $f^*\colon\nums (\X)\to\nums(\X)$.
We shall prove that $\lambda_1(f)$ is an isolated point of the spectrum of $f^*$ and is a simple eigenvalue (see Theorem~\ref{thm:spectrum-nums}).

\medskip

When applied to a polynomial map $f\colon \A^d_\K \to \A^d_\K$ of the affine space, the spectral gap property gives
a valuation $v$ on the ring $\K[x_1,\cdots,x_d]$ satisfying $f_*v = \lambda_1(f)\, v$. 
This generalizes one of the main results of Jonsson and the  second author
concerning polynomial maps in dimension two~\cite{favre_jonsson_eigenvaluations,favre_jonsson_dynamical_compactification}.
Note that these works exploited the tree structure of the space of valuations centered at infinity on $\A^2_\K$, whereas our current approach avoids any information about the topology of the
space of valuations on $\A^d_\K$.

From the equation $f_*v = \lambda_1(f)\, v$ and Abhyankar's inequalities, we deduce: 
\begin{thm_int}\label{thm_poly}
Let $K$ be any field of characteristic $0$. 
For any proper polynomial map $f\colon \A^d_\K \to \A^d_\K$ such that $\lambda_1(f)^2 > \lambda_2(f)$, the dynamical degree
$\lambda_1(f)$ is an algebraic number whose degree is at most $d$ over $\Q$.
\end{thm_int}

We expect the properness assumption to be superfluous.

\smallskip

Suppose that $f$ is a polynomial automorphism of $\A^3_K$ such that $\lambda_1(f)\neq1$. 
When $\lambda_1(f)^2 > \lambda_2(f)$ then the previous theorem implies that 
$\lambda_1(f)$ is an algebraic number whose degree is at most $3$. If $\lambda_1(f)^2 = \lambda_2(f)$, then 
we have $\lambda_2(f)^2= \lambda_1(f^{-1})^2 > \lambda_2(f^{-1})$ hence 
$\lambda_2(f)$ is an algebraic number whose degree is at most $3$. It follows that either $\lambda_1(f)$ or $\lambda_1(f)^2$
is algebraic of degree $\le 3$. 
We have proved:
\begin{cor_int}
Dynamical degrees of polynomial automorphisms of $\A^3_K$ are algebraic numbers whose degree over $\Q$ is  at most $6$.
\end{cor_int}
Note that dynamical degrees of polynomial automorphisms of $\mathbb{A}^3$ are rather constrained compared to the ones of arbitrary rational transformations. 
On the one hand, the set of dynamical degrees of rational projective surface automorphisms contains a sequence of algebraic integers of unbounded degree, see e.g.~\cite[\S 7]{blanc_cantat}.
On the other hand, Bell, Diller and Jonsson recently contructed a rational map of $\Pp^2$ whose dynamical degree is transcendental~\cite{bell_diller_jonsson}.

Recall that a  real number $\lambda \ge 1$ is a (weak) Perron number if it is an algebraic integer
such that any of its Galois conjugates $\mu$ satisfies $|\mu|\le \lambda$. 
Integers, real quadratic integers $>1$, Pisot and Salem numbers are particular cases of Perron numbers.
We can now state the following intriguing conjectures by Blanc and Van Santen, see~\cite[Question~1.4.1]{blanc_santen_dynamical}.

\begin{conj} \label{conj1}
Dynamical degrees of a polynomial map of $\A^d_K$ are Perron numbers of degree at most $d$.
\end{conj}

\begin{conj} \label{conj2}
Dynamical degrees of a polynomial automorphism of $\A^d_K$ are Perron numbers of degree at most $d-1$.
\end{conj}

Theorem~\ref{thm_poly} gives a strong indication in favor of Conjecture~\ref{conj1}.

Both conjectures hold true in dimension $2$ according to the work of Friedland and Milnor~\cite{friedland_milnor} and Favre and Jonsson~\cite{favre_jonsson_dynamical_compactification}.
Computations for specific families of polynomial maps have been made in dimension $3$, including
quadratic automorphisms (Maegawa~\cite{maegawa});
cubic and triangular automorphisms (Blanc and van Santen~\cite{blanc_santen_automorphism,blanc_santen_dynamical}); and 
shift-like automorphisms (Jonsson, unpublished, see~\cite[\S 4.2]{blanc_santen_dynamical}). 
All these results support both conjectures.

---

Let us indicate some possible generalizations of our work:
\begin{itemize}
\item 
We use the resolution of singularities of algebraic varieties which explains our assumption on the characteristic of $\K$.
The space $\nums(\X)$ is constructed using intersection theory over all birational models over $X$, 
and this theory is much better behaved when the models are smooth. Note however, that 
during the course of the proof of Theorem~\ref{thm:extension}, 
we apply Gruson-Raynaud's flattening theorem which makes the appearance of singular models unavoidable.
In any case, we expect our results to be valid in any characteristic.

\item
We also use the fact that $\K$ is countable in order to have a countable set of birational models for $X$. 
This assumption appears in Lemma~\ref{lem_def_BPF}. It implies $\nums(\X)$ and several other spaces that we construct in Section~\ref{sec:numerical_b_class}
such as $\wnum^{\bullet}(\X), \numbpf{\bullet}(\X)$ and $\numbpfs{\bullet}(\X)$ to be separable (and
the space  $\FS(\X)$ introduced in Section~\ref{sec:FSX} to be Fréchet). 
Note that $X$ and $f$ are both defined over a countable field, hence this countability assumption is
not restrictive for proving Theorems~\ref{thm_int_gap} and~\ref{thm_poly}. 

\item It is likely that our method also extends to K\"ahler manifolds. This is the context of the works of Dinh, Sibony and their co-authors~\cite{dinh_sibony_2004_groupes,dinh_sibony_reg_currents_entropy,dinh_nguyen_truong,blms.12082}, and this usually requires
subtle analytic arguments. In particular, numerical classes of algebraic cycles have then to be replaced by cohomology classes, and the definition of
$\nums(\X)$ in this case remains unclear.

\item The spectrum of $f^*$ in the resonant case $\lambda_1(f)^2=\lambda_2(f)$ is harder to analyze. It is understood only in the case of birational surface maps (Diller and Favre~\cite{diller_favre}), of polynomial maps of $\A^2$ (Favre and Jonsson~\cite{favre_jonsson_dynamical_compactification}) and of toric maps (Lin~\cite{lin} and Favre and Wulcan~\cite{favre_wulcan}). Partial results on degree growths for arbitrary rational maps have been obtained by Urech~\cite{zbMATH06911556}, Cantat and Xie~\cite{cantat_xie_degrees}, and recently by Lonjou and Urech~\cite{urech-lonjou}.  

\item Extensions of Theorem~\ref{thm_poly} to other affine varieties would be interesting to explore. The first author has considered dynamical degrees of tame automorphisms
of the smooth affine quadric $3$-fold~\cite{dang_tame}, for which it is conjectured that the dynamical degrees are integers.

\item Lamy and Przytycki~\cite{lamy} have recently succeeded in putting a natural metric structure on a subspace of the space of valuations at infinity
in $\A^3_K$ turning it into a CAT$(0)$-space. They proved the linearizability of tame finite subgroups of $\Aut[\A^3_K]$ by looking at their induced action on 
this space. It would be interesting to apply our techniques to deduce further consequences on the structure of subgroups of $\Aut[\A^3_K]$ and more generally 
of $\Aut[\A^d_K]$ for any $d\ge3$. 
\end{itemize}

---

Several spaces of $b$-classes play an important role in our paper. 
\begin{itemize}
\item
$\wnum^k(\X)$: Weil $b$-classes of codimension $k$ (\S\ref{section_bclass}). This space is endowed with a weak topology in which compactness can be easily guaranteed.
This space contains most of the spaces below.
\item
$\cnum^k(\X)$: Cartier $b$-classes of codimension $k$ (\S\ref{section_bclass}). Such $b$-classes \emph{live} in a fixed birational model which makes them easy to manipulate. 
In particular, there are natural pairings $\cnum^k(\X)\times\wnum^l(\X)\to\wnum^{k+l}(\X)$. 
Although one can endow $\cnum^k(\X)$ with a natural topology, we shall make no use of such a topology in the paper.
\item
$\BPF^k(\X)$: the cone of basepoint free $b$-classes of codimension $k$ (\S\ref{sec:cone-BPF}). This is the natural positive cone to consider in $\wnum^k(\X)$. It is invariant by the operator induced by a rational map, and generalizes to  arbitrary $b$-classes and codimensions the closure of the ample cone in the Néron-Severi space of a projective variety.
\item
$\numbpf{k}(\X)$: the completion of $\cnum^k(\X)$ for the norm induced by the cone $\BPF^k(\X)$ (\S\ref{section_nbpf}). This is the first natural choice of a Banach space on which a rational map
induces a bounded operator. It is however too small in general to contain an eigenvector of this operator. 
\item
$\numbpfs{k}(\X)$: the completion of $\cnum^k(\X)$ for the dual norm induced by the cone $\BPF^k(\X)$ (\S\ref{sec:dual-norm}). The space $\numbpfs{2}(\X)$ is the natural receptacle for the intersection pairing
on $\nums(\X)$ (see Theorem~\ref{thm_banach}).
\item
$\FS(\X)$: a Fréchet space in which Cartier $b$-divisor classes are dense and for which the Hodge index theorem extends in a natural way (\S\ref{sec:FSX}). 
\item
$\nums(\X)$: a Banach space contained in $\FS(\X)$ in which Cartier $b$-divisor classes remain dense (\S\ref{section_nsigma}). This space is large enough to contain eigenvectors of operators induced by rational maps. It is also small enough so that at the same time Hodge index theorem applies, and the key compactness result Theorem~\ref{thm_compactness} holds.
\end{itemize}
We refer to Section~\ref{sec:dual-norm} below and to our companion paper~\cite{dang_favre} for a comparison between these spaces.

\subsection*{Acknowledgements}
Charles Favre started discussing the spectral interpretation of dynamical degrees back in 2014 with Mattias Jonsson. They obtained
a version of the Hodge index theorem for nef classes, and the construction of nef eigenvectors that is presented here 
(Theorems~\ref{thm:construct-eigenvec-pushpull} and~\ref{thm:construct-eigenvec-pushpull}). 
Although the technical details to complete the spectral analysis turned out to be much trickier than expected,  the influence
of Mattias' ideas to our paper was immense. We thank him for his insights and his thoughtful comments.

Nguyen-Bac Dang is very much indebted to Mikhail Lyubich for his constant financial, scientific and personal support.
He is also grateful to John Lesieutre  for interesting discussions on various aspects of 
positive classes in algebraic geometry.

We are particularly thankful to Junyi Xie and the referees for their careful readings of a first version of this work.
We have received comments from Jeffrey Diller, Romain Dujardin, Stéphane Lamy that helped us improving the paper.

Our project was launched during a symposium on "Algebraic, Complex and Arithmetic Dynamics" held in Schlo{\ss} Elmau in Germany.
We thank the Simons foundation for funding this  conference which created an awesome
inspiring atmosphere.
We finally thank the PIMS institute and the CNRS for their constant support and 
for providing us exceptional working conditions in Vancouver. 

%%%%%%%%%%%%%%%%%%%%%%%%%%%%%%%%%%%%%%%%%%%%%%%%%%%%%%%%%%%%%%%%%%%

\section{Positive cycles and dynamical degrees}

Let $\K$ be any algebraically closed field of characteristic $0$.

\subsection{Numerical cycles and their intersection} \label{section_numerical}

Let us first suppose that $X$ is a smooth projective variety of dimension $d$ defined over $\K$. 
We let $Z^k(X)$ be the $\R$-vector space freely generated by irreducible subvarieties of dimension $d-k$ in $X$. 
Given any two cycles $\alpha,\beta$ of complementary dimensions in  $X$, 
we denote by $(\alpha \cdot \beta )\in \R$ their algebraic intersection number as defined in \cite[Section 8.1 p.131]{fulton}. 
Note that when the cycles have integral coefficients, their intersection is also an integer.

The numerical space of cycles of codimension $k$, denoted $\num^k(X)$ is defined as the quotient of $Z^k(X)$ by the 
vector space of cycles $z$ such that $(\alpha \cdot z) = 0$ for all cycles $\alpha$ of dimension $k$.
It follows from~\cite[Example 19.1.3]{fulton} (see also \cite[Theorem 2.5.1]{dang_degree_normal})
that $\num^k(X)$ is a finite dimensional $\R$-vector space, and that the pairing 
$\num^k(X) \times \num^{d-k}(X) \to \mathbb{R}$ is perfect.

The space $\num^1(X)$ is the tensor product of the Néron-Severi group of $X$ with $\R$. 
Intersection products of $k$ divisors define numerical cycles of codimension $k$ but these classes do not span $\num^k(X)$ in general (see \cite[Example 2.16]{MATH06741284}).

\medskip

Let $Y$ be a smooth projective variety
of dimension $d'$.
Take any regular morphism
$f\colon X\to Y$. 
One defines the pushforward by $f$ of the cycle $[V]$ associated to  any irreducible subvariety $V$ of $X$ by
the following formula:
\begin{equation*}
f_* [V] = \left \lbrace \begin{array}{ll}
 n [f(V)] &\text{ if } \dim(f(V)) = \dim(V), \\
 0  &\text{ if } \dim(f(V)) < \dim(V),
\end{array} \right.
\end{equation*}
where $n$ is the degree of the field extension $[\K(V):\K(f(V))]$. 
This map descends to a linear map $f_*\colon \num^k(X) \to \num^k(Y)$, see e.g~\cite[Proposition~2.1.4]{dang_degree_normal}.

The pullback $f^*\colon \num^k(Y) \to \num^k(X)$ is defined by duality, and one checks 
that the projection formula 
holds (see \cite[\S 2.3 and Theorem 2.3.2]{dang_degree_normal}):
\begin{equation}\label{eq:projection-form}
 f_*(f^* \alpha \cdot \beta) = \alpha \cdot f_* \beta,
 \end{equation}
for all $\alpha \in \num^k(Y)$, and  $\beta \in \num^l(X)$ for any $k,l$.

\medskip

Let us now discuss the case where $X$ is merely normal and projective. The intersection product of two arbitrary cycles is not defined anymore. However one can intersect Chern classes of vector bundles with arbitrary cycles \cite[Chapter 3]{fulton}. We thus consider the $\R$-vector space generated by cycles obtained as a product:
\begin{equation*}
c_{i_1}(E_1) \cdot  \ldots \cdot c_{i_p}(E_p),
\end{equation*}
where $i_1, \ldots, i_p$ are integers satisfying $i_1 + \ldots + i_p = k$, $E_1 , \ldots , E_p$ are vector bundles on $X$ and $c_i(E)$ denotes the $i$-th Chern class associated to the vector bundle $E$.
The  space of numerical cycles of codimension $k$, also denoted $\num^k(X)$ is the quotient of that vector space by the subspace of classes $\alpha \in Z^k(X)$ spanned by Chern classes such that $(\alpha \cdot \beta)= 0$ for all $\beta \in Z^{d-k}(X)$ and is also a finite dimensional vector space. 
In the case where $X$ is smooth, the two definitions of $\num^k(X)$ coincide (see \cite[\S~2.1]{MATH06741284} or \cite[Theorem 2.4.2 (i)]{dang_degree_normal}).    
\medskip

Now if $X,Y$ are normal projective varieties of dimension $d,d'$ respectively, and $f:X \to Y$ is a regular morphism, then the pullback of vector bundles by $f$ is well-defined and induces a morphism $f^* : \num^k(Y) \to \num^k(X)$. In this degree of generality, it is not known whether the pushforward on cycles descends to the span of Chern classes modulo numerical equivalence,  
the reason being that the pullback is not defined for arbitrary cycles.

If we suppose $Y$ smooth, one uses the perfect pairing $\num^k(Y) \times \num^{d'-k}(Y) \to \R$ on $Y$ and the existence of  the pullback $f^*\colon\num^k(Y) \to \num^k(X)$  to define the pushforward. Precisely, for any intersection $c \in \num^{d-k}(X)$ of Chern classes of vector bundles, the class $f_* c \in \num^{d'-k}(Y)$ is the unique class such that for any cycle $\alpha \in \num^{k}(Y)$: 
\begin{equation*}
  (f_*c \cdot \alpha) = (c\cdot f^*\alpha)~.
  \end{equation*}  
We then obtain a morphism $f_* \colon \num^{d-k}(X) \to \num^{d'-k}(Y)$ compatible with the pushforward $f_* \colon Z^{d-k}(X) \to Z^{d'-k}(Y)$ on cycles representing Chern classes of vector bundles.

When the map $f: X \to Y$ is flat of relative dimension $e$ and $X,Y$ are normal projective varieties, then the pullback $f^*$ is well-defined on $Z^{d'-k}(Y)$ and for any intersection $\beta$ of $e+k$ Cartier divisors, we define the element $f_* \beta \in \num^k(Y)$ as the dual class induced by the linear form:
\begin{equation} \label{eq_def_pushforward_flat}
f_* \beta := z \in Z^{d-k}(Y) \mapsto (\beta\cdot f^* z).
\end{equation}

\subsection{Positive cycles}

We assume in this section  that $X$ is a smooth projective variety.
A class in $\num^k(X)$  is called pseudo-effective, if it belongs to the closure of the convex cone spanned by effective cycles. 
The set of pseudo-effective classes forms a closed salient convex cone $\psef^k(X)$ inside $\num^k(X)$, see~\cite[Theorem 1.4]{MATH06741284}.
For any  $\alpha, \beta \in \num^k(X)$, we write
\[
\alpha \le \beta
\]
whenever the difference $\beta - \alpha $ is pseudo-effective.

A class $\alpha \in \num^1(X)$ is nef if its intersection with any pseudo-effective curve class is non-negative.
The nef cone $\nef^1(X)$ is the closure of the cone of real ample classes, see~\cite[Theorem 1.4.23]{lazarsfeld_positivity_1}. A class $\alpha\in \num^k(X)$ with 
$k\ge2$ is nef when $(\alpha\cdot \beta)\ge0$ for all $\beta\in \psef^{d-k}(X)$.
The nef cone in codimension $k\ge2$ is not included in the pseudo-effective cone in general, see~\cite{debarre_ein_lazarsfeld_voisin}; however the cone of 
basepoint free classes defined below is better behaved, see~\cite{MATH06741284} for a general discussion on this problem.

 A class $\alpha\in\num^k(X)$ is called \textit{strongly basepoint free} 
 if it is the pushforward under a flat morphism of relative dimension $e$ 
 of the intersection of $k+e$ ample divisors. 
The closure of the cone generated by strongly basepoint free classes is called the basepoint free cone of codimension $k$. We denote it by 
$\BPF^k(X)$.

Note that our definition of basepoint free classes differs from  Fulger-Lehmann's original formulation \cite[Definition 5.1]{MATH06741284} but the two approaches are equivalent when $X$ is smooth (see \cite[Corollary 3.3.4]{dang_degree_normal}).
\bigskip

 The main properties of this cone are summarized in the following result. 

\begin{thm}[\cite{MATH06741284}] \label{thm_bpf} Let $X, Y$ be smooth projective varieties. 
\begin{enumerate}
\item The set $\BPF^k(X)$ forms a closed, convex, salient  cone in $\num^k(X)$ and has non-empty interior.
\item The basepoint free cones respect the intersection product, i.e. $\BPF^k(X) \cdot \BPF^l(X) \subset \BPF^{k+l}(X)$ for all $k,l$.
\item Classes in $\BPF^k(X)$ are both pseudo-effective and nef.
\item For any integer $k$ and for any morphism $\pi\colon Y \to X$, 
the inclusion $\pi^* \BPF^k(X) \subset \BPF^k(Y)$ is satisfied. 
\item For any flat morphism $\pi\colon Y \to X$ of relative dimension $e$, one has $\pi_* (\BPF^k(Y)) \subset \BPF^{k-e}(X)$.
\item 
Strongly BPF classes of $\num^k(X)$ are contained in the interior of  $\BPF^k(X)$.
\item The cone $\BPF^1(X)=\nef^1(X)$ coincides with the nef cone.
\end{enumerate}
\end{thm}

Recall that a class $\alpha\in\num^1(X)$ is big if it lies in the interior of the pseudo-effective cone. 
A nef class  $\alpha\in\num^1(X)$ is big iff $(\alpha^d)>0$.
The next result is a consequence of Siu's inequalities, see \cite[Theorem 2.2.15]{lazarsfeld_positivity_1}, \cite{trapani}
which allows us to compare basepoint free classes with complete intersections.
It was proven in a K\"ahler context by~\cite{xiao}, and later generalized to arbitrary algebraic varieties by the first author in~\cite[Corollary~3.4.5]{dang_degree_normal}.

\begin{thm}[Siu's inequalities]\label{thm_siu_bpf} 
There exists a constant $C_d>0$ such that the following holds.

Let $\om\in\num^1(X)$ be any big and nef class over a smooth projective variety $X$ of dimension $d$. 
Then for any basepoint free class $\alpha \in \num^k(X)$, we have
 \begin{equation}\label{eq:siu_bpf}
0\le  \alpha \le C_d\, \dfrac{(\alpha \cdot \om^{d-k})}{(\om^d)} \om^k.
 \end{equation}
\end{thm}
It is a fact that when $k=1$, we may take $C_d=d$.

\subsection{Dynamical degrees of dominant rational maps}\label{sec:dynamical_degrees}
Fix any big and nef class $\om\in\num^1(X)$ where $X$ is a normal projective variety. 
For any dominant rational self-map $f\colon X \dashrightarrow X$, and for any integer $0 \le k \le d$, we set
\begin{equation*}
\deg_{k,\om}(f) = (\pi_2^* \om^k\cdot \pi_1^* \om^{d-k}),
\end{equation*} 
where $\pi_1, \pi_2$ are the projections from a resolution of singularities \cite{kollar} of the graph of $f$ in $X\times X$ onto the first and second component respectively \footnote{observe that the above notion of degree coincides with the one given in the introduction using the projection formula and setting $f^* = {\pi_1}_*\circ\pi_2^*$ (see also \cite[Definition 5.4.1]{dang_degree_normal})}.
This quantity does not depend on the choice of resolution (see e.g \cite[Theorem 1.(ii)]{dang_degree_normal}), and we call it the $k$-th degree of $f$ with respect to $\om$.
\footnote{when $\om = c_1(L)$ for some line bundle $L\to X$, then $\deg_{k,\om}(f)$ is an integer.}

The basic results concerning the sequence of degrees $\{\deg_{k,\om}(f^n)\}_n$ were first obtained by Dinh and Sibony~\cite{dinh_sibony_une_borne_sup} in the K\"ahler
case, and were further extended to algebraic varieties over any field by T. T. Truong~\cite{tuyen2} and the first author~\cite{dang_degree_normal}.
\begin{thm}\label{thm:sbmlt}
\begin{itemize}
\item
There exists a positive constant $C= C(\om) >0$ such that or any dominant self-map
$f\colon X \dashrightarrow X$ and for any integers $n,m$, we have:
\[
\deg_{k,\om}(f^{n+m}) \le C \deg_{k,\om}(f^n) \deg_{k,\om}(f^m).
\]
\item
For any big and nef class $\omega'$, there exists a constant $C= C(\omega,\omega')>1$ such that for any dominant self-map
$f\colon X \dashrightarrow X$, we have 
\[
\frac1C\, 
\deg_{k,\om}(f)
\le 
\deg_{k,\omega'}(f)
\le C\, \deg_{k,\om}(f).
\]
\end{itemize}
\end{thm}
Since the sequence $(C\deg_{k,\om}(f^n))_n$ is sub-multiplicative, the sequence $(C\deg_{k,\om}(f^n))^{1/n}$ converges to
its infimum so that we can set:
\begin{equation*}
\lambda_k(f):=\lim_{n\rightarrow +\infty} (\deg_{k,\om}(f^n))^{1/n}~.
\end{equation*} 
The quantity $\lambda_k(f)$ belongs to $[1,+\infty)$, and does not depend on the choice of big and nef class $\omega$.
It also follows from Khovanskii-Teissier inequalities, see e.g.~\cite[Corollary 1.6.3]{lazarsfeld_positivity_1}, that the sequence $k \mapsto \lambda_k(f)$ is log-concave. 
In particular, we always have $\lambda_1(f)^2 \ge  \lambda_2(f)$.

%%%%%%%%%%%%%%%%%%%%%%%%%%%%%%%%%%%%%%%%%%%%%%%%%%%%%%%%%%%%%%%%%%%%

\section{Numerical $b$-classes}\label{sec:numerical_b_class} 

Let $X$ be a normal projective variety of dimension $d$ defined over a field $\K$. From now on, we further assume $\K$ to be countable, algebraically closed
and of characteristic $0$.

A model (resp. a smooth model) over $X$ is a projective birational morphism $\pi \colon X' \to X$  from a normal (resp. smooth) projective variety $X'$. 
Let $\cM_X$ be the category of all smooth models over $X$. Since any projective birational morphism is isomorphic over $X$ to the blow-up of some ideal sheaf on $X$ (see e.g.~\cite[Theorem~7.17]{hartshorne}), 
$\cM_X$ is essentially small. We may thus make a slight abuse of language and think of 
$\cM_X$ as a set (which is countable since $K$ is). 
 It is also a poset for the domination relation defined by $X'\ge X''$ if and only if the map $(\pi'')^{-1}\circ \pi'\colon X'\to X''$ is regular.
Given any two elements $X_1, X_2\in \cM_X$ one can find
$X_3 \in \cM_X$ such that $X_3\ge X_1$ and $X_3\ge X_2$.
Since we work over a field of characteristic $0$, any model $X'$ is dominated by a smooth one.

Although this is not crucial for our discussion, it is convenient to introduce the projective limit over the inductive system $\cM_X$
of all models over $X$, each endowed with the Zariski topology. One obtains in this way a quasi-compact topological space that we denote by $\X$, 
which is the Riemann-Zariski space attached to $X$. We refer to~\cite{vaquie} for a description of this space in terms of valuations on the function field $\K(X)$.

\subsection{Weil and Cartier  $b$-classes} \label{section_bclass}

In this section, we explain the construction of Weil and Cartier $b$-classes. 
Note that the intersection theory is more subtle when the ambient space is not smooth. We use here in a crucial way the existence of desingularization in characteristic zero to define the intersection of $b$-classes.

A Weil $b$-class $\alpha$ of codimension $k$ is a map  which assigns to any smooth model $X'\in \cM_X$ a numerical class $\alpha_{X'}\in \num^k(X')$,
such that $\pi_*(\alpha_{X'})=\alpha_{X''}$ for any pair of \textit{smooth} models $X'\ge X''$ with $\pi= (\pi'')^{-1}\circ \pi'$.
The class $\alpha_{X'}$ is called the incarnation of $\alpha$ in $X'$. 

A Cartier $b$-class is a Weil  $b$-class $\alpha$ for which one can find a model $X''$
and a class $\beta\in \num^k(X'')$
such that $\alpha_{X'} = \pi^*(\beta)$ for any smooth model $X'\ge X''$. When it is the case, we say that 
$\alpha$ is determined in $X''$. A \emph{determination} of a Cartier class $\alpha$ is a pair consisting of a model $X''$ and a class
$\beta \in \num^k(X'')$ such that $\alpha$ is determined in $X''$ and $\alpha_{X''} =\beta$. 

We denote by $[\beta]$ the Cartier $b$-class associated to a class $\beta\in\num^k(X'')$.
We also used the notation $[V]$ for the fundamental class of a subvariety $V\subset X$, but there will be
no chance of confusion in the sequel.

\smallskip

When $d\ge 2$, the space of Weil  $b$-classes is an infinite dimensional real vector space, which we denote by $\wnum^k(\X)$.
It contains the set of Cartier  $b$-classes $\cnum^k(\X)$ as a subspace, and for each smooth model $X'$ the map
$\alpha \mapsto [\alpha]$ induces an injective linear map $\num^k(X')\to\cnum^k(\X)$.

\smallskip

The space  $\wnum^k(\X)$ can be seen as the projective limit of the spaces $\{\num^k(X')\}_{X'\in \cM_X}$ where
morphisms are given by pushforwards on smooth models. It is thus endowed with a natural product (or weak) topology. 
Since $\K$ is countable, $\wnum^k(\X)$ is a projective limit over a countable inductive set of a family of finite dimensional vector spaces. It can thus 
be identified to a closed subset of a product of countably many metrizable and separable topological spaces.
Hence $\wnum^k(\X)$ is  a metrizable, separable topological vector space.
A sequence $\alpha_n \in \wnum^k(\X)$ converges to $\beta \in \wnum^k(\X)$ if and only if
for any smooth model $X'$ we have $(\alpha_n)_{X'} \to \beta_{X'}$ in $\num^k(X')$.

By a diagonal extraction argument, any sequence $\alpha_n \in \wnum^k(\X)$ such that $\alpha_{n,X'}$ is bounded in $\num^k(X')$ for all $X'$ admits
a convergent subsequence.

Since for any $\alpha\in\wnum^k(\X)$ the sequence of Cartier $b$-classes $[\alpha_{X'}]$ converges to $\alpha$,
the space $\cnum^k(\X)$ is dense in $\wnum^k(\X)$.

\begin{rem}
In a similar way, $\cnum^k(\X)$ can be identified with the injective limit of $\{\num^k(X')\}_{\cM_X}$. 
Although we will make no use of it, it may be endowed with a natural topology, see e.g.~\cite[\S I.1, p.27]{yosida}. 
\end{rem}

Pick $\alpha\in \cnum^k(\X)$ and $\beta\in \wnum^l(\X)$. 
Suppose $\alpha$ is determined in a model $X_0$. For any smooth model $X'\ge X_0$, we set 
$(\alpha \cdot \beta)_{X'} := (\alpha_{X'} \cdot \beta_{X'}) \in \num^{k+l}(X')$. 
By the projection formula~\eqref{eq:projection-form}, $\pi_* (\alpha \cdot \beta)_{X''} = (\alpha \cdot \beta)_{X'}$ for any smooth model $X''\ge X'$
so that we may define the class $\alpha \cdot \beta \in \wnum^{k+l}(\X)$ as the unique Weil  $b$-class
whose incarnation in any smooth model dominating $X_0$ is equal to $(\alpha \cdot \beta)_{X'}$. Beware that if the canonical birational map
$\pi\colon X_0\to X'$ is regular so that $X_0\ge X'$, then the incarnation of $(\alpha \cdot \beta)_{X'}$ is equal to $\pi_*(\alpha \cdot \beta)_{X'}$ which is different from $(\alpha_{X'} \cdot \beta_{X'})$  in general.
Observe also that the intersection of two Cartier $b$-classes remains Cartier.

If we let $\wnum^\bullet(\X)= \oplus_k \wnum^k(\X)$, and $\cnum^\bullet(\X)= \oplus_k \cnum^k(\X)$, the intersection product defined above induces
a bilinear pairing \[\cnum^\bullet(\X)\times\wnum^\bullet(\X) \to \wnum^\bullet(\X) \]
respecting the natural grading, and continuous for the weak topology in the second variable. 

Note that $\cnum^d(\X) = \wnum^d(\X)$ is canonically isomorphic to $\R$ so that
Poincar\'e duality~\cite[Theorem 2.4.2]{dang_degree_normal} on each smooth model implies that the induced pairing
\[\cnum^k(\X) \times \wnum^{d-k}(\X) \to \wnum^{d}(\X)= \R\] 
is perfect.

\subsection{Positive cones of  $b$-classes}\label{sec:cone-BPF}

A class $\alpha\in\wnum^k(\X)$ is said to be pseudo-effective (and we write $\alpha\ge0$) when $\alpha_{X'}\ge0$ for all smooth models $X'$.
Note that the pseudo-effectivity is only preserved by push-forward, and not by pull-back\footnote{Blow-up a smooth rational curve $C$
in a smooth threefold $X$ with normal bundle $N_{C/X} = \cO(-1) \oplus  \cO(-1)$, see~\cite[6.20]{debarre_higher} on this specific example or~\cite[Theorem~6.7]{fulton} for a general blow-up formula.}, so that it may happen that a Cartier $b$-class
determined by a pseudo-effective class in a smooth model $X$ is not pseudo-effective in $\wnum^k(\X)$. 
In codimension $1$, a Cartier  $b$-class $\alpha \in \cnum^1(\X)$ is pseudo-effective if and only if one (or any) or its determination on a smooth model is pseudo-effective.

\smallskip

The notion of base-point free $b$-classes is more subtle to define. We build on the construction of nef $b$-classes from~\cite{MR2426355,MR3010168}.
We let $\cBPF^k(\X)$ be the convex cone in $\cnum^k(\X)$ generated by Cartier $b$-classes
$[\alpha]$ with $\alpha\in \BPF^k(X')$ for some \textit{smooth} model $X'$.
Observe that $\Vect(\cBPF^k(\X)) = \cnum^k(\X)$.

Since BPF classes are preserved by pull-backs, a class lies in $\cBPF^k(\X)$ if and only if 
 one (or any) of its determinations is BPF.

\begin{defi} 
The cone $\BPF^k(\X)$ is the (weak) closure in $\wnum^k(\X)$ of the cone $\cBPF^k(\X)$.
\end{defi}
When $k=1$, we write $\nef^1(\X)=\BPF^1(\X)$, and $\cNef(\X)=\cBPF^1(\X)$.
Note that any class in  $\BPF^k(\X)$ is pseudo-effective.
\smallskip

In other words, a class $\alpha\in \wnum^k(\X)$ is BPF if and only if for any model $X'$ one can find a sequence $\alpha_n\in \cBPF^k(\X)$
such that $(\alpha_n)_{X'} \to \alpha_{X'}$ in $\num^k(X')$. In general $\alpha_n$ depends on the model $X'$ and is not determined in $X'$.
A diagonal extraction argument yields:
\begin{lem}\label{lem_def_BPF}
 For any element $\alpha\in\BPF^k(\X)$, there exists a sequence of classes $\alpha_n \in \cBPF^k(\X)$ such that 
 $(\alpha_n)_{X'} \to \alpha_{X'}$  in $\num^k(X')$ for all model $X'$.
\end{lem}

We mention the following
\begin{conj}
For all $k$, we have 
\[\cBPF^k(\X)= \cnum^k(\X) \cap \BPF^k(\X) .\] 
\end{conj}

The inclusion $\cBPF^k(\X)\subset \cnum^k(\X) \cap \BPF^k(\X)$ follows from the definition, but the other inclusion is quite deep. 
We shall prove this conjecture when $k=1$ in~\cite{dang_favre}.
We also refer to~\cite[Theorem~5.11]{boucksom_fernex_favre_urbinati} for a result in codimension $1$ of a similar flavor.

\subsection{The norm $\normbpf{\cdot}$}

Pick any big and nef class $\omega\in\num^1(X)$, that we normalize so that $(\om^d)=1$. 
In the sequel, we abuse notation and still denote by $\omega$ the unique Cartier $b$-class determined in $X$
by $\omega$.

For any class $\alpha\in \Vect(\BPF^k(\X))$, we set 
\[
\normbpf{\alpha } = \inf_{ \substack{\alpha = \alpha_+ - \alpha_- \\
\alpha_+,\alpha_- \in {\BPF^k(\X)}}} (\alpha_+ \cdot \omega^{d-k}) + (\alpha_- \cdot \omega^{d-k}). 
\]
Since $(\alpha \cdot \omega^{d-k})\ge 0$ for any BPF class, it follows that $\normbpf{\alpha }\in \R_+$.
We begin with
\begin{lem}\label{lem:pleasethereferee}
For any class $\alpha\in\BPF^k(\X)$, we have 
\begin{equation}\label{eq:5436}
\normbpf{\alpha} = (\alpha \cdot \omega^{d-k})
\text{ and  } 0 \le \alpha \le C_d \normbpf{\alpha} \omega^k
\end{equation}
thus 
 $\normbpf{\alpha}=0$ if and only if $\alpha=0$.
\end{lem}
\begin{proof}
Since $\alpha = \alpha - 0$, we immediately get
$\normbpf{\alpha} \le (\alpha \cdot \omega^{d-k})$.
On the other hand, for any $\alpha_+, \alpha_- \in \BPF^k(\X)$ such that $\alpha = \alpha_+ - \alpha_-$,
we have
\[
(\alpha \cdot \omega^{d-k}) \le (\alpha_+ \cdot \omega^{d-k}) + (\alpha_- \cdot \omega^{d-k})
\]
hence $(\alpha \cdot \omega^{d-k}) \le \normbpf{\alpha}$ proving $\normbpf{\alpha} = (\alpha \cdot \omega^{d-k})$.

Pick a resolution of singularities $X'$ of $X$. Let $\alpha_n$ be any sequence of classes in $\cBPF^k(\X)$ such that
$(\alpha_n)_{X'} \to \alpha_{X'}$. We infer from Siu's inequalities~\eqref{eq:siu_bpf}:
\[
0 \le \alpha_n \le C_d (\alpha_n \cdot \omega^{d-k}) \omega^k = C_d \normbpf{\alpha_n} \omega^k~.
\]
Letting $n\to\infty$, we obtain $0 \le \alpha \le C_d \normbpf{\alpha} \omega^k$
which implies the lemma.
\end{proof}

\begin{prop} \label{prop_intermediate_banach}
There exists a constant $C>0$ depending only on $d$ such that
for all $\alpha\in \Vect(\BPF^k(\X))$, we have 
\begin{equation}\label{eq:siu-vecbpf}
- C \normbpf{\alpha} \omega^k 
\le
\alpha
\le 
C \normbpf{\alpha} \omega^k
\end{equation}
so that the function $\alpha \mapsto\normbpf{\alpha}$ defines a norm on the vector subspace of $\wnum^k(\X)$ spanned by $\BPF^k(\X)$.

Moreover, the normed space $(\Vect(\BPF^k(\X),\normbpf{\cdot})$
is complete. 
\end{prop}

\begin{rem}
By Siu's inequalities~\eqref{eq:siu_bpf}, replacing $\omega$ by another big and nef class changes the norm $\normbpf{\cdot}$ by an equivalent one. 
 \end{rem}

\begin{proof} Pick any class $\alpha \in \Vect(\BPF^k(\X))$.
Then for any $\epsilon>0$, we may find
two classes $\alpha_\pm\in\BPF^k(\X)$ such that $\alpha=\alpha_+-\alpha_-$ and $\normbpf{\alpha_+}+\normbpf{\alpha_-}\le  \normbpf{\alpha}+\epsilon$
so that~\eqref{eq:siu-vecbpf} follows from~\eqref{eq:5436} by letting $\epsilon\to0$.

\smallskip

It remains to prove that $(\Vect(\BPF^k(\X)),\normbpf{\cdot})$ is complete.  Take a Cauchy sequence $\alpha_n\in\Vect(\BPF^k(\X))$.
  
We first claim that the sequence $\alpha_n$ converges (weakly) in $\wnum^{k}(\X)$. By duality and by Theorem~\ref{thm_bpf} (1), this is equivalent to proving that 
 the sequence $(\alpha_n \cdot [\gamma])$ has a limit for any basepoint free class $\gamma\in\BPF^{d-k}(X')$ for any model $X'$.
We shall prove that the sequence of real numbers $(\alpha_n \cdot [\gamma])$  in fact forms a Cauchy sequence. 
Indeed, for all $n,m$ large enough, the difference $\alpha_n - \alpha_m$ can be decomposed as:
\begin{equation*}
\alpha_n- \alpha_m = \beta_{nm} - \delta_{nm}
\end{equation*}
with $\beta_{nm}, \delta_{nm} \in {\BPF^k(\X)}$ and $(\beta_{nm} \cdot \omega^{d-k}) + (\delta_{nm} \cdot \omega^{d-k}) < \epsilon$. 
Siu's inequalities~\eqref{eq:siu_bpf} yield $\gamma \le C_d \omega^{d-k}$, and 
we have
\begin{equation} \label{eq_cauchy}
|(\alpha_n \cdot \gamma) - (\alpha_m \cdot \gamma)| = |(\beta_{nm} \cdot \gamma) - (\delta_{nm} \cdot \gamma)| \le C_d (\beta_{nm} \cdot \omega^{d-k}) + C_d (\delta_{nm}\cdot \omega^{d-k}) \le C_d \epsilon  
\end{equation}
which proves our claim.
\smallskip

Denote by $\alpha$ the limit in $\wnum^{k}(\X)$ of the sequence $\alpha_n$. 
By taking an appropriate subsequence of $\alpha_n$, we may suppose that $\normbpf{\alpha_{n+1} - \alpha_n} \le 1/n^2$. 
We may thus decompose $\alpha_{n+1}- \alpha_n$ as:
\[
\alpha_{n+1} - \alpha_n = \beta_n - \delta_n,
\]
with $\beta_n, \delta_n \in \BPF^{k}(\X)$ and $ (\beta_n \cdot \omega^{d-k}) + (\delta_n \cdot \omega^{d-k}) \le 2/n^2$.
Define $\alpha_+ = \sum_{m=1}^{+\infty} \beta_m$, $\alpha_- = \sum_{m=1}^{+\infty} \delta_m$, $R_n =\sum_{m\ge n} \beta_m$, and 
$S_n =\sum_{m\ge n} \delta_m$. Observe that these four classes belong to $\BPF^k(\X)$.
By construction:
\begin{equation*}
\alpha -\alpha_1 = \alpha_+ - \alpha_-
\end{equation*}
hence $\alpha\in\Vect(\BPF^k(\X))$, and the difference $\alpha - \alpha_n$ can be expressed as:
\begin{equation*}
\alpha - \alpha_n = R_{n} - S_{n},
\end{equation*}
with $(R_n \cdot \omega^{d-k}) + (S_n \cdot \omega^{d-k}) \le 4\sum_{m\ge n} m^{-2} \to 0$. 
This concludes our proof.
\end{proof}

\begin{cor} \label{cor_intermediate_banach} 
The vector space ${\Vect(\BPF^\bullet(\X))}$, endowed with the norm $\normbpf{\cdot} $ is a graded Banach space for the natural grading by codimension.
\end{cor}

\begin{rem} \label{rem:not-dense}
We expect $\cnum^k(\X)$ \emph{not} to be dense in ${\Vect(\BPF^k(\X))}$ for the norm $\normbpf{\cdot}$ (at least when $0<k<\dim(X)$).
This phenomenon prevents the intersection form from extending continuously to $(\Vect(\BPF^\bullet(\X)), \normbpf{\cdot})$, thereby forcing us to consider
smaller Banach spaces that are better behaved.
\end{rem}

\subsection{The space $\numbpf{k}(\X)$} \label{section_nbpf}
For any $0\le k \le d$, we let  $\numbpf{k}(\X)$ be the closure of $\cnum^k(\X)$ in the Banach space $(\Vect(\BPF^k(\X)),\normbpf{\cdot})$.
Equivalently  $\numbpf{k}(\X)$ is the completion for the norm $\normbpf{\cdot}$ of the space of Cartier $b$-classes  $\cnum^k(\X)$.
We write 
\[\numbpf{\bullet}(\X) = \oplus_k \numbpf{k}(\X).\]

\begin{thm} \label{thm_banach_algebra} The following properties hold.
\begin{enumerate}
\item $(\numbpf{\bullet}(\X), \normbpf{\cdot})$ is a graded Banach space for the natural grading containing $\cnum^\bullet(\X)$ as a dense subset.
\item The intersection product on $\cnum^\bullet(\X)$ extends continuously to $\numbpf{\bullet}(\X)$ and endows the latter with a structure of a graded normed Banach algebra with unit $[X]$. 
\end{enumerate}
\end{thm}

\begin{proof}
The first statement is a direct consequence of the definition. Let us explain why the intersection product extends to $\numbpf{\bullet}(\X)$. 
Take $\alpha \in \numbpf{k}(\X)$ and $\beta \in \numbpf{l}(\X)$ and consider two sequences $\alpha_n ,\beta_n \in \cnum^{\bullet}(\X)$ converging in norm to $\alpha$ and $\beta$ respectively.  We claim that $(\alpha_n \cdot \beta_n)$ forms a Cauchy sequence in $\numbpf{k+l}(\X)$.
Grant this claim. Then we may define $(\alpha\cdot \beta)$ as the limit of  $(\alpha_n \cdot \beta_n)$ as $n\to\infty$ and this limit does not depend on the choice
of approximating sequences $\alpha_n$ and $\beta_n$.

To prove the claim, we fix any $\epsilon >0$, and write 
\begin{equation*}
\alpha_n - \alpha_m = u_{nm}^+ - u_{nm}^-,\text{ and } \beta_n - \beta_m = v_{nm}^+ - v_{nm}^-
\end{equation*} 
with $\normbpf{u_{nm}^+} + \normbpf{ u_{nm}^-}+ \normbpf{v_{nm}^+}+ \normbpf{ v_{nm}^-} \le \epsilon$.
We then have
\begin{align*} \label{eq_difference}
(\alpha_n\cdot \beta_n) - (\alpha_m\cdot \beta_m)
&= \left( (\alpha_n - \alpha_m) \cdot \beta_n\right) + \left(\alpha_m \cdot (\beta _n - \beta_m)\right).
\end{align*}
By Siu's inequalities~\eqref{eq:siu-vecbpf}, we get
\[
\normbpf{\left( (\alpha_n - \alpha_m) \cdot \beta_n\right)}
\le 
C_d\,
\left(\normbpf{u_{nm}^+} + \normbpf{ u_{nm}^-}\right) \sup_n  \normbpf{\beta_n}
\]
so that 
\[\normbpf{(\alpha_n\cdot \beta_n) - (\alpha_m\cdot \beta_m)}
\le 
C_d\,
\epsilon\, \left( \sup_n  \normbpf{\beta_n} + \sup_m  \normbpf{\alpha_m} \right)
\le C' \epsilon~,\]
completing the proof of our claim. 

The proof that the product is continuous follows from the same estimates.
\end{proof}

\subsection{The dual norm $\normbpfs{\cdot}$}\label{sec:dual-norm}
Recall that for any BPF class $\gamma\neq0$ of codimension $k$, we have $(\omega^k \cdot \gamma)\neq0$ by~\eqref{eq:5436}.
For any Cartier $b$-class $\alpha \in \cnum^k(\X)$ we set:
\begin{equation}
\normbpfs{\alpha} := \sup_{\substack{\gamma \in \cBPF^{d-k}(\X) \\
\gamma\neq 0 }} \frac{|(\alpha \cdot \gamma)|}{(\omega^k \cdot \gamma)}.
\end{equation} 
Since for any smooth model $X'$ the pairing $\num^k (X')\times \num^{d-k}(X')\to \R$ is perfect, 
we have $\normbpfs{\alpha}=0$ if and only if $\alpha=0$. Observe also that $\normbpfs{\alpha}= \sup_{\normbpf{\gamma}=1} |(\alpha \cdot \gamma)|$
justifying its appellation of dual norm.

\smallskip

We let $\numbpfs{k}(\X)$ be the completion of $\cnum^k(\X)$ with respect to  $\normbpfs{\cdot}$. The topology induced by this norm being finer than the weak topology, 
we have a continuous injection  $\numbpfs{k}(\X) \subset \wnum^k(\X)$.

\begin{prop}\label{prop:bpf_vs_bpf*}
There exists a constant $C>0$ depending only on $d$ such that 
for any class $\alpha\in\cnum^k(\X)$, we have $\normbpfs{\alpha}\le C \normbpf{\alpha}$. In particular, the inclusion
$\numbpf{k}(\X)\subset \numbpfs{k}(\X)$ is a continuous injection.
\end{prop}
\begin{proof}
Write $\alpha = \alpha_+- \alpha_-$ with $\alpha_\pm\in\BPF(\X)$
so that  
\[\max \left\{ \normbpf{\alpha_+} , \normbpf{\alpha_-}\right\}\le 2 \normbpf{\alpha}.\] 
By Siu's inequalities~\eqref{eq:siu_bpf}, 
we get $\alpha_\pm \le C_d \normbpf{\alpha} \omega^k$.
For any $\gamma \in \cBPF^{d-k}(\X)$ such that $(\omega^k \cdot \gamma) =1$, by intersecting the previous relation with $\gamma$, we conclude  that
$|(\alpha\cdot \gamma)|\le 2C_d \normbpf{\alpha}$,  hence $\normbpfs{\alpha}\le 2C_d \normbpf{\alpha}$. 
\end{proof}

Let us now discuss the relationship of $\numbpfs{k}(\X)$  with the continuous duals $\numbpf{d-k}(\X)^*$ and $\Vect(\BPF^{d-k}(\X))^*$.

\begin{prop} \label{prop_inclusion_continuous_duals}
The continuous dual $\numbpf{d-k}(\X)^*$ can be canonically identified with the space of classes $\alpha\in \wnum^k(\X)$ for which there exists $C_\alpha>0$
such that
\[|(\alpha\cdot \gamma)| \le C_\alpha \normbpf{\gamma} \text{ for all } \gamma\in\cBPF^{d-k}(\X).\]
In particular, we have the following continuous injections
\[\numbpf{k}(\X)\subset \numbpfs{k}(\X)\subset \numbpf{d-k}(\X)^* \subset \wnum^k(\X)
, \text{ and }\]
\[\Vect(\BPF^{k}(\X))\subset \numbpfs{d-k}(\X)^* \subset \numbpf{d-k}(\X)^*  \subset \wnum^k(\X)
.\]
  \end{prop}
  
 \begin{conj}
 We have the inclusion $\Vect(\BPF^{k}(\X))\subset \numbpfs{k}(\X)$, and the isomorphism
  $(\Vect(\BPF^{d-k}(\X)),\normbpf{\cdot}) \simeq \numbpfs{k}(\X)^*$.
 \end{conj}
  
 We shall prove the inclusion  $\Vect(\nef^1(\X))\subset \numbpfs{1}(\X)$  in~\cite{dang_favre}. We will also discuss the isomorphim
 $(\Vect(\BPF^{d-1}(\X)),\normbpf{\cdot}) \simeq \numbpfs{1}(\X)^*$ for torus invariant classes.

\begin{proof}
Let $\ell$ be any continuous linear form on $ \numbpf{d-k}(\X)$. Since $ \numbpf{d-k}(\X) \supset \cnum^{d-k}(\X)$ and the pairing 
$\wnum^k(\X) \times  \cnum^{d-k}(\X)\to \R$ is perfect, there exists a (unique) class $\alpha_\ell\in \wnum^k(\X)$ such that
$\ell(\gamma)= (\alpha_\ell\cdot \gamma)$ for any $\gamma\in\cnum^{d-k}(\X)$. By density of $\cnum^{d-k}(\X)$, the map
$\ell \mapsto \alpha_\ell$ is injective.

As $\ell$ is continuous, there exists $C>0$ such that for any $\gamma \in \cBPF^{d-k}(\X)$ we have  $|(\alpha_\ell\cdot \gamma)|\le C \normbpf{\gamma}$, and
the dual norm of $\ell$ is equal to 
\begin{align}\label{eq:dual-norm}
\| \ell\|_* 
&=
 \sup_{\substack{\normbpf{\gamma}=1\\
\gamma \in \cBPF^{d-k}(\X)}} |\ell(\gamma)| 
=
\sup_{\substack{\normbpf{\gamma}=1\\
\gamma \in \cBPF^{d-k}(\X)}} |(\alpha_\ell\cdot \gamma)|.
\end{align}
Conversely any class $\alpha\in\wnum^k(\X)$ such that $|(\alpha\cdot \gamma)|\le C \normbpf{\gamma}$ for all $\gamma \in\cBPF^{d-k}(\X)$
defines a continuous linear form on $ \numbpf{d-k}(\X)$ since $\cnum^{d-k}(\X)$ is dense in $ \numbpf{d-k}(\X)$ by definition. 
This proves the first assertion of the proposition.

If $\alpha \in \cBPF^k(\X)$ and  $\gamma \in \cnum^{d-k}(\X)$, then we have
$|(\alpha\cdot \gamma)| \le (\alpha\cdot \om^{d-k}) \normbpfs{\gamma}$ by definition, 
and this estimate remains valid if $\alpha \in \BPF^k(\X)$ by density.
This shows that the linear map $\gamma \mapsto (\alpha\cdot \gamma)$ extends as a continuous
form to $\numbpfs{d-k}(\X)$, therefore $ \Vect(\BPF^k(\X))\subset \numbpfs{d-k}(\X)^*$. The inclusion $\numbpfs{d-k}(\X)^* \subset \numbpf{d-k}(\X)^*$ follows
from the continuous injection $\numbpf{d-k}(\X)\subset \numbpfs{d-k}(\X)$.

Since $\cnum^k(\X)\subset \Vect(\BPF^k(\X))$ we get  $\cnum^k(\X)\subset \numbpf{d-k}(\X)^*$. The restriction of the dual norm
on $\cnum^k(\X)$ is equal to $\normbpfs{\cdot}$ by~\eqref{eq:dual-norm}, hence the inclusion $\numbpfs{k}(\X)\subset\numbpf{d-k}(\X)^*$ is a continuous injection.
\end{proof}

\begin{prop}
There exists a canonical continuous map 
\[\rho \colon (\Vect(\BPF^{d-k}(\X)) , \normbpf{\cdot} )^* \to \wnum^k(\X)\] 
such that 
$\ell(\gamma)= (\rho(\ell) \cdot \gamma)$ for all $\gamma \in \cnum^{d-k}(\X)$. 
The image of $\rho$ is equal to $\numbpf{d-k}(\X)^*$ and $\rho$ coincides with the  restriction
morphism induced by the inclusion $\numbpf{d-k}(\X)\subset\Vect(\BPF^{d-k}(\X))$.
\end{prop}
\begin{proof}
The existence of the map $\rho$ is a consequence of the perfect pairing $\wnum^k(\X) \times  \cnum^{d-k}(\X)\to \R$.
Its image is the set of $\alpha\in \wnum^k(\X)$ such that $|(\alpha\cdot \gamma)| \le C_\alpha \normbpf{\gamma}$ for all $\gamma\in\cnum^{d-k}(\X)$, hence
coincides with $\numbpf{d-k}(\X)^*$ by Proposition \ref{prop_inclusion_continuous_duals}.

The restriction morphism induced by the inclusion 
\[\numbpf{d-k}(\X)\subset\Vect(\BPF^{d-k}(\X))\] sends a continuous form $\ell\in \Vect(\BPF^{d-k}(\X))^*$
to its restriction on  $\numbpf{d-k}(\X)$. Since $\rho(\ell)$ is determined by its values on $\cnum^{d-k}(\X)$, which is dense in $\numbpf{d-k}(\X)$, it actually follows that $\rho(\ell)$ equals
$\ell|_{\numbpf{d-k}(\X)}$. 
Conversely, Hahn-Banach theorem implies the surjectivity of $\rho$ onto $(\numbpf{d-k}(\X))^*$. 
\end{proof}

\begin{rem}
The morphism $\rho$ should not be injective in general, see Remark~\ref{rem:not-dense} above.
\end{rem}

\begin{cor}
The intersection pairing on $\cnum^k(\X) \times \cnum^{d-k}(\X) \to \R$ induces a bilinear continuous form
\[\numbpfs{k}(\X) \times  (\Vect(\BPF^{d-k}(\X)), \normbpf{\cdot} ) \to \R.\]
\end{cor}

To illustrate our construction, we show that the basepoint free cone satisfies some weak-compactness properties. 

\begin{prop} \label{prop_bpf_weak_compactness} Let $\alpha_n \in \cBPF^k (\X) $ be a sequence  such that $(\alpha_n \cdot \omega^{d-k})_{n\in \mathbb{N}}$ is bounded. 
Then there exists a subsequence $(\alpha_{n_j})$ converging in $\wnum^k(\X)$.
\end{prop}

\begin{proof} We show  that the sequence $\alpha_n$ belongs to a  ball in $\numbpfs{k}(\X)$.
Take $\gamma \in \cBPF^{d-k}(\X)$. By Siu's  inequalities~\eqref{eq:siu_bpf}, we have:
\begin{equation*}
0\le (\alpha_n \cdot \gamma) \le C_d (\gamma \cdot \omega^{k}) (\alpha_n \cdot \omega^{d-k}),
\end{equation*}
and so  the sequence $\normbpfs{\alpha_n}$ is bounded. By Proposition \ref{prop_inclusion_continuous_duals}  the inclusion $\numbpfs{k}(\X) \subset \numbpf{d-k}(\X)^*$ is continuous, hence the sequence $\alpha_n$ belongs to a ball inside $\numbpf{d-k}(\X)^*$. We conclude using Banach-Alaoglu's theorem that $(\alpha_n)$ has a subsequence which converges in $\wnum^k(\X)$.
\end{proof}

%%%%%%%%%%%%%%%%%%%%%%%%%%%%%%%%%%%%%%%%%%%%%%%%%%%%%%%%%%%

\section{Intersection of numerical  $b$-divisors}

We introduce a subspace $\nums(\X)$ of $\wnum^1(\X)$ generalizing to higher dimensions the construction of the Picard-Manin space
introduced by Cantat~\cite{cantat_bir_surfaces} and Boucksom-Favre-Jonsson~\cite{boucksom_favre_jonsson_deggrowth}.

\subsection{$\LL^2$-norm on surfaces}

Let $S$ be any normal projective surface defined over $\K$. Recall that one can define in this context an intersection form on Weil divisors
using Mumford numerical pull-back to a resolution of singularities of $S$, see~\cite{MR153682} for details.

Denote by $\Ss$ the Riemann-Zariski space of $S$ as discussed in \S \ref{sec:numerical_b_class}.
For any big and nef class $\om\in\cnum^1(\Ss)$, we set 
\begin{equation}\label{eq:def-L2}
\bar q_\om (\alpha)= 2\frac{(\alpha\cdot \om)^2}{(\om^2)} - (\alpha^2) = \frac{(\alpha\cdot \om)^2}{(\om^2)} - \left( \left(\alpha - \frac{(\alpha\cdot \om)}{(\om^2)}\om\right)^2 \right),
\end{equation}
where $\alpha \in \cnum^1(\Ss)$. Note that the second equality comes from the orthogonal decomposition $\alpha = u + v$ with $u= \dfrac{(\alpha \cdot \omega)}{(\omega^2)} \omega \in \R\cdot \omega$ and $v = \alpha - u \in \omega^\perp$.

Note that $\bar q_\om (\alpha)= \bar q_{t \om} (\alpha)$ for any $t>0$.
It follows from the Hodge index theorem~\cite[Theorem 1.9]{hartshorne} that 
the intersection form is negative definite on $\{ \alpha, \, (\alpha\cdot \omega) =0\}$, so that
\begin{equation}\label{eq:hodge-onS}
\bar q_\om(\alpha) \ge \frac{(\alpha\cdot \om)^2}{(\om^2)}
\end{equation}
and $\bar q_\om$
 defines a positive definite quadratic form on $\cnum^1(\Ss)$.
Taking the square root, the function $\sqrt{\bar q_\omega(\cdot)}$  induces a norm on $\cnum^1(\Ss)$ which depends on the choice of $\omega$.
The next result shows that two big and nef divisors give equivalent norms.
\begin{prop}\label{prop_estimates_minkowski}  
Pick any two big and nef divisors $\om_1, \om_2\in \cnum^1(\Ss)$.
Then for any $\alpha \in \cnum^1(\Ss)$, one has:
\begin{equation*}
\bar q_{\om_1}(\alpha) \le 4\, \frac{(\om_1 \cdot \om_2)^2}{(\om_1^2) (\om_2^2)} \,\bar  q_{\om_2}(\alpha).
\end{equation*}
\end{prop}

\begin{proof}
We may assume that $\om_1$ and $\om_2$ are determined in a model $S'$ over $S$, and that $(\om_1^2) = (\om_2^2) = 1$. 
Observe that any class $\alpha \in \cnum^1(\Ss)$ can be written as $\alpha = \alpha_1 + \beta$
with $\alpha_1\in\Vect(\om_1,\om_2)\subset \num^1(S')$ and $(\beta\cdot\om_1)=(\beta\cdot\om_2)=0$. 
Observe that  $\Omega :=  (\om_1\cdot \om_2)\ge 1$ by Hodge index theorem.
Since $\bar q_{\om_i}(\alpha)=q_{\om_i}(\alpha_1) -(\beta^2)$, and 
$(\beta^2)\le0$,
we are reduced to prove the inequality when $\alpha\in\Vect(\om_1,\om_2)$.

Suppose that $\alpha =  \om_1+ t\om_2$. A direct computation yields
\begin{align*}
\frac{\bar  q_{\om_1}(\alpha)}{\bar q_{\om_2}(\alpha)}
&=
\frac{ 2(\om_1\cdot \alpha)^2 - (\alpha^2)}{2(\om_2\cdot \alpha)^2 - (\alpha^2)}
=
\frac{( 2\Omega^2 -1) t^2 + 2 \Omega t + 1}
{ t^2 + 2 \Omega t + 2\Omega^2 -1}.
\end{align*}
Taking the limit as $t\rightarrow +\infty$ and the value at $t = -\Omega$ give:
\[
\frac{\bar  q_{\om_1}(\om_2)}{\bar q_{\om_2}(\om_2)}=
\frac{\bar  q_{\om_1}(\om_1 - \Omega \om_2)}{\bar q_{\om_2}(\om_1 - \Omega \om_2)}
=
2\Omega^2 -1.
\]
Since the two vectors $e_1=\om_1 - \Omega \om_2$ and $e_2=\om_2$ form an orthogonal basis for $(\Vect(\om_1,\om_2),q_{\om_2})$, 
we conclude using the above inequality that 
\begin{align*}
\sup_{\Vect(\om_1,\om_2)}
\left(
\frac{\bar  q_{\om_1}(\alpha)}{\bar q_{\om_2}(\alpha)}
\right)^{1/2}
&=
\sup_{t_1, t_2}\, 
\frac{\bar  q_{\om_1}(t_1 e_1 + t_2 e_2)^{1/2}}{\bar q_{\om_2}(t_1 e_1 + t_2 e_2)^{1/2}}
\\
&\le
\sup_{t_1, t_2}\, 
\frac{\bar q_{\om_1}(t_1 e_1)^{1/2} +\bar  q_{\om_1}(t_2 e_2)^{1/2}}{(\bar q_{\om_2}(t_1 e_1) + \bar q_{\om_2}(t_2 e_2))^{1/2}}
\\
&\le
\sup_{t_1, t_2}\, 
 \frac{\bar  q_{\om_2}(t_1 e_1)^{1/2}}{\bar q_{\om_2}(t_1 e_1)^{1/2}} + \frac{\bar q_{\om_2}(t_2 e_2)^{1/2}}{\bar q_{\om_2}(t_2e_2)^{1/2}} 
\le
2\sqrt{2\Omega^2 -1},
\end{align*}
as required.
\end{proof}

\begin{prop} \label{prop_surfaces_continuity_intersection}
 For any big and nef class $\om\in\cnum^1(\Ss)$ such that $(\om^2)=1$, 
  for any $\alpha, \beta \in \cnum^1(\mathcal{S})$, the following inequality holds:
\begin{equation}
|(\alpha \cdot \beta)| \le 3\bar  q_\om(\alpha)^{1/2}\bar  q_\om(\beta)^{1/2}.  
\end{equation}
\end{prop}

\begin{proof} 
Let $\langle \cdot , \cdot  \rangle_\om$ be the bilinear symmetric form associated to $q_\om$ so that 
\begin{equation*}
\langle \alpha, \beta \rangle_\om = 2 (\alpha \cdot \om) (\beta \cdot \om) - (\alpha\cdot \beta)
\end{equation*}
for any $\alpha, \beta \in \cnum^1(\mathcal{S})$. 
Note that this formula is equivalent to 
$(\alpha \cdot\beta) =  2(\alpha\cdot \om) (\beta \cdot \om) - \langle \alpha, \beta \rangle_\om$.
Cauchy-Schwarz and triangle inequalities then imply
\begin{equation*} \label{eq_pairing}
|(\alpha \cdot \beta)| \le 2|(\alpha \cdot \om)(\beta \cdot \om)| + \sqrt{\bar q_\om(\alpha) \bar q_\om(\beta)}.
\end{equation*}
Observe that $|\bar q_\om(\alpha)| \ge  (\alpha\cdot \om)^2$ since the self-intersection of $\alpha - (\alpha \cdot \om)\om$ is non-positive, so that 
\begin{equation*}
|(\alpha\cdot \beta)| \le 3 \sqrt{\bar q_\om(\alpha)\bar q_\om(\beta)},
\end{equation*}
as required.
\end{proof}

We observe the following:
\begin{lem}\label{lem:increase-square}
Let $\alpha$ be any Weil class in $\wnum^1(\mathcal{S})$. For any two smooth models $X'' \ge X'$, we have
$ - (\alpha_{X''}^2) \ge- (\alpha_{X'}^2)$.
\end{lem}
\begin{proof}
By Castelnuovo's factorization, we reduce to the case $X''$ is the blow-up of $X'$ at a single point. Let $E$ be the exceptional divisor of 
the canonical morphism $\pi\colon X'' \to X'$. Then $\alpha_{X''} = \pi^*\alpha_{X'} + t [E]$ for some $t$
so that $(\alpha_{X''})^2=  (\alpha_{X'})^2 - t^2 \le  (\alpha_{X'})^2$.
\end{proof}

\begin{defi}
The space $L^2(\mathcal{S})$ consists of those Weil classes $\alpha\in\wnum^1(\mathcal{S})$
such that $\sup_{X'\in\mathcal{M}_X} - (\alpha_{X'})^2 < +\infty$.
\end{defi}

The space $L^2(\mathcal{S})$ is the Picard-Manin space of $S$, see~\cite{boucksom_favre_jonsson_deggrowth}, or~\cite{cantat_bir_surfaces}.
The quadratic form $\bar q_\om(\cdot)$ induces a scalar product on $L^2(\mathcal{S})$ which endows it with a structure of a Hilbert space.

%%%%%%%%%%%%%%%%%%%%%%%%%%%%%%%%%%%%

\subsection{The space $\FS(\X)$}\label{sec:FSX}
Let $X$ be a smooth projective variety of dimension $d$ defined over a countable field $\K$
of characteristic $0$, and let $\om\in\cnum^1(\X)$ be any big and nef class.
For any non-zero class $\gamma \in \BPF^{d-2}(\X)$ and any $\alpha \in \cnum^1(\X)$, we set:
\begin{equation} \label{formula_q_w}
q_{\om,\gamma} (\alpha) :=  2 \frac{(\alpha \cdot  \om \cdot \gamma )^2}{(\om^2\cdot\gamma)^2}  -  \frac{( \alpha^2 \cdot \gamma)}{(\om^2\cdot\gamma)}~.
\end{equation}

Note that $(\om^2\cdot \gamma) >0$ by~\eqref{eq:5436} hence $q_{\om,\gamma} (\alpha)$ is well-defined, and that
$q_{\om,\gamma}$ is a quadratic form in $\alpha$ which is homogeneous of degree $0$ in $\gamma$ (resp. of degree $-2$ in $\om$).

{Observe that when $\dim X=2$, this quadratic form is related to equation \eqref{eq:def-L2} by the relation $q_{\omega, \gamma} = \bar q_\omega / (\omega^2) $.}

\begin{lem}\label{lem:hausd}
For all classes $\alpha\in\cnum^1(\X)$, $0\neq\gamma\in\BPF^{d-2}(\X)$,
we have 
\begin{equation}\label{eq:pos-nums}
q_{\om,\gamma} (\alpha)= \frac1{(\om^2\cdot\gamma)}\left[ \frac{(\alpha\cdot \om\cdot\gamma)^2}{(\om^2\cdot\gamma)} - \left( \left(\alpha - \frac{(\alpha\cdot \om\cdot\gamma)}{(\om^2\cdot\gamma)}\om\right)^2\cdot\gamma \right)\right]
\ge 0.
\end{equation}
Moreover the equality $q_{\om,\gamma} (\alpha)=0$ holds for all $\gamma\in\cBPF^{d-2}(\X)$
iff $\alpha=0$. 
\end{lem}
\begin{proof}
The equality~\eqref{eq:pos-nums}
follows from a direct computation. To prove $q_{\om,\gamma} (\alpha)\ge0$ we may assume by continuity  
of the intersection product with respect to the weak topology
that $\gamma$ is Cartier, and normalize it so that $(\om^2\cdot\gamma)=1$.
Fix a smooth model $X'$ on which  the three classes $\alpha, \om$ and $\gamma$ are all determined. 
Again by continuity, we may assume that $\gamma_{X'}$ is strongly basepoint free so that
we have a flat morphism $p\colon Y \to X'$ of relative dimension $e$ between smooth projective varieties, 
and $e+d-2$ very ample divisors $D_1 , \ldots , D_{e+d-2}$ on $Y$ such that $\gamma_{X'} = p_* (D_1 \cdot \ldots \cdot D_{e+d-2})$. 
Let $S$ be any codimension $e+d-2$ smooth subvariety in $Y$  representing the class $D_1 \cdot \ldots \cdot D_{e+d-2}$ and such that $p(S)$ is of pure dimension $2$.
Let $\imath \colon S \hookrightarrow Y$  and $\jmath \colon p(S) \hookrightarrow X'$ be the two natural inclusions and by $\tilde{p} \colon S \to p(S)$ the morphism induced by $p$ on $S$ so that the following diagram is commutative:
\begin{equation*}
\xymatrix{ S \ar[d]^{\tilde{p}} \ar[r]^\imath & Y \ar[d]^{p}  \\
p(S) \ar[r]^{\jmath} & X'.}
\end{equation*}
Using the projection formula, we get:
\begin{align*}
q_{\om,\gamma}(\alpha) &= 2 (\alpha \cdot \om \cdot \gamma)^2 - (\alpha^2 \cdot \gamma)\\
 &= 2 (p^* \alpha_{X'} \cdot p^*\om_{X'} \cdot D_1 \cdot \ldots \cdot D_{e+d-2})^2 - (p^* \alpha_{X'}^2 \cdot D_1 \cdot \ldots \cdot D_{e+d-2}) \\
  &= (\imath^* p^* \alpha_{X'} \cdot \imath^* p^* \om_{X'} )^2 - (\imath^* p^* \alpha_{X'}^2 ) = q_{\imath^* p^* \om_{X'} }(\imath^* p^* \alpha_{X'})\ge 0
\end{align*} 
by~\eqref{eq:hodge-onS}.

Suppose that $q_{\om,\gamma} (\alpha)=0$ for all non-zero $\gamma\in\cBPF^{d-2}(\X)$. It is enough to prove that $(\alpha \cdot \beta)=0$
for any strongly basepoint free class $\beta \in \cBPF^{d-1}(\X)$.  Replacing $\beta$ by a multiple if necessary,
one can find a smooth model $X'$, a flat morphism $p\colon Y \to X'$ of relative dimension $e$ between smooth projective varieties, and $e+d-1$ very ample divisors $D_1 , \ldots , D_{e+d-1}$ on $Y$ such that $\beta = p_* (D_1 \cdot \ldots \cdot D_{e+d-1})$. 

As before, let $S$ be a codimension $e+d-2$ smooth subvariety in $Y$  representing the class $D_1 \cdot \ldots \cdot D_{e+d-2}$  such that $\dim(p(S))=2$, and
 $\imath \colon S \hookrightarrow Y$  and $\jmath \colon p(S) \hookrightarrow X'$ the natural inclusions. We obtain in a similar way:
 \[
 0= q_{\om,\gamma}(\alpha) = q_{\imath^* p^* \om_{X'} }(\imath^* p^* \alpha_{X'})
\] 
which implies $\imath^* p^* \alpha = 0$ in $\num^1(S)$ since $q_{\imath^* p^* \om }$ is a norm on $\cnum^1(\Ss)$.
Now we may write
\[
(\alpha \cdot \beta) = (p^* \alpha \cdot D_1 \cdot \ldots \cdot D_{e+d-1}) = (\imath^* p^*\alpha \cdot \imath^* D_{e+d-1})=0
\]
as was to be shown.
\end{proof}

It follows that the (countable) family of semi-norms $\{q_{\om,\gamma}^{1/2}\}$ where $\gamma$ ranges over all BPF classes defines a Hausdorff topology on $\cnum^1(\X)$, see~\cite[\S I.1]{yosida}. 

\begin{defi}
The space $\FS(\X)$ is defined as the completion of $\cnum^1(\X)$ with respect to the family of semi-norms $q_{\omega, \gamma}^{1/2}$ for $\gamma \in \BPF^{d-2}(\X)$.
\end{defi}

The space $\FS(\X)$ is a Fréchet space
\footnote{when $\K$ is no longer assumed to be countable, there is no countable family of Cartier basepoint free classes that is dense in $\BPF^{d-2}(\X)$, so $\FS(\X)$
is not a Fr\'echet space.}, see~\cite[p.52]{yosida}.
We call the natural topology of $\FS(\X)$ induced by $(q_{\omega, \gamma}^{1/2})_{\gamma \in \BPF^{d-2}(\X)}$  the Fr\'echet topology.

A sequence $(\alpha_n)_n \in \cnum^1(\X)$ is a Cauchy sequence for the Fr\'echet topology if for all $\gamma \in \BPF^{d-2}(\X)$, $(\alpha_n)$ is Cauchy for the semi-norm $q_{\omega,\gamma}^{1/2}$. 
This topology is weaker than the topology we shall define in section \S \ref{section_nsigma}, for which the convergence will be assumed uniform in the choice of $\gamma \in \BPF^{d-2}(\X)$.

We begin with the following observation.
\begin{prop}\label{prop:equi-semi-norms}
For any two big nef divisors $\om_1,\om_2$ on $X$, there exists a constant $C>1$ which depends on $\om_1, \om_2$ such that  
\[
C^{-1}\, q_{\om_2,\gamma} 
\le 
q_{\om_1,\gamma} 
\le
C\, q_{\om_2,\gamma} 
\]
for all non-zero classes $\gamma \in \BPF^{d-2}(\X)$.
\end{prop}
In particular, the space $\FS(\X)$ does not depend on the choice of a big and nef class $\om$.

\begin{proof}
We may suppose that $(\om_1^{d})=(\om_2^{d})=1$. 
We prove the proposition for $C=\left (2d \max((\om_1\cdot \om_2^{d-1}),(\om_2\cdot \om_1^{d-1}) )\right )^2$.
Proposition~\ref{prop_estimates_minkowski} implies that for all $\alpha\in \cnum^1(\X)$ and all strongly basepoint free classes $\gamma \in \cBPF^{d-2}(\X)$,
\begin{equation}\label{eq:7654}
q_{\om_1,\gamma} (\alpha) \le  4\frac{(\om_1\cdot \om_2 \cdot \gamma)^2}{(\om_1^2 \cdot \gamma)^2} q_{\om_2, \gamma}(\alpha).
\end{equation}
By Siu's  inequalities~\eqref{eq:siu_bpf}, we have $\om_2 \le d (\om_2 \cdot \om_1^{d-1})\om_1$ and the above inequality gives:
\[
q_{\om_1,\gamma}(\alpha) \le 4d^2 (\om_2\cdot \om_1^{d-1})^2
q_{\om_2,\gamma}(\alpha), 
\]
proving that the two semi-norms are equivalent.
\end{proof}

\begin{thm}
The inclusion $\cnum^1(\X) \subset \wnum^1(\X)$ extends as a continuous morphism
$\imath\colon \FS(\X) \rightarrow \wnum^1(\X)$ whose image is included in the set of classes $\alpha \in \wnum^1(\X)$
for which
\begin{equation}\label{eq:sup-norm}
\sup_{X'\in \cM_X
 } |(\alpha_{X'}^2 \cdot \gamma)|  
 \le C (\gamma\cdot\om^2)
\end{equation}
for some $C= C_\alpha>0$ and for all  $\gamma \in \BPF^{d-2}(\X)$.
\end{thm}

\begin{rem}
We do not know whether $\imath$ is injective or not (see Conjecture~\ref{conjimath} below)
\end{rem}

\begin{proof}
Let $(\alpha_n)_n$ be any Cauchy sequence for the Fr\'echet topology, and let us prove that $\alpha_n$ converges weakly in $\wnum^1(\X)$.
We need to show that $\alpha_n\cdot \beta$ is convergent for any Cartier class $\beta\in\cnum^{d-1}(\X)$. 
We may suppose by Theorem \ref{thm_bpf} that $\beta=(D\cdot\gamma)$ with $\gamma\in\cBPF^{d-2}(\X)$, and $D\in\cNef(\X)$.
The assumption that $\alpha_n$ is a Cauchy sequence in $\FS(\X)$ shows that $q_{\om,\gamma}(\alpha_n - \alpha_m) \to0$ uniformly in $n,m$, 
hence $(\alpha_n\cdot D\cdot \gamma)$ is also a Cauchy sequence by Proposition~\ref{prop_surfaces_continuity_intersection}.
The inclusion morphism $\imath_\Sigma\colon \cnum^1(\X) \to \wnum^1(\X)$ thus extends to $\FS(\X)$ so that 
$\imath_\Sigma(\alpha_n) \to \imath_\Sigma(\alpha)$ weakly in $\wnum^1(\X)$ if $\alpha_n\to\alpha$ in $\FS(\X)$.

\medskip
Now pick any $\alpha\in\FS(\X)$ and any class $\gamma\in\cBPF^{d-2}(\X)$. By Lemma~\ref{lem:increase-square}, the sequence of intersection numbers
$ - (\alpha_{X'}^2\cdot \gamma)$ is increasing in the poset of smooth models $X'\in\cM_X$. We claim that $ \sup_{X'} - (\alpha_{X'}^2\cdot \gamma)$ is finite.  Take any Cauchy sequence $\alpha_n \in \cnum^1(\X)$ converging to $\alpha$ in $\FS(\X)$. 
As $\alpha_n$ converges to $\alpha$ weakly, the sequence $(\alpha_n \cdot \om \cdot \gamma)$ converges to $(\alpha \cdot \om \cdot \gamma)$. But $\alpha_n$ is a Cauchy sequence for the semi-norm $q_{\omega, \gamma}$, hence the sequence $|(\alpha_n^2 \cdot \gamma)| $ is bounded, say by $M>0$ by the next lemma.

\begin{lem}\label{lem:bounded-norm}
For any non-zero $\gamma\in\BPF^{d-2}(\X)$, if $\alpha_n$ is $q_{\om,\gamma}$-Cauchy then both sequences
$(\alpha_n\cdot\om\cdot\gamma)$ and $(\alpha_n^2\cdot\gamma)$ are Cauchy.
\end{lem}

We infer
\[\sup_{X' \in\cM_X} -(\alpha_{n,X'}^2 \cdot \gamma ) = -(\alpha_{n}^2 \cdot \gamma ) \le M,
\]
and taking the limit as $n\rightarrow +\infty$ in each model yields
$\sup_{X' \in\cM_X} |(\alpha_{X'}^2 \cdot \gamma )| \le M$
as required.

Now the linear form $L_{X'}(\gamma) := (\alpha_{X'}^2 \cdot \gamma )$ is continuous on $\Vect(\BPF^{d-2}(\X))$
and for each $\gamma$ we have  $\sup_{X'} |L_{X'}(\gamma) |<\infty$. 
By Banach-Steinhaus theorem (a.k.a. the uniform boundedness principle), we get $\sup_{X'} |L_{X'}(\gamma)| \le C \normbpf{\gamma}$.
\end{proof}

\begin{proof}[Proof of Lemma~\ref{lem:bounded-norm}]
This follows from the observation that if $(\om^2\cdot\gamma)=1$, then we have
\[
q_{\om,\gamma}(\alpha) = (\alpha\cdot\om\cdot\gamma)^2 + \Delta 
\]
where $\Delta = - ((\alpha - (\alpha\cdot\om) \om )^2 \cdot \gamma )$.
By~\eqref{eq:pos-nums}, $\Delta \ge0$  for any BPF class.
If $q_{\om,\gamma}^2(\alpha_n-\alpha_m)\le \epsilon$, then $|(\alpha_n - \alpha_m\cdot\om\cdot\gamma)| \le \epsilon$
and the triangle inequality implies $q_{\om,\gamma}(\alpha_n)$ to be a Cauchy sequence.
Using~\eqref{formula_q_w}, this implies  $(\alpha^2_n\cdot\gamma)$ to be a Cauchy sequence too.
\end{proof}

%%%%%%%%%%%%%%%%%%%%%%%%%%%%%%%%%%%%%%%%%%%%%%%%%%%%%%%%%%%%%%

\subsection{The pairing $\FS(\X)\times \FS(\X) \to \wnum^{2}(\X)$}

Let $\om\in\cnum^1(\X)$ be any big and nef  class.

\begin{thm}For any $\gamma \in \BPF^{d-2}(\X)$,
the map
$\alpha, \beta\mapsto (\alpha\cdot\beta\cdot\gamma)\in\R$
with $\alpha,\beta\in\cnum^1(\X)$ defines a bilinear
form on $\FS(\X)$. It induces a bilinear map
\begin{equation}\label{eq:intersection-FS}
\FS(\X)\times \FS(\X) \to \wnum^{2}(\X)
\end{equation}
which is continuous in the following sense: 
if $\alpha_n\to\alpha$ and $\beta_n\to\beta$
in the Fr\'echet topology, then $(\alpha_n\cdot\beta_n)\to (\alpha\cdot\beta)$ weakly.  
Moreover it
satisfies
\begin{equation}\label{eq:bound-norms}
|(\alpha \cdot \beta\cdot\gamma)| \le 3 q_{\om,\gamma}(\alpha)^{1/2} q_{\om,\gamma}(\beta)^{1/2} (\om^2\cdot\gamma).
\end{equation}
\end{thm}

\begin{proof}
The map extends by continuity as a multi-linear map to $\FS(\X)\times \FS(\X)\times\Vect(\BPF^{d-2}(\X))$ by Lemma~\ref{lem:bounded-norm}.
It remains to prove~\eqref{eq:bound-norms}.

By continuity, we may suppose that $\alpha,\beta \in \cnum^1(\X)$, and $\gamma \in \cBPF^{d-2}(\X)$ is a strongly basepoint free class such that $(\om^2 \cdot \gamma)=1$. 
 Write $\gamma= p_* (D_1 \cdot \ldots \cdot D_{e+d-2})$ where $D_1, \ldots, D_{e+d-2}$ are very ample divisors and where $p\colon Y \to X'$ is a flat morphism of relative dimension $e$. 
 Take a smooth surface $S$ whose cycle class represents 
 $\hat{\gamma}:=(D_1 \cdot \ldots \cdot D_{e+d-2})\in\num^{e+d-2}(Y)$ and denote by $\imath\colon S \to Y$ the natural inclusion morphism.
We obtain
 \[
(\alpha \cdot \beta\cdot\gamma) = (\imath^* p^*\alpha \cdot \imath^* p^*\beta ).
\]
By Proposition~\ref{prop_surfaces_continuity_intersection}, and the projection formula we have:
\begin{equation*}
|(\alpha \cdot \beta\cdot\gamma)| \le 3 \sqrt{q_{\imath^* p^*\om}(\imath^*p^*\alpha) q_{\imath^* p^*\om}(\imath^*p^*\beta )}
=3 q_{\om,\gamma}(\alpha)^{1/2} q_{\om,\gamma}(\beta)^{1/2}.
\end{equation*}
This concludes the proof.
\end{proof}

From the continuity of the pairing with respect to the Fr\'echet topology, and from the Banach-Steinhaus theorem applied to the Banach space
$(\Vect(\BPF^{d-2}(\X)),\normbpf{\cdot})$, we infer:
\begin{cor}\label{cor:comput}
For any $\alpha\in\FS(\X)$ and any $\gamma\in\BPF^{d-2}(\X)\setminus\{0\}$, we have 
\[
q_{\om,\gamma}(\alpha) =  2 \frac{(\alpha \cdot  \om \cdot \gamma )^2}{(\om^2\cdot\gamma)^2}  -  \frac{( \alpha^2 \cdot \gamma)}{(\om^2\cdot\gamma)}
~,\]
and
\[\sup_{0\neq \gamma\in\BPF^{d-2}(\X)} q_{\om,\gamma}(\alpha) <\infty.\]
\end{cor}

The following three questions remain open. 
\begin{conj}\label{conjimath}
The canonical morphism $\imath_{\Sigma}\colon\FS(\X)\to\wnum^1(\X)$ is injective.
That is suppose $\alpha_n\in \cnum^1(\X)$ is a Cauchy sequence for any $q_{\om,\gamma}$-seminorm, and
$\alpha_n \to0$ weakly. Is it true that $q_{\om,\gamma}(\alpha_n) \to 0$ for all $0\neq\gamma \in\BPF^{d-2}(\X)$? 
\end{conj}
Note that under the above assumptions, it is possible to prove that $q_{\om,\gamma}(\alpha_n) \to 0$ for all $\gamma \in\cBPF^{d-2}(\X)$.

\begin{conj}\label{conjFS}
For any $\alpha \in \FS(\X)$, we have $\alpha_{X'}:=[\imath_\Sigma(\alpha)_{X'}] \to \alpha$ in the Fr\'echet topology. 
\end{conj}

Note that Conjecture~\ref{conjFS} implies Conjecture~\ref{conjimath}.

\begin{conj}
The image $\imath_\Sigma(\FS(\X))$ coincides with the set of Weil classes $\alpha$ such that $\sup_{X'} |(\alpha_{X'}^2\cdot \gamma)|< \infty$ for all $\gamma \in\BPF^{d-2}(\X)$.
\end{conj}

%%%%%%%%%%%%%%%%%%%%%%%%%%%%%%%%%%%%%%

\subsection{The Banach space $\nums(\X)$}
\label{section_nsigma}
Our spectral analysis will take place in a Banach space instead of the Fr\'echet space $\FS(\X)$. To that end, we
consider the function
$ \normsigma{\cdot}{\om} \colon \cnum^1(\X) \to \R\cup\{\infty\}$ defined by:
\begin{equation} \label{def_norm_l20}
\normsigma{\alpha}{\om} :=
\sup_{ \substack{\gamma \in \cBPF^{d-2}(\X) \\
(\gamma \cdot \om^2) = 1}}  \left( 2 (\alpha \cdot  \om \cdot \gamma )^2  -  ( \alpha^2 \cdot \gamma) \right)^{1/2}
= 
\sup_{0\neq\gamma \in \BPF^{d-2}(\X)} 
q_{\om,\gamma}(\alpha)^{1/2}~.
\end{equation}
We shall see below that $\normsigma{\alpha}{\om}$ is finite for any Cartier divisor class.

Observe that
\[
\normsigma{\alpha}{t \om}
= 
t^{-1}\, \normsigma{\alpha}{\om}
\]
for all $t>0$.

\begin{thm} \label{thm_norm_l2} For any big nef divisor $\om\in\cnum^1(\X)$, and for any $\alpha\in\cnum^1(\X)$, the quantity 
$\normsigma{\alpha}{\om}$ is finite, and the function $\normsigma{\cdot }{\om} \colon \cnum^1(\X) \to \R$ defines a norm. 

Moreover, for any pair of big nef divisors $\om,\om'$ on $X$, the two norms $\normsigma{\cdot}{\om}$ and $\normsigma{\cdot}{\om'}$ are equivalent.
\end{thm}

\begin{proof}
Siu's  inequalities~\eqref{eq:siu_bpf} imply $\gamma \le C_\om \om^{d-2}$ for any BPF class such that $(\om^2\cdot \gamma)=1$, with $C_\om= C_d/(\omega^d)>0$, hence
for $\alpha \in \nef^1(\X)\cap \cnum^1(\X)$,
\begin{equation} \label{eq_nef_comp}
\normsigma{\alpha}{\om}^2 = 2  (\alpha \cdot \om \cdot \gamma)^2 - (\alpha^2 \cdot \gamma) \le 2 C^2_\om (\alpha\cdot \om^{d-1})^2<\infty
\end{equation}
if $\alpha$ is nef. 
Since any Cartier $b$-divisor class  is the difference of two nef $b$-divisors, we get that $\normsigma{\cdot}{\om}$ is well-defined on $\cnum^1(\X)$.
It is a norm by Lemma~\ref{lem:hausd} and Proposition~\ref{prop:equi-semi-norms} implies $\normsigma{\cdot}{\om}$ and $\normsigma{\cdot}{\om'}$ are equivalent.
\end{proof}

\begin{defi} We let $\nums(\X)$ be the Banach space obtained as the completion of the space $\cnum^1(\X)$ with respect to the norm $\normsigma{\cdot}{\om}$.   
\end{defi}

\begin{rem} 
When $X$ is a surface,  $\FS(\mathcal{X})=\nums(\mathcal{X}) = \LL^2(\mathcal{X})$ is a Hilbert space.
\end{rem}

\begin{thm} \label{thm_banach} 
We have a canonical injection $\nums(\X) \subset \FS(\X)$. For any pair of classes $\alpha,\beta\in \nums(\X)$
the intersection product $\alpha\cdot \beta$ defined by~\eqref{eq:intersection-FS} belongs to $\numbpfs{2}(\X)$ and
we have: 
\begin{align} \label{def_norm_l2}
&\normsigma{\alpha}{\om}^2 =
\sup_{\substack{\gamma \in \cBPF^{d-2}(\X) \\
(\gamma \cdot \om^2)>0}} 
q_{\om,\gamma}(\alpha)~.
\\
\label{eq:bound-norms2}
&\normbpfs{\alpha \cdot \beta} \le 3 \normsigma{\alpha}{\om} \normsigma{\beta}{\om},
\end{align}
The restriction of $\imath_\Sigma \colon \FS(\X) \to \wnum^1(\X) $ to $\nums(\X)$ is injective.
Moreover, there exists a constant $C>0$ such that $\normsigma{\cdot}{\om}\le C\,\normbpf{\cdot}$ on $\cnum^1(\X)$.
In particular, we have the following continuous injections:
\begin{equation*}
( \numbpf{1}(\X), \normbpf{\cdot})  \hookrightarrow (\nums(\X), \normsigma{\cdot}{\om}) \hookrightarrow \wnum^1(\X)
\end{equation*}
where the latter space is endowed with the weak topology.
\end{thm}

\begin{rem}\label{rem:compare-for-nef}
Note that if $\alpha\in\nef^1(\X)$, then 
\[
\sqrt{2}\, \sup_{\substack{\gamma\in\cBPF^{d-2}(\X)\\(\gamma\cdot\om^2)=1}} (\alpha \cdot \om \cdot \gamma)
\ge \normsigma{\alpha}{\om}
\ge 
\sup_{\substack{\gamma\in\cBPF^{d-2}(\X)\\(\gamma\cdot\om^2)=1}} (\alpha \cdot \om \cdot \gamma)
\]
so that Siu's  inequalities~\eqref{eq:siu_bpf} and the fact that $\om$ is Cartier give
\begin{equation}\label{eq:666}
\sqrt{2}C_d \, \frac{(\alpha\cdot\om^{d-1})}{(\om^d)}
\ge \normsigma{\alpha}{\om}
\ge  \frac{(\alpha\cdot\om^{d-1})}{ (\om^d)}.
\end{equation}
\end{rem}

\begin{rem}
In the sequel to this paper~\cite{dang_favre}, we shall prove the inclusions 
$\nef^1(\X)\subset\nums(\X) \subset \numbpfs{1}(\X)$.
\end{rem}

\begin{proof}[Proof of Theorem \ref{thm_banach}] 
Note that any Cauchy sequence for the norm $\normsigma{\cdot}{\om}$ is a Cauchy sequence for $q_{\om,\gamma}$ for all $\gamma\in\BPF^{d-1}(\X)$ satisfying $(\gamma \cdot \om^2)>0$.
This implies the existence of a canonical embedding $\nums(\X)\subset \FS(\X)$.

Let us prove the identity~\eqref{def_norm_l2}.
 Pick any sequence $\alpha_n\in\cnum^1(\X)$,
such that \[\normsigma{\alpha_n - \alpha}{\om}^2 = \sup_{0\neq\gamma\in\cBPF^{d-2}(\X)} 
q_{\om,\gamma}(\alpha_n-\alpha) \to 0,\]
{where the right hand side follows from  \eqref{def_norm_l20}.}\\
 By definition, we have 
$\normsigma{\alpha}{\om}= \lim_{n\to\infty} \normsigma{\alpha_n}{\om}$.
Choose any $\epsilon>0$. 
Since $\alpha_n$ is Cartier, one can find $\gamma_n\in\cBPF^{d-1}(\X)$
such that $q_{\om,\gamma_n}(\alpha_n) \ge \normsigma{\alpha_n}{\om}-\epsilon$
which implies for $n$ large enough
\[q_{\om,\gamma_n}(\alpha)\ge 
 q_{\om,\gamma_n}(\alpha_n) - \epsilon 
\ge  \normsigma{\alpha_n}{\om}^2-2 \epsilon
\ge  \normsigma{\alpha}{\om}^2-3 \epsilon
\] concluding the proof.
The second identity~\eqref{eq:bound-norms2} follows from~\eqref{eq:bound-norms} and we obtain that $\alpha\cdot\beta$  belongs to $\numbpfs{2}(\X)$.

\medskip

Take $\alpha \in \cNef^1(\X)$. For any class $\gamma \in \cBPF^{d-2}(\X)$ such that $(\om^2\cdot \gamma) =1$,  we have 
$
q_{\omega, \gamma}(\alpha) \le 2 C_\om^2 (\alpha \cdot \om^{d-1})^2
$
by~\eqref{eq_nef_comp} 
which proves that $\normsigma{\alpha}{\om} \le \sqrt{2} C_\om \normbpf{\alpha}$.
For a general class  $\alpha \in\cnum^1(\X)$, decompose $\alpha=\alpha_+-\alpha_-$ with 
\[\max\{\normbpf{\alpha_+},\normbpf{\alpha_-}\}\le 2\normbpf{\alpha}.\] The previous estimates imply
\begin{equation}\label{eq:SCHSD}
\normsigma{\alpha}{\om} \le C'_\om \normbpf{\alpha},
\end{equation}
where $C'_\om=  2\sqrt{2} C_\om = 2\sqrt{2} \cdot C_d/(\omega^d)>0$.
This implies that the completion of $\cnum^1(\X)$ with respect to  $\normbpf{\cdot}$ injects continuously into $\nums(\X)$.

Let us now show that the canonical homomorphism $\imath_\Sigma \colon \nums(\X)\to\wnum^1(\X)$  is injective.
Pick $\alpha\in\nums(\X)$ such that $\imath_\Sigma(\alpha)=0$, i.e. $(\alpha\cdot\beta)=0$ for all Cartier $b$-class $\beta \in \cnum^{d-1}(\X)$.
By what precedes, we have $\normsigma{\alpha}{\om}= \sup \{  2(\alpha \cdot \om \cdot \gamma)^2 - (\alpha^2\cdot\gamma)\} =  \sup \{ - (\alpha^2\cdot\gamma)\}$ where the supremum is taken over all
$\gamma\in\cBPF^{d-2}(\X)$ such that $(\gamma \cdot \om^2)=1$.

Take any sequence $\alpha_n\in\cnum^1(\X)$ such that $\normsigma{\alpha_n-\alpha}{\om}\to0$.
Then since $(\alpha \cdot \alpha_n \cdot \gamma) = 0$ for all $\gamma \in \cBPF^{d-2}(\X)$ satisfying  $(\gamma\cdot\om^2)=1$, we get
\[
|(\alpha^2\cdot\gamma)|
=
|(\alpha^2\cdot\gamma)- (\alpha\cdot\alpha_n\cdot\gamma)|
\le
3  \normsigma{\alpha}{\om} \normsigma{\alpha -\alpha_n}{\om}
 \to 0\]
 concluding the proof.
\end{proof}

%%%%%%%%%%%%%%%%%%%%%%%%%%%%%%%%%%%%%%

\subsection{Compactness in $\nums(\X)$}\label{sec:compac}

The following theorem is a key ingredient to our approach. 
It is a surprising compactness result in the Banach space $\nums(\X)$ endowed with its norm topology.

\begin{thm} \label{thm_compactness} Suppose that $\alpha_n \in \cnum^1(\X)$ is a sequence satisfying the following two conditions.
\begin{enumerate}
\item The sequence $\normbpfs{(\alpha_n \cdot \om)}$ is bounded.
\item There exists a class $\beta \in \numbpfs{2}(\X)$ such that 
\[\normbpfs{(\alpha_n \cdot \alpha_m) - \beta} \rightarrow  0\]  uniformly in $n,m \rightarrow +\infty$.
\end{enumerate}
Then one can find $\alpha \in \nums(\X)$, and one can extract a subsequence of $\alpha_{n_j}$ 
such that $\normsigma{\alpha_{n_j} - \alpha}{\om} \to0$.
\end{thm}

\begin{proof}
Since $\normbpfs{(\alpha_n \cdot \om)}$ is bounded and since $\alpha_n\cdot \om$ are Cartier classes, we have that $|(\alpha_n\cdot \om \cdot \gamma)| $ is uniformly bounded for all $\gamma \in \BPF^{d-2}(\X)$ satisfying $(\gamma \cdot \om^2)=1$.  Applying Banach-Alaoglu's theorem in the Banach space $(\Vect(\BPF^{d-2}(\X)), \normbpf{\cdot})$, we  extract a subsequence of $(\alpha_{n_j} \cdot \om)$ such that  $(\alpha_{n_j} \cdot \om\cdot\gamma)$ converges for all $\gamma\in\BPF^{d-2}(\X)$ such that $(\gamma \cdot \om^2)=1$. To simplify notation we shall suppose that $(\alpha_{n} \cdot \om\cdot\gamma)$ is converging.

Pick any $\epsilon>0$. We may find $p,q$ and $N$ such that for all $n,m\ge N$, we have
$\normbpfs{(\alpha_n \cdot \alpha_m) - (\alpha_p \cdot \alpha_q)}\le \epsilon$, 
so that  
\begin{equation}\label{eq:76543}
|(\alpha_n \cdot \alpha_m\cdot\gamma) - (\alpha_p \cdot \alpha_q \cdot \gamma)|\le \epsilon (\gamma\cdot\om^2)
\end{equation}
 for all  $\gamma \in \cBPF^{d-2}(\X)$.
Since all classes $\alpha_n, \alpha_m , \alpha_p, \alpha_q$ are Cartier, this upper bound still holds for any   $\gamma \in \BPF^{d-2}(\X)$ by density.

Choose any $\gamma\in\BPF^{d-2}(\X)$ such that $(\om^2\cdot\gamma)=1$.
Then for all $n,m\ge N$, we have:
\begin{multline} \label{eq_interm_compacity}
q_{\omega, \gamma} (\alpha_n - \alpha_m ) 
= 2 ((\alpha_n - \alpha_m) \cdot \omega \cdot \gamma)^2 - ((\alpha_n - \alpha_m)^2 \cdot \gamma)
\\
\le 2 ((\alpha_n - \alpha_m) \cdot \omega \cdot \gamma)^2+ 4\epsilon \ \ \ \ \ \ \ \ \ \ \ \ \ \ \ \ \ \ \ \ \ \ \ \ \ \ \ \ \ \ \ \ \ \  
\end{multline}
which implies $q_{\omega, \gamma} (\alpha_n - \alpha_m ) $ to be a Cauchy sequence. 
This proves $\alpha_n$ converges in $\FS(\X)$ to some class $\alpha$. 

We shall now argue that $\alpha_n$ is a Cauchy sequence inside $\nums(\X)$.
We rely on the following lemma.
\begin{lem} \label{lem_FS_to_nums} 
There exists a constant $C>0$ depending only on $d$, such that:
\begin{equation*}
(\alpha \cdot \omega \cdot \gamma)^2 \le C  \left( q_{\om ,\om^{d-2}}(\alpha ) + \normbpfs{\alpha^2} \right) (\om^2 \cdot \gamma)^2
\end{equation*}
for all $\alpha\in\cnum^1(\X)$
and for any $\gamma \in \BPF^{d-2}(\X)$ such that $(\om^2\cdot\gamma)>0$.
\end{lem}
The lemma applied to $\alpha_n - \alpha_m$ together with \eqref{eq_interm_compacity} implies for all $n,m\ge N$:
\begin{align*}
\normsigma{\alpha_n - \alpha_m}{\om}^2 &= \sup_{\substack{\gamma \in \BPF^{d-2}(\X) \\
(\om^2 \cdot \gamma)= 1}} 2 ((\alpha_n - \alpha_m)\cdot \om \cdot \gamma)^2 - \left ( (\alpha_n - \alpha_m)^2 \cdot \gamma \right ), \\
& \le  4\epsilon +  2 \sup_{\substack{\gamma \in \BPF^{d-2}(\X) \\
(\om^2 \cdot \gamma)= 1}} ((\alpha_n - \alpha_m) \cdot \om \cdot \gamma)^2 ,
\\
&\le   C' \epsilon +  2C q_{\om ,\om^{d-2}}(\alpha_n - \alpha_m)  \le C'' \epsilon,
\end{align*}
 since $q_{\om ,\om^{d-2}}(\alpha_n - \alpha_m)$ is a Cauchy sequence and where $C'= 4+ 4 C$ and $C''= C' + 2 C$. 
This concludes the proof of the theorem.
\end{proof}

\begin{proof}[Proof of Lemma~\ref{lem_FS_to_nums}]
 
By homogeneity we may assume $(\gamma\cdot\om^2)=1$.
Suppose first that $\gamma$ is a strongly basepoint free class (hence Cartier).

By definition, there exists a smooth model $X'$ and a flat morphism $\pi : Y \to X'$ of relative dimension $e$ such that
\begin{equation*}
\gamma = \pi_* (D_1 \cdot \ldots \cdot D_{e + d-2})
\end{equation*}
where $D_1, \ldots , D_{e+d-2} \in \nef^1(Y)$, so that 
$(D_1 \cdot \ldots \cdot D_{e+d-2} \cdot \pi^* \om^2) =1$.
Set
\begin{equation*}
\beta_i = (D_i \cdot \ldots \cdot D_{e+d-2})\in\BPF^{e+d-i-1}(\Y)
\end{equation*}
 so that $\beta_{i} = D_i\cdot \beta_{i+1}$ for each $i =1, \cdots, e+d-3$ where $\Y$ is the collection of models over $Y$. 
We rescale inductively these classes replacing $\beta_{i+1}$ by $t_i^{-1} \beta_{i+1}$  and $D_i$ by $t_i D_i$ with $t_i= (\beta_{i+1} \cdot\pi^* \omega^{i+2})$. More precisely, we rescale $\beta_2$ and $D_1$, then $\beta_3$ and $D_2$, and we stop once we have rescaled $\beta_{d-1}$ and $D_{d-2}$.  
We arrive to the situation where:
\begin{equation} \label{eq_condition_scaling}
 (\beta_i \cdot \pi^* \om^{i+1}) = 1
\end{equation}  
for all $i= 1, \cdots,  e+ d-2$. Note that this condition is automatic for $i=1$.

By Proposition \ref{prop_surfaces_continuity_intersection} we have:
\begin{align*}
(\alpha \cdot \omega \cdot \gamma)^2 &= |\pi^* (\alpha \cdot \om) \cdot D_1 \cdot \ldots \cdot D_{e+d-2}|^2, \\
 & \le 9 q_{\pi^*\omega, \beta_2 \cdot \pi^*\omega}(\pi^*\alpha ) q_{\pi^*\om, \beta_2 \cdot \pi^*\om} (D_1), \\
& \le 18 q_{\pi^*\om, \beta_2\cdot \pi^*\om}(\pi^*\alpha ) (D_1\cdot \pi^*\om^2\cdot \beta_2)^2
\\& \le 18 q_{\pi^*\om, \beta_2\cdot \pi^*\om}(\pi^*\alpha  ).
\end{align*}
Let $\epsilon := \normbpfs{\alpha^2}$.
We proceed estimating the term $q_{\pi^*\om, \beta_2\cdot \pi^* \om}(\pi^*\alpha)$. 
To do so, we write:
\begin{align*}
q_{\pi^*\om, \beta_2\cdot \pi^*\om}(\pi^*&\alpha ) 
\\ &= 2 ( \pi^*(\alpha \cdot \om) \cdot \beta_2 \cdot \pi^*\om)^2
- (\pi^*\alpha^2 \cdot \beta_2 \cdot \pi^* \om) \\
 & 
 \leqslant
 2 (\pi^*\alpha \cdot D_2 \cdot  \pi^*\om^2  \cdot \beta_3)^2 + \epsilon , \\
 & 
 \mathop{\leqslant}\limits^{\text{Prop.\ref{prop_surfaces_continuity_intersection}}} 
  18 q_{\pi^*\om , \beta_3 \cdot \pi^*\om^2}(\pi^*\alpha) q_{\om , \beta_3 \cdot \pi^*\om^2}(D_2) + \epsilon, \\
 & \le 36 q_{\pi^*\om , \beta_3 \cdot \pi^*\om^2}(\pi^*\alpha) + \epsilon,
\end{align*}
where the first and the third inequalities follow from the fact that $(\beta_2 \cdot \pi^* \om^3) = 1$.
An immediate induction yields
\begin{align*}
(\alpha \cdot \omega \cdot \gamma)^2 &\le  
18 q_{\pi^*\om, \beta_2\cdot \pi^*\om}(\pi^*\alpha )
\\
& \le 36\cdot 18\, q_{\pi^*\om , \beta_3 \cdot \pi^*\om^2}(\pi^*\alpha) +  18\, \epsilon
\\
& \le 36^{d-3}\cdot18\, q_{\pi^*\om ,\beta_{d-1} \pi^*\om^{d-2}} (\pi^*\alpha) +  18\cdot  36^{d-3}\, \epsilon.
\end{align*}
Now observe that $\pi_* \beta_{d-1}$ is a class in $\num^0(X')$, hence it is a multiple of the fundamental class $[X']$. 
The projection formula yields
$1=(\beta_{d-1}\cdot \pi^*\om^{d}) = (\pi_* \beta_{d-1}\cdot \om^d)$
and
$ q_{\pi^*\om ,\beta_{d-1} \pi^*\om^{d-2}} (\pi^*\alpha)  = q_{\om, \om^{d-2}}(\alpha)$ (recall the normalization $(\om^d)=1$), hence:
\begin{equation*}
(\alpha \cdot \omega \cdot \gamma)^2
 \le C(d) \left(q_{\om, \om^{d-2}}(\alpha) + \epsilon\right), 
\end{equation*}
where $C(d)=36^{d-3}\cdot 18$, which concludes the proof when $\gamma \in \cBPF^{d-2}(\X)$ is strongly basepoint free.

\smallskip

Take any finite family of strongly basepoint free classes $\gamma_1, \cdots, \gamma_k$, and set $\gamma = \sum_1^k\gamma_i$.
By linearity, we have:
\begin{align*}
|(\alpha \cdot \om \cdot \gamma)| &\le \sum_1^k |(\alpha \cdot \om \cdot \gamma_i)|, \\
 &\le \sum_1^k C(d)^{1/2} (\om^2 \cdot \gamma_i)  \left(q_{\om , \om^{d-2}}(\alpha) + \epsilon\right)^{1/2}, \\
 & \le  C(d)^{1/2}\, (\om^2 \cdot \gamma) \left(q_{\om , \om^{d-2}}(\alpha) + \epsilon\right)^{1/2}.
\end{align*} 
Finally pick any class $\gamma \in \BPF^{d-2}(\X)$. Then one can find a sequence of strongly basepoint free classes $\gamma^{(n)}_1, \cdots, \gamma^{(n)}_{k_n}$
such that $\gamma^{(n)} := \sum_1^{k_n}\gamma^{(n)}_i$ converges weakly to $\gamma$. 
We have
\[
|(\alpha  \cdot \om \cdot \gamma^{(n)})| 
\le  C(d)^{1/2}\, (\om^2 \cdot \gamma^{(n)}) \left(q_{\om , \om^{d-2}}(\alpha ) + \epsilon\right)^{1/2}.
\]
Since the two classes $(\alpha\cdot \om)$, $\om^2$ are Cartier, we may let $n\to \infty$ and the weak convergence $\gamma^{(n)}\to\gamma$ yields the required estimate.
\end{proof}

%%%%%%%%%%%%%%%%%%%%%%%%%%%%%%%%%%%%%%%%%%%%%%%%%%%%%%%%%%%%%%%%

\subsection{Restriction operator}\label{sec:restrict}

This section is technical in nature. It will be used to transfer the Hodge index theorem from $L^2$--classes on surfaces to $\FS(\X)$.
As before, we suppose that $\K$ is algebraically closed, countable and of characteristic $0$, and we fix some
extension $\K'/\K$ which is countable, algebraically closed and has infinite transcendence degree over $K$.
Recall that for any smooth algebraic variety $X$ defined over $\K$ the canonical homomorphism $\num^1(X)\to\num^1(X_{\K'})$
is an isomorphism, see e.g.~\cite[Proposition~3.2]{maulik_poonen}.

\medskip
We begin with some general terminology.

Given any generically finite morphism $\pi\colon Y\to Z$ between equidimensional algebraic varieties, we let $\Crit(\pi)$ be the locus where $\pi$ is not locally \'etale. When $\pi$ is proper, the set of critical values $\CV(\pi) = \pi(\Crit(\pi))$ is an algebraic sub-variety. When $Y$ is smooth, then $\Crit(\pi)$ is a Cartier divisor; when $\pi$ is birational, 
then $\CV(\pi)$ has codimension at least $2$ in $Z$.

Two irreducible subvarieties $Y$ and $Z$ of an algebraic variety intersect properly when 
the codimension of any irreducible component of $Y\cap Z$ is equal to $\codim(Y)+\codim(Z)$.
Two subvarieties $Y$ and $Z$ intersect properly when all their irreducible components intersect properly (see e.g the definition \cite[p.7]{fulton}).

\medskip

Take any class $\sigma\in\cnum^{d-2}(\X)$
 which is determined in some model $X_1$ as the pushforward under a \emph{smooth} morphism of a complete intersection $\sigma_Y$ of very ample divisors on a smooth projective variety $Y_1$. 

If $X_2$ is any other smooth model of $X$ dominating $X_1$, we
 can consider the fibered product $Y_2=Y_1 \times_{X_1} X_2 $ and one gets the following diagram. 
\begin{equation*}
\xymatrix{Y_2 \ar[d]^{\varpi} \ar[r]^{p_2}& X_2 \ar[d]^{\pi} \\
Y_1 \ar[r]^{p_1} & X_1,}
\end{equation*}
 where the horizontal arrows are smooth morphisms and the vertical arrows are birational. Observe that by construction, $Y_2$ is also a smooth variety,~\cite[p.268]{hartshorne}.
 
Since the transcendence degree of $K'$ over $K$ is infinite, one can find a smooth surface $S\subset Y_{1,K'}$ defined over $K'$
whose class in $\num^{d+e-2}(Y_{1,K'})$ is equal to $\sigma_Y$, and intersects properly all subvarieties of $Y_1$ defined over $K$. 
In particular, it intersects properly  $\CV(\varpi)$ for \emph{all} proper 
birational morphisms $\varpi\colon Y_2\to Y_1$
defined over $K$. 
 Let $\Ss$ be the Riemann-Zariski space associated to $S$.
Observe that the incarnation of $\sigma$ in $Y_2$ is given by the strict transform of $S$ by $\varpi$, and this incarnation is smooth.
\smallskip

We define a restriction linear operator $r_{S}\colon \cnum^1(\X_\K)\to \cnum^1(\Ss_{\K'})$ as follows. 
Pick any element of $\cnum^1(\X_\K)$ determined in a model $\pi\colon X_2\to X_1$ by a class $\alpha \in \num^1(X_{2,\K})$. 
Consider the smooth model $Y_2$ of $Y_1$ and set $\varpi: Y_2 \to Y_1$ the corresponding birational morphism and $p_2 : Y_2 \to X_2$ the corresponding smooth morphism.
Since $S$ intersects  $\CV(\varpi)$ properly, the strict transform $S_{Y_2}$ of $S$ inside $Y_{2,\K'}$ represents the class $p_2^*[S]$. 
Then we let $r_S([\alpha])$ be the Cartier $b$-class determined by the image of $p_2^*\alpha_{| S_{Y_2}}$ in $\num^1(Y_{2,\K'})$.

Note that if $ X_3 \to X_2$ is another model, then we have the following commutative diagram
\begin{equation} \label{diagram_restriction}
\xymatrix{ 
 S_{Y_3} \ar[d]^{\varpi_{S,3}} \ar@{^{(}->}[r]& Y_{3,K'} \ar[r]^{p_3} \ar[d]^{\varpi_3} & X_{3,K'} \ar[d]^{} \\
S_{Y_2}
\ar@{^{(}->}[r]
\ar[d]^{\varpi_{S,2}}
&
Y_{2,\K'}
\ar[d]^{\varpi_2}
\ar[r]^{p_2} & X_{2, K'} \ar[d]^{} \\
S_{Y_1}
\ar@{^{(}->}[r]&
Y_{1,\K'} \ar[r]^{p_1} & X_{1,K'}
}
\end{equation}
and
$\varpi_3^*(p_2^*\alpha)|_{S_{Y_3}}= (\varpi_{S,3})^* (p_2^*\alpha|_{S_{Y_2}})$, 
so that $r_S([\alpha])$ does not depend on the choice of models.

\begin{prop}
Let $\om$ be any big and nef class in $\num^1(X_2)$, and write $\om_S :=p_2^*\om|_{S_{Y_2}}\in\num^1(S_{Y_2})$ using the same notation as in the above discussion.

The restriction operator  $r_{S}\colon \cnum^1(\X_\K)\to \cnum^1(\Ss_{\K'})$ preserves nef classes, 
and $\normsigma{r_S(\alpha)}{\om_S} \le \sqrt{(\om^2\cdot\sigma)}\, \normsigma{\alpha}{\om}$ for any Cartier class $\alpha\in \cnum^1(\X_{\K})$. 

In particular $r_S$ extends to a continuous operator $r_{S}\colon \FS(\X_\K)\to L^2(\Ss_{\K'})$,
which satisfies
\begin{equation}\label{eq-compat-restr}
(\alpha\cdot\alpha'\cdot\sigma)= (r_S(\alpha)\cdot r_S(\alpha'))
\end{equation}
for any $\alpha, \alpha'\in\FS(\X_\K)$.
\end{prop}

\begin{rem} Since the class $\sigma$ is strongly basepoint free, the morphism $r_S$ preserves pseudo-effective classes since for each $i$ the pushforward map induced by the composite morphism $S_{Y_i} \to Y_{i,K'} \to X_{2,K'}$ preserves pseudo-effectivity. 
\end{rem}

\begin{proof}
Without loss of generality, we may suppose $\alpha$ is determined by a class $\beta\in \num^1(X_3)$ where $X_3$ is a model dominating $X_2$. 
Then $r_S(\alpha)$ is the Cartier $b$-class determined by $\beta_{S} := p_3^* \beta|_{S_{Y_3}}$ using the notation of diagram \eqref{diagram_restriction}, and
\[
\normsigma{r_S(\alpha)}{\om_S}^2
= 
2 \frac{(\beta_S\cdot\om_S)^2}{(\om_S^2)} - (\beta^2_S)
=
2 \frac{(\beta\cdot\om\cdot \sigma)^2}{(\om^2\cdot\sigma)} - (\beta^2\cdot\sigma)
\le
(\om^2\cdot\sigma)\, \normsigma{\alpha}{\om}^2
\]
as required. Note that~\eqref{eq-compat-restr} is clear whenever $\alpha,\beta\in\cnum^1(\X)$ and the result follows by density.
\end{proof}

Pick $\alpha\in\wnum^1(\Ss)$. For each smooth model $\pi\colon X_3\to X$ (defined over $\K$), 
we may consider the successive images: 
\[\alpha \in \wnum^1(\Ss) \mapsto \jmath_* \alpha \in \wnum^{d-1}(\Y_{K'}) \mapsto  p_* \jmath_* \alpha \in \wnum^{d-1}(\X_{K'})\] 
under 
 the pushforward by the injection morphism
$\jmath\colon S_{Y_3} \to Y_{3,\K'}$ and by $p_3 : Y_{3,K'} \to X_{3,K'}$.
Since $Y_{3,K'}$ and $S_{Y_3}$ are both smooth (see \S \ref{section_numerical}), this defines a class $p_*\jmath_*(\alpha)\in \num^{d-1}(X_{3,\K'})$, hence a class  $p_*\jmath_*(\alpha)\in \num^{d-1}(X_3)$ since $\num^{d-1}(X_3)\to\num^{d-1}(X_{3,\K'})$ is an isomorphism (by Poincar\'e duality, see the discussion above).

We thus obtain a natural operator $\jmath_S\colon \wnum^1(\Ss)\to \wnum^{d-1}(\X)$
which makes the following diagram commutative:
\[
\xymatrix{ 
\nums (\X)
\ar[r]^{r_S}
\ar[dr]^{\cdot \sigma}
&
\nums (\Ss)
\ar[d]^{\jmath_S}
\\
&\wnum^{d-1}(\X)
}
\]

%%%%%%%%%%%%%%%%%%%%%%%%%%%%%%%%%%%%%%
\subsection{Hodge-index theorem on $\FS(\X)$}

The  product on $\FS(\X)$ naturally extends  to a continuous product on $\FS(\X)_\C:=\FS(\X)\otimes_\R\C$ and we have the following generalization of Hodge index theorem. 

\begin{thm}\label{thm_hodge_index} 
Let $\alpha, \beta \in \FS(\X)_\C$  be two classes such that 
$(\alpha \cdot \bar\alpha\cdot \gamma)\ge0$, $(\beta\cdot \bar\beta \cdot \gamma) \ge 0$ for all $\gamma\in \cBPF^{d-2}(\X)$.
If $\alpha\cdot \bar\beta=0$ in $\wnum^{2}(\X)$, then 
$\imath_\Sigma(\alpha)$ and $\imath_\Sigma(\beta)$ are proportional in  $\wnum^1(\X)$.
\end{thm}

Observe that the positivity conditions $(\alpha \cdot \bar\alpha\cdot \gamma)\ge0$, $(\beta\cdot \bar\beta \cdot \gamma) \ge 0$ are satisfied as soon as $\alpha$ and $\beta$ are real nef classes.

When both classes belong to $\nums(\X)$, the previous theorem takes the following form:
\begin{cor}\label{cor_hodge_index} 
Let $\alpha, \beta \in \nums(\X)_\C$  be two classes such that 
$(\alpha \cdot \bar\alpha\cdot \gamma)\ge0$, $(\beta\cdot \bar\beta \cdot \gamma) \ge 0$ for all $\gamma\in \cBPF^{d-2}(\X)$.
If $\alpha\cdot \bar\beta=0$ in $\wnum^{2}(\X)$, then 
$\alpha$ and $\beta$ are proportional.
\end{cor}

\begin{proof}
Pick any two classes $\alpha , \beta \in \FS(\X)_\C$ satisfying the above conditions.
Pick any BPF class $\sigma\in\cnum^{d-2}(\X)$  such that
 $\sigma= p_*(D_1 \cdot \ldots \cdot D_{e+d-2})$ where $D_1, \ldots , D_{e+d-2}$ are very ample divisors on a smooth
projective variety $Y$ equipped with a projective smooth morphism $p\colon Y\to X'$ of relative dimension $e$. Write $\theta= (D_1 \cdot \ldots \cdot D_{e+d-2})$ so that
$\sigma=p_*(\theta)$. As in the previous section,  we may choose a smooth surface $S\subset Y$ defined over an extension of $\K$, of infinite transcendence degree, whose fundamental class represents $\theta$. Using the notation of \S\ref{sec:restrict} and~\eqref{eq-compat-restr}, we have:
\[
0= (\alpha\cdot \bar\beta \cdot \sigma) =  (p^*(\alpha)\cdot p^*(\bar\beta) \cdot \theta)= (r_S(p^*(\alpha))\cdot r_S(p^*(\bar\beta)))~.
\]
Since by assumption, $(r_S(p^*(\alpha))\cdot r_S(p^*(\bar \alpha)) )=(r_S(p^*(\beta))\cdot r_S(p^*(\bar \beta)))\ge0$,
Hodge index theorem on $L^2(\Ss)$ implies that $r_S(\alpha)$ and $r_S(\beta)$ are proportional. Taking the images of $r_S(\alpha)$ and $r_{S}(\beta)$ by
$\jmath_S$, we get $t(\theta)\ge 0$ such that $(p^*(\alpha)\cdot\theta)= t(\theta)\, (p^*(\beta)\cdot\theta)$ whenever $(\beta\cdot\sigma)\neq0$.

\smallskip

We claim that $t(\theta)$ does not depend on the choice of $\theta$. 
Observe that linear combinations of Chern classes of vector bundles generate $\cnum^{d-2}(\X)$, hence
the set of classes $\sigma=p_*(\theta)$  satisfying our conditions  generate $\cnum^{d-2}(\X)$ too. And
we conclude that $\imath_\Sigma(\alpha)$ and $\imath_\Sigma(\beta)$ are proportional in $\wnum^1(\X)$.

It thus remains to prove the claim. 
Pick first any other very ample divisor $D'_1$ on $Y$, and write
$\theta' =(D'_1 \cdot D_2 \cdot \ldots \cdot D_{e+d-2})$, $\sigma'=p_*(\theta')$.
Suppose moreover that $(p^*(\beta)\cdot\theta)$ and $(p^*(\beta)\cdot\theta')$ are both non zero.
Observe that $( D_1' \cdot \theta) = ( D_1 \cdot \theta')$ so that
\[
t(\theta) (p^*(\beta) \cdot \theta\cdot D_1')= (p^*(\alpha) \cdot \theta\cdot D_1') =  (p^*(\alpha) \cdot \theta' \cdot D_1) = t(\theta') (p^*(\beta) \cdot \theta\cdot D_1').
\]
Since $D_1'$ is ample in $Y$, we have $((D_1')^2 \cdot \theta) >0$. Recall that by assumption, 
$(p^*(\beta)\cdot p^*(\bar{\beta}) \cdot \theta)\ge 0$, so that the Hodge index theorem on $L^2(\Ss)$
implies $(p^*(\beta) \cdot \theta\cdot D_1') \neq 0$. Dividing the above displayed equation by $(p^*(\beta) \cdot \theta\cdot D_1')$ yields
$t(\theta) =t(\theta')$ in this case. 

Iterating this argument implies $t(\theta)=t(\theta')$ whenever $\theta$ and $\theta'$ are push-forwards of complete intersection classes in $Y$.
Finally suppose $\theta'$ is a complete intersection class in some 
smooth projective variety $Y'$ endowed with a flat morphism $p' : Y' \to X'$ of relative dimension $e'$. Then the fibered product $Y'\times_{p'\times p} Y$  defines a smooth projective variety, flat over $X'$ and any complete intersection on $Y$ or on $Y'$ can be pulled back to $Y'\times_{p'\times p} Y$. We are thus reduced to the previous situation and $t(\theta)=t(\theta')$ as well.
\end{proof}

%%%%%%%%%%%%%%%%%%%%%%%%%%%%%%%%%%%%%%%%%%%%

\section{Spectral interpretation of dynamical degrees}
Let $X$ and $Y$ be any two normal projective varieties of dimension $d$ defined over a countable, algebraically closed field $K$ of characteristic $0$.
Let $f \colon  X \dashrightarrow Y$ be any dominant (hence generically finite) rational map.
Denote by $\X, \Y$ the Riemann-Zariski spaces of $X$ and $Y$ respectively.

\subsection{Action of rational maps on $b$-cycles}
Let $\alpha\in\wnum^k(\X)$ and pick any smooth model $Y'$ over $Y$.
Then there exists a model $X'$ over $X$ such that the map $f_{X'/Y'}\colon X'\to Y'$ induced by $f$ is regular. 
We set $(f_*\alpha)_{Y'} := (f_{X'/Y'})_* \alpha_{X'}$. This definition does not depend on the choice of $X'$, and
yields a linear map $f_*\colon \wnum^k(\X) \to \wnum^k(\Y)$, which is continuous for the weak topology, and preserves
pseudo-effective classes.

\smallskip

Suppose now that $\beta\in\cnum^k(\Y)$ is determined in some
 model $Y'$ over $Y$. As above, choose any model $X'$ over $X$ such that the map $f_{X'/Y'}\colon X'\to Y'$ induced by $f$ is regular.
We let $f^*\beta$ be the Cartier $b$-class determined by $(f_{X'/Y'})^*(\beta_{Y'})$. Again this definition does not depend on the choice of $Y'$, 
and yields a linear map $f^*\colon \cnum^k(\Y) \to \cnum^k(\X)$ satisfying
$f^*(\cBPF^k (\Y))\subset\cBPF^k (\X)$ since basepoint free classes are preserved by pullback.

\begin{thm}\label{thm:extension}
For all $k \leq \dim X$, the following property holds:
\begin{equation}
f_* ( \cBPF^k(\X)) \subset \cBPF^k(\Y). \label{eq:f*BPF}
\end{equation}
In particular, we have $f_* ( \cnum^k(\X)) \subset \cnum^k(\Y)$, and 
$f_*(\BPF^k (\X))\subset\BPF^k (\Y)$.
 \end{thm}
 
Theorem \ref{thm:extension} relies on the following flattening theorem which is a 
particular case of Gruson-Raynaud's theorem.
Recall that by generic flatness (see Proposition 29.27.1 of the StacksProject), for any morphism $f \colon X \to Y$ between algebraic varieties, there exists a Zariski dense open subset $U\subset Y$
such that the restriction of $f\colon f^{-1} (U) \to U$ is flat. 

\begin{thm} \label{thm_raynaud_light} Let $f \colon X \to Y$ be a dominant morphism of projective varieties and a Zariski dense open subset $U\subset Y$
such that $f\colon f^{-1} (U) \to U$ is flat. 

Then there exist a  projective variety $Y'$ and a blowing-up $\pi\colon Y' \to Y$ with center in $Y\setminus U$ 
such that the map $X' \to Y'$ is flat, where $X'$ denotes the Zariski closure in the fibered product $X \times_{Y} Y'$ 
of  $f^{-1}(U) \times_{U} \pi^{-1}(U)$. 
\end{thm}
Note that the previous result yields a commutative diagram 
\[
\xymatrixcolsep{5pc}\xymatrix{
X' \ar@{~>}[r] \ar[d]& Y'\ar[d]
\\
X \ar[r]^f & Y}
\]
such that $X' \to Y'$ is a flat morphism, and $Y' \to Y$, $X'\to X$ are proper birational morphisms.

 We refer to~\cite[Theorem~1 on p.26]{MR302645} for a  short proof of this result, or to~\cite{MR393556} for a proof in the complex case; see also~\cite{guignard} for a recent account on Gruson-Raynaud's flattening theorem.

\begin{proof}[Proof of Theorem~\ref{thm:extension}]
Grant~\eqref{eq:f*BPF}. Since $\cBPF^k(\X)$ generates $\cnum^k(\X)$, \\
 $f_* ( \cnum^k(\X)) \subset \cnum^k(\Y)$ follows
and we get $f_*(\BPF^k (\X))\subset\BPF^k (\Y)$ by continuity.

It thus remains to prove~\eqref{eq:f*BPF}. 
To that end, we first fix a model $X'$ over $X$ such that the map $f_{X'/Y}$ is regular.

By Theorem \ref{thm_raynaud_light}, there exist  projective varieties $X'', Y'$ and proper birational morphisms $Y' \to Y$, $X'' \to X'$ such that 
$f_{X''/Y'}\colon  X'' \to Y'$ is flat and the next diagram is commutative:
\[
\xymatrixcolsep{5pc}\xymatrix{
X'' \ar@{~>}[r]^{f_{X''/Y'}} \ar[d]& Y'\ar[d]
\\
X' \ar[r]^{f_{X'/Y}} & Y}
\]
Since flatness is preserved by base change, we may replace $Y'$ by any of its resolution of singularities and
thus assume that $Y'$ is smooth.

We also fix any birational proper morphism $\pi'\colon Y''\to Y'$ such that ${Y''}$ is smooth. By base change we obtain a birational proper morphism $q\colon\tilde{X}\to X''$
and a proper flat morphism $f_{\tilde{X}/\tilde{Y}}\colon \tilde{X}\to Y''$. 
Choose a smooth model $X^{(3)}$ which dominates $\tilde{X}$ so that  the following  diagram
is commutative (flat maps are indicated by curly arrows). 
\[
\xymatrixcolsep{5pc}\xymatrix{ U^{(3)} \ar@{~>}[r]^{s_3} \ar[dd] & X^{(3)} \ar[d]^{q'} \ar@/^1pc/[rdd]^{f_{X^{(3)}/Y''}} & \\
&
\tilde{X} \ar@{~>}[rd]^{f_{\tilde{X}/Y''}} 
\ar[d]_q
&
\\ 
U''
\ar@{~>}[r]^{s''}
\ar[dd]
&
X''
\ar@{~>}[dr]^{f_{X''/Y'}}
\ar[d]
\ar[d]^{} 
&
Y''
\ar[d]^\pi
\\
&
X'
\ar[d]
\ar[dr]^{f_{X'/Y}}
&
Y'
\ar[d]
\\
U
\ar@{~>}[r]^s
\ar[d]^p
&
X
\ar@{.>}[r]^f
&
Y
\\
W
&
&
}
\]
Take any BPF Cartier class $\alpha$ determined in $X$. We shall first prove that $f_*\alpha$ is Cartier and determined in $Y'$, that 
is $(f_*\alpha)_{Y''} = \pi^*((f_*\alpha)_{Y'})$.

By density and since strongly basepoint free classes generate $\cnum^k(\X)$, we  are reduced to proving the result  when the class $\alpha$ is also strongly basepoint free.
Using the definition of strongly basepoint free classes from \cite[Definition 5.1]{MATH06741284}, one can thus find an equidimensional quasi-projecti\-ve scheme $U$ of finite type, a flat morphism $s\colon U \to X$ 
and a proper morphism $p\colon U \to W$ of relative dimension $n-k$ to a quasi-projective scheme
$W$ such that each irreducible component of $U$ surjects onto $W$, and
$\alpha_X = (s|_{F_w})_* [F_w]$ where $[F_w]$ is the fundamental class of a  fiber of $p$ over a general point $w \in W$.

Base change yields two flat maps $s''\colon U'' \to X'', s_3 \colon U^{(3)} \to X^{(3)}$ and  proper morphisms $p''\colon U'' \to W$, $p_3 \colon U^{(3)} \to W$.

Recall that a cycle $Z$ properly intersects  an algebraic subvariety $A$ if for any 
irreducible components $Z_i$ and $A_j$ of $Z$ and $A$ respectively, we have $\codim(Z_i\cap A_j)=\codim(Z_i)+ \codim(A_j)$.

For any generic point $w \in W$, denote by $Z_w$ the cycle in $X''$ given as the pushforward by $s''$  of the fundamental class of a fiber  $p''^{-1}(w)$. 

By~\cite[Remark~5.2]{MATH06741284}, for a generic element $w$ in $W$, the cycles 
$Z_w$ and $f_{X''/Y'} (Z_w)$
properly intersect $\CV(q \circ q')$ and $\CV(\pi)$ respectively
so that by~\cite[Lemma~4.2]{tuyen2}, $\alpha_{X^{(3)}}$  is represented by $ (q \circ q')^{-1}(Z_w)$
and $\pi^*(f_*\alpha)_{Y'}$ by $\pi^{-1}(f_{X''/Y'} (Z_w))$.
Now we have
$(f_{X^{(3)}/Y''})_*((q\circ q')^{-1}(Z_w)) = \pi^{-1}(f_{X''/Y'}(Z_w))$
which implies $(f_*\alpha)_{Y''} = \pi^*(f_*\alpha)_{Y'}$ as required.

\smallskip

Finally $(f_*\alpha)_{Y'}$ is the image under the flat map $f_{X''/Y} \circ s''$ 
of the fundamental class of a fiber of $U''\to W$. It is thus strongly basepoint free.
This concludes the proof of~\eqref{eq:f*BPF}.
\end{proof}

Given $\beta\in\wnum^k(\Y)$, we may define $f^*\beta\in  \wnum^k(\X)$ by duality so that $(f^*\beta\cdot\alpha)=(\beta\cdot f_*\alpha)$
for any $\alpha\in\cnum^{d-k}(\X)$. Then $f^*\colon\wnum^k(\Y) \to \wnum^k(\X)$ is a linear map, which coincides with the pull-back on Cartier $b$-classes, 
and $f^*(\BPF^k (\Y))\subset\BPF^k (\X)$ by continuity. 
By continuity and density of $\Vect(\BPF^k(\Y)),\Vect(\BPF^k(\X))$ in $\wnum^k(\Y)$ and $\wnum^k(\X)$ respectively, we obtain:
 \begin{cor}\label{cor:thm-extension}
The pull-back morphism extends to a linear map $f^*\colon\wnum^k(\Y) \to \wnum^k(\X)$ which is continuous
for the weak topology and satisfies $f^*(\BPF^k (\Y))\subset\BPF^k (\X)$, and we have 
\begin{equation}\label{eq:proj-form}
f_*(\alpha\cdot f^*\beta) =(f_*\alpha\cdot \beta)
\end{equation}
for any $b$-classes $\alpha,\beta$ such that one of them is Cartier.

Moreover, we have \[(f^*\alpha\cdot f^*\beta) = f^*(\alpha\cdot \beta) = \deg(f) (\alpha\cdot \beta)\] for any
$\alpha\in\numbpf{k}(\Y)$, $\beta\in\numbpf{d-k}(\Y)$, where $1\le \deg(f) = [K(X):K(Y)]$ is the topological degree of $f$.
\end{cor}

\subsection{Pullback and pushforward morphisms on $\numbpf{\bullet}(\X)$}

Let $\om_X$ and $\om_Y$ be two big and nef Cartier $b$-divisors on $X$ and $Y$ respectively. 
Recall that the $k$-degree of $f$ with respect to these polarizations is by definiton the quantity: 
\[
\deg_k(f):=(f^* \om_Y^k \cdot \om_X^{d-k}) >0.
\]

\begin{prop} \label{prop:pullBPF}
The pullback morphism $f^*\colon \cnum^k(\Y) \to \cnum^k(\X)$ extends as a linear operator $f^*\colon  \numbpf{k}(\Y) \to \numbpf{k}(\X)$
which satisfies:
\begin{equation}\label{eq:bdnormbpf}
\frac{\deg_k(f)}{(\om_Y^d)}
\le
\normbpf{f^*}
\le
C\,\frac{\deg_k(f)}{(\om_Y^d)}
\end{equation}
for some constant $C>0$ depending only on $d$.
\end{prop}

Note that since $f^*\colon \cnum^\bullet(\Y) \to \cnum^\bullet(\X)$ is a graded ring homomorphism, its extension to $f^*\colon  \numbpf{\bullet}(\Y) \to \numbpf{\bullet}(\X)$
is a Banach ring homomorphism.

\begin{proof} 
We prove~\eqref{eq:bdnormbpf}. The lower bound is clear since
\[
\normbpf{f^*} \ge \frac{\normbpf{f^*\om_Y^k}}{\normbpf{\om_Y^k}} = \frac{\deg_k(f)}{(\om_Y^d)}
\]
For the upper bound, we pick any class $\alpha \in \numbpf{k}(\Y)$, and write $\alpha=\alpha_+ - \alpha_-$ 
with $\alpha_+, \alpha_-\in \BPF^k(\Y)$ such that:
\begin{equation*}
\normbpf{\alpha}\le (\alpha_+ \cdot \om_Y^{d-k}) + (\alpha_- \cdot \om_Y^{d-k}) \le \normbpf{\alpha} + \epsilon
\end{equation*}
for some $\epsilon>0$. Since $f^* \alpha_\pm$ are BPF, we obtain using Siu's inequalities~\eqref{eq:siu_bpf} applied to $f^* \alpha_{\pm}$ and $f^*\omega^k$:
\begin{align*}
\normbpf{f^*\alpha}
&\le 
(f^*\alpha_+ \cdot \om_X^{d-k}) + (f^*\alpha_- \cdot \om_X^{d-k})
\\
& \mathop{\le}\limits^{\eqref{eq:siu_bpf}}
C_d (\normbpf{\alpha_+}+\normbpf{\alpha_-})
\frac{(f^*\om_Y^k \cdot \om_X^{d-k})}{(\om_Y^d)}
\\
&\le 
C_d (\normbpf{\alpha}+\epsilon) \frac{\deg_k(f)}{(\om_Y^d)}.
\end{align*}
We conclude letting $\epsilon\to0$.
\end{proof}

\begin{prop} \label{prop:pushBPF}
The pushforward morphism $f_*\colon \cnum^k(\X) \to \cnum^k(\Y)$ extends as a linear operator $f_*\colon  \numbpf{k}(\X) \to \numbpf{k}(\Y)$
which satisfies:
\begin{equation}\label{eq:bdnormbpf*}
\frac{\deg_{d-k}(f)}{(\om_X^d)}
\le
\normbpf{f_*}
\le
C\,\frac{\deg_{d-k}(f)}{(\om_X^d)}
\end{equation}
for some constant $C>0$ depending only on $d$.
\end{prop}

Using $(f^*\alpha\cdot \beta)= (\alpha\cdot f_*\beta)$, the proof is identical to the previous one and left to the reader.

\subsection{Pullback and pushforward morphisms on $\numbpfs{\bullet}(\X)$}

\begin{prop} \label{prop_pullback_bpfs}
For any integer $k\le d$,  the linear maps $f^*  : \cnum^k(\Y) \to \cnum^k(\X)$ and $f_*: \cnum^k(\X) \to \cnum^k(\Y)$ extend as  bounded linear operators $f^* :  \numbpfs{k}(\Y) \to \numbpfs{k}(\X)$ and $f_* : \numbpfs{k}(\X) \to \numbpfs{k}(\Y)$ respectively. Moreover,
\begin{equation*}
\dfrac{\deg_k(f)}{(\omega_Y^d)} \le \normbpfs{f^*} \le C \dfrac{ \deg_{k}(f)}{(\omega_Y^d)}, 
\end{equation*}
and 
\begin{equation*}
\dfrac{\deg_{d-k}(f)}{(\omega_X^d)} \le \normbpfs{f_*} \le C  \dfrac{\deg_{d-k}(f)}{(\omega_Y^d)}, 
\end{equation*}
for some constant $C>0$ depending only on $d$.
\end{prop}

The proof uses Siu's  inequalities~\eqref{eq:siu_bpf} and the projection formula and is left to the reader.

\subsection{Pullback and pushforward morphisms on $\nums(\X)$}

Let $f:X \dashrightarrow Y$ be any dominant rational map between two normal and projective varieties. We suppose that the dimension of $Y$ is $d\ge 2$ but we may have $\dim(X)\neq \dim(Y)$. We fix big and nef classes $\om_X, \om_Y$ on $X$ and $Y$ respectively and write $\deg_1(f) := (f^*\om_Y\cdot \om_X^{\delta -1})$ where $\delta = \dim(X)$.

We shall denote by $\normsigma{L}{\omega}$ the norm of  a bounded linear operator $L : (\nums(\Y), \normsigma{\cdot}{\omega_Y} ) \to (\nums(\X),\normsigma{\cdot}{\omega_X})$.

\begin{prop} \label{prop_continuity_pullback} 
The  linear map  $f^*\colon \cnum^1(\Y) \to \cnum^1(\X)$ extends as a bounded linear operator
$f^*\colon  \nums(\Y) \to \nums(\X)$
that satisfies:
\begin{equation}
\label{eq:bdnormspull}
\frac{\deg_1(f)}{(\om_X^\delta)}
\le
\normsigma{f^*}{\om}
\le
C
\frac{\deg_1(f)}{(\om_X^\delta)}
\end{equation}
where  $C>0$ depends only on $\delta$. Moreover, we have
$f^*(\alpha\cdot\beta)=(f^*\alpha\cdot f^*\beta)\in\numbpfs{2}(\X)$ for all $\alpha,\beta\in\nums(\X)$.
\end{prop}

\begin{rem}
Although we shall not use it, one can check that $f^*$ also induces a continuous operator on  the Fr\'echet space $\FS(\X)$.
\end{rem}

\begin{proof}
Normalize $\om_X$ so that $(\om_X^\delta)=1$.
Recall from \eqref{eq:pos-nums}  that:
\begin{equation*}
\normsigma{f^*\om_Y}{\om_X}^2 = \sup_{\substack{\gamma \in \cBPF^{\delta-2}(\X)\\
(\gamma \cdot \om_X^2)=1}}  (f^* \om_Y \cdot \om_X \cdot \gamma)^2 - \left ( \left (f^*\om_Y - (f^* \om_Y \cdot \om_X \cdot \gamma) \om_X \right )^2 \cdot \gamma \right ).
\end{equation*}
Since the second term in the supremum is always non-positive, we deduce by taking $\gamma = \om_X^{\delta-2}$ that:
\begin{equation*}
\deg_1(f) \le \normsigma{f^*\om_Y}{\om_X}.
\end{equation*}

For the upper bound, one is
reduced to estimate $\normsigma{f^*\alpha}{\om_X}$ for any given $\alpha\in\cnum^1(\Y)$.
For any strongly basepoint free Cartier class $\gamma\in \cBPF^{\delta-2}(\X)$, we have:
\begin{align*}
\frac{q_{\om_X,\gamma}(f^*\alpha)}{(\om_X^2\cdot\gamma)}
&\le
4\frac{(\om_X \cdot f^* \om_Y\cdot \gamma)^2 q_{f^*\om_Y, \gamma} (f^* \alpha)}{(\om_X^2\cdot \gamma)^2}
\\
&\le
(4C_\delta^2)\, \deg_1(f)^2
{q_{f^*\om_Y, \gamma} (f^* \alpha)}
\\
&\le
(4C_\delta^2)\, \deg_1(f)^2
{q_{\om_Y, f_*\gamma} (\alpha)},
\end{align*}

The first inequality follows from~\eqref{eq:7654}; the second 
from  Siu's  inequalities~\eqref{eq:siu_bpf} $\gamma \le C_\delta (\gamma\cdot \om_X^2) \om_X^{\delta-2}$;  and the last one from~\eqref{eq:f*BPF} and the projection formula.

We conclude that 
\[
\normsigma{f^*\alpha}{\om_X}^2 
\le 
(4C_\delta^2)\, \deg_1(f)^2\, \normsigma{\alpha}{\om_Y}^2. 
\]
by taking supremum over $\gamma$.
\end{proof}

The situation for the push-forward is more involved.
Recall that an operator between Banach spaces is said to be closed if its graph is closed.  
 We refer to \cite[\S 2.6]{brezis} for basic notions on unbounded operators.

\begin{prop} \label{prop_continuity_pushforward} 
When $\dim X = \dim Y = d$, the  push-forward $f_*\colon  \nums(\X) \to \nums(\Y)$
defines an unbounded closed and densely defined operator
whose domain contains the subspace $\Vect(\nef^1(\X)) \supset\cnum^1(\X)  $, 
and satisfies:
\begin{equation}
\label{eq:bdnormspush}
\normsigma{f_*\alpha}{\om_Y}
\le C \deg_{d-1}(f) \normsigma{\alpha}{\om_X}
\end{equation}
for all $\alpha\in\nef^1(\X)$,
where  $C$ is a positive constant depending only on $d$. 
\end{prop}

\begin{proof}
Recall that $f_*$ is a linear map from $\wnum^1(\X)$ to $\wnum^1(\Y)$.
The domain of  $f_*\colon  \nums(\X) \to \nums(\Y)$ is by definition the set of classes $\alpha\in\nums(\X)$
such that $f_*\alpha$ lies in $\nums(\X)$. It contains
$\cnum^1(\X)$ hence $f_*$ is densely defined. 

Suppose that $\alpha_n$ is a sequence of classes in the domain of $f_*$, that 
$\alpha_n\to \alpha$, and $f_*\alpha_n\to \beta$ in $\nums(\X)$. Then these convergences
hold in the weak topology, and  $f_*$ being weakly continuous we get $\beta = f_*\alpha$
hence $f_*$ is closed.

Suppose that $\alpha$ is nef. Since $f_*\alpha$ is again nef,~\eqref{eq:666} yields:
\begin{align*}
\normsigma{f_*\alpha}{\om_Y}
& \le \sqrt{2}C_d (f_*\alpha\cdot\om_Y^{d-1})
\\
&
\le \frac{\sqrt{2}dC_d}{(\om_X^d)}\deg_{d-1}(f)  (\alpha\cdot \om_X^{d-1}) 
\\
& \le \sqrt{2}dC_d\, \deg_{d-1}(f) \normsigma{\alpha}{\om_X},
\end{align*}
which concludes the proof.
\end{proof}

\begin{rem}
There is a difficulty in estimating directly $((f_*\alpha)^2\cdot\gamma)$ when $\alpha$ is not a nef class.
If $\alpha$ is Cartier, then it decomposes as a difference of two nef classes 
but the $\normsigma{\cdot}{\om}$-norm of these nef classes are not
comparable to the one of $\alpha$.
It follows that the push-forward operator is unlikely to be bounded on $\nums(\X)$. 
Note that when $f$ is birational, we have $f_* = (f^{-1})^*$ so that $f_*$ is bounded.
\end{rem}

\subsection{Dynamical degrees are spectral radii}
In this section, we let $f \colon  X \dashrightarrow X$ be any dominant (hence generically finite) rational self-map.
Let $\om$ be any big and nef $b$-class on $X$. Recall the definition of dynamical degrees from \S\ref{sec:dynamical_degrees}.

For any bounded operator $L\colon E\to E$ on a Banach space, we let $\rho(L)=\lim_{n\to\infty} \| L^n\|^{1/n}$ be its spectral radius.

In the next result, we endow $\Vect(\BPF^k(\X))$ with the norm $\normbpf{\cdot}$.
\begin{thm}\label{thm:dey-degree-interpret}
For any rational dominant self-map $f \colon  X \dashrightarrow X$, we have
\begin{align*}
\lambda_k(f)
&=\rho \left(f^*|\numbpf{k}\right) = \rho \left(f^*|\Vect(\BPF^k)\right)
=\rho \left(f^*| \numbpfs{k}\right) 
\\
&=\rho \left(f_*|\numbpf{d-k}\right)=\rho (f_*| \Vect(\BPF^{d-k}))
=\rho (f_*| \numbpfs{d-k}), 
\end{align*}
and
\[
\lambda_1(f)= \rho(f^*| \nums ).\]
\end{thm}

\begin{proof}
The theorem is a consequence of Propositions~\ref{prop:pullBPF},~\ref{prop:pushBPF},~\ref{prop_pullback_bpfs} and~\ref{prop_continuity_pullback}, and the facts that $f_*$ and $f^*$ preserves
BPF classes, see Theorem~\ref{thm:extension} and Corollary~\ref{cor:thm-extension}.
\end{proof}

\begin{thm}\label{thm:construct-eigenvec-pushpull}
For any integer $0\le k\le d$, there exist
\begin{enumerate}
\item[(i)]a non-zero class $\theta^*_k \in \BPF^k (\X)$ such that $f^*\theta^*_k=\lambda_k(f)\, \theta^*_k$; and
\item[(ii)] a non-zero class $\theta_{*,k} \in \BPF^k (\X)$ such that $f_*\theta_{*,k}=\lambda_{d-k}(f)\, \theta_{*,k}$.
\end{enumerate}
\end{thm}

\begin{rem}
The eigenclasses $\theta^*_k$ (and $\theta_{*,k}$) are expected not to belong to $\numbpf{k}(\X)$ in general. 
\end{rem}

\begin{proof}
Our construction of $\theta^*_k$ is classical in the theory of operators preserving convex cones, see~\cite{zbMATH03142405}.
We consider the 
following formal power series  $\Theta^*_{k}(t)=\sum_{n\ge0} t^n f^{n*}\om^k$, and observe that  
 \[T^*_k(t) := (\Theta^*_{k}(t)\cdot\om^{d-k}) = \sum_{n\ge0} t^n \deg_{\om,k}(f^n)~.\]
Observe that since $a\deg_{k,\om}(f^n)$ forms a sub-multiplicative sequence for some $a>0$, then
for any $\epsilon>0$, one can find a positive constant $C>0$ such that 
$C^{-1}\lambda_k(f)^n\le \deg_{\om,k}(f^n)\le C( \lambda_k(f)+\epsilon)^n$. It follows that the radius of convergence of
$\Theta^*_{k}(t)$ is equal to $\lambda_k(f)^{-1}$. For any $|t|<\lambda_k(f)^{-1}$, the class $\Theta^*_{k}(t)$
belongs to $\BPF^k(\X)$. 

For any $0<t<\lambda_k(f)^{-1}$, write $\theta^*_{k}(t)=\Theta^*_{k}(t)/T^*_k(t)$. A direct computation shows:
\[
t\times f^* \theta^*_k(t) = \theta^*_k(t) - \frac{\om^k}{T^*_k(t)}~.
\]
Now observe that $C^{-1}\lambda_k(f)^n\le \deg_{\om,k}(f^n)$ implies
$T^*_k(t)\to\infty$ as $t$ increases to $\lambda_k(f)^{-1}$. Since $(\theta^*_k(t)\cdot\om^{d-k})=1$, the classes $\theta_k^*(t)$ induce elements in the dual $\numbpf{d-k}(\X)^*$ which are bounded (in the unit ball), by Banach-Alaoglu's theorem, one can extract a sequence converging weakly to some BPF
class $\theta^*_k$ which satisfies $\lambda_k(f)^{-1}\,  f^* \theta^*_k = \theta^*_k$ by continuity.
\end{proof}

%%%%%%%%%%%%%%%%%%%%%%%%%%%%%%%%%%%%%%%%%%%%%%%%%%%%%%%%%%%%

\section{The spectral gap on $\nums(\X)$ when $\lambda_1(f)^2>\lambda_2(f)$}

We analyze the spectral properties of the action by pull-back of $f^*\colon\nums(\X)\to\nums(\X)$
for a dominant rational map $f\colon X \dashrightarrow X$ satisfying $\lambda_1(f)^2>\lambda_2(f)$.

%%%%%%%%%%%%%%%%%%%%%%%%%%%%%%%%%%%%%%%%%%%%%%%%%%%%%%%%

\subsection{Spectral theory of Banach operators}\label{sec:spectral-general}

Let $(E,\|\cdot\|)$ be any Banach space and $f\colon E \to E$ be any bounded linear operator. 
The spectrum $\Spec(f)$ of $f$  is defined as the set of complex numbers $\lambda $ for which the operator $f - \lambda \Id$ is not invertible.
It is a non-empty compact subset of $\C$.
Recall that the spectral radius $\rho(f) = \lim_{n\to\infty} \| f^n\|^{1/n}$ is equal to the supremum of $|\lambda|$ where $\lambda$ ranges over all elements in $\Spec(f)$, see \cite[Chapter 8]{yosida}.

A \emph{quasi-eigenvector} associated to a given $\lambda\in\Spec(f)$ is by definition a sequence of vectors $u_n$  such that  $\| u_n\| =1$ and $\|f u_n - \lambda u_n\| \to 0$.
Every point in the boundary of $\Spec (f)$ admits a quasi-eigenvector. Indeed, choose $\lambda_n \notin\Spec(f)$ such that $\lambda_n\to\lambda$, and write 
$T_n=  (f - \lambda_n \Id)^{-1}$.
We have
$(f - \lambda \Id) \circ T_n = \Id + (\lambda_n - \lambda) T_n$
hence
$\| (f - \lambda_n \Id)^{-1} \|\to \infty$, and  
there exists a sequence $w_n\in E$ such that $\|w_n\| =1$ and
\[
v_n  = \frac{(f - \lambda_n \Id)^{-1} w_n}{\| (f - \lambda_n \Id)^{-1} \|}
\]
where $\| v_n\|$ is bounded away from $0$. Then $u_n=\frac{v_n}{\|v_n\|}$ is a quasi-eigenvector.

\smallskip

Recall that an injective bounded operator $f\colon E \to F$ has closed range 
if and only if  $\| f (v)\| \ge c \| v\|$ for some $c>0$. One implication is trivial and the converse is the closed graph theorem due to
Banach, see~\cite[\S 2.3]{brezis}.

\smallskip

We shall also use the spectral decomposition theorem, see~\cite[Theorem 5.6.1]{hille_phillips}.

\begin{thm}\label{thm_spectral_decomp}
 Suppose that the spectrum of $f$ is the union of two disjoint compact sets $K_1$ and $K_2$.
Then there exist  two closed subspaces $E_1$ and $E_2$ such that 
$E= E_1 \oplus E_2$, and $\Spec(f|_{E_1})= K_1 $ and $\Spec(f|_{E_2})= K_2 $.
\end{thm}

\begin{proof}
This is a direct consequence of the holomorphic functional calculus.  
Let $\phi$ be any holomorphic function defined in a (disconnected) neighborhood of $\Spec(f)$, which is identically $1$ in a neighborhood of $K_1$ and 
$0$ near $K_2$. Note that we may suppose $\phi^2 = \phi$.
Pick a holomorphic function $\psi$ such that $\psi(z)=z$ near $K_1$ and $\psi \equiv 0$ near $K_2$.

Since $\phi^2 = \phi$, the operator $\pi_1:= \phi(f)$ defines a continuous projection
onto a closed subspace $E_1$ which is $f$-invariant. The restriction $f_1$ of $f$ to $E_1$
is equal to  $f_1 = \phi(f) \psi (f) = \pi_1 \circ \psi(f)$, and by the spectral formula we have
$\Spec(\psi (f)) = K_1 \cup \{0\}$ hence $\Spec(f_1) \subset K_1 \cup \{0\}$.

To conclude we need to argue that  $\Spec(f_1) \subset K_1$ if $0\notin K_1$. 
Suppose $0\notin K_1$, and let $\psi_-$ be a holomorphic function 
which is equal to $1/z$ near $K_1$ and $0$ near $K_2$. Then
$\psi \, \psi_- = \phi$ then $\psi(f) \, \psi_-(f) = \pi_1$ hence $f_1$ is invertible. 
This shows $0\notin\Spec(f_1)$ and completes the proof.
\end{proof}

%%%%%%%%%%%%

\subsection{The eigenvector $\theta^*_1$ belongs to $\nums(\X)$}

In the next statement  $\theta^*_1$ denotes any eigenvector obtained as in Theorem~\ref{thm:construct-eigenvec-pushpull}.
We shall see in the next sections that under the assumption  $\lambda_1(f)^2>\lambda_2(f)$ this eigenclass is unique (up to normalization).

\begin{thm}\label{thm:eigenvector-nums}
For any dominant rational self-map $f \colon X\dashrightarrow X$ such that $\lambda_1(f)^2>\lambda_2(f)$
the eigenvector $\theta^*_1$ belongs to $\nums(\X)$ and satisfies $(\theta^*_1\cdot \theta^*_1)=0$.
\end{thm}

We set 
\[
\Theta^*_1(t) = \sum_{n\ge0} t^n f^{n*}\om, \,
T^*_1(t)= \sum_{n\ge0} t^n \deg_{1,\om}(f^n), 
\text{ and }
\theta^*_1(t) = \frac{\Theta^*_1(t)}{ T^*_1(t)}
 \]
as in the proof of Theorem~\ref{thm:construct-eigenvec-pushpull}. Note that we proved that $\theta^*_1(t_n)\to \theta^*_1$ weakly for a sequence $t_n < 1/\lambda_1(f)$ satisfying $\lim_{n\rightarrow +\infty} t_n = 1/\lambda_1(f)$.

The proof of the theorem relies on the following general computation.
\begin{prop} \label{prop_square_estimates} 
Suppose that $\lambda_1(f)^2>\lambda_2(f)$.
Then there exists $C>0$ such that 
for any $0<t'\le t <\lambda_1(f)^{-1}$,
we have:
\begin{equation*}
0 \leq \theta^*_1(t)\cdot \theta^*_1(t')  \leq
\frac{C}{T^*_1(t')}\, \om^{2}.
\end{equation*}
\end{prop}

\begin{proof}[Proof of Proposition \ref{prop_square_estimates}]
We first expand the product $\Theta^*_1(t)\cdot \Theta^*_1(t')$:
\begin{equation*}
\Theta^*_1(t)\cdot \Theta^*_1(t') = \sum_{n,m} t^n (t')^m f^{n*} (\om) \cdot f^{m*}(\om) .
\end{equation*}
Suppose that $t\ge t'$.
Since each term in the sum is pseudo-effective, we have:
\begin{equation*} 
\Theta^*_1(t)\cdot \Theta^*_1(t') 
\le 
2 \sum_n t^{n} \left(\sum_{\substack{p + q = n\\ p\ge q}}
  f^{p*} (\om) \cdot f^{q*}( \om) \right).
\end{equation*}

By Siu's  inequalities~\eqref{eq:siu_bpf}, there exists $C>0$ such that
$f^{(p-q)*}( \om)  \le C \deg_1(f^{p-q}) \om$. 
Since $f^{q*}$ preserves pseudo-effectivity, we infer
\begin{equation*}
f^{p*}( \om)  \le C \deg_1(f^{p-q}) f^{q*}\om~.
\end{equation*}
Similarly, we have
\begin{equation*}
f^{q*} (\om^2) \le C \deg_2(f^q) \om^2,
\end{equation*}
hence
\begin{align*}
\Theta^*_1(t)\cdot \Theta^*_1(t')  
&\le 
2C \sum_n t^{n} \left(\sum_{\substack{p + q = n \\
p \ge  q
}}  \deg_1(f^{p-q}) f^{q*} \om^{2}\right)
\\
&
\le 2C' \sum_n t^{n} \left(\sum_{\substack{p + q = n \\
p \ge  q
}}  \deg_1(f^{p-q}) \deg_2(f^{q})\right) \om^{2}
\\
&=
2C' \sum_n t^{n} \left(\sum_{q=0}^{[n/2]}  \deg_1(f^{n-2q}) \deg_2(f^{q})\right) \om^{2}
\\
&=
2C' \sum_m \deg_1(f^m)  \left(\sum_{q\ge 0} t^{m+2q} \deg_2(f^{q})\right) \om^{2}
\\
&\le
2C'' \sum_m \deg_1(f^m) t^m \left(\sum_{q\ge 0} t^{2q} (\lambda'_2)^{q}
\right) \om^{2}
\\
&\le C''' \left(\sum_m \deg_1(f^m) t^m\right) \om^2.
\end{align*}
where we used the estimate $\deg_2(f^{q})
\le 
C 
(\lambda'_2)^{q}
$
for some $\lambda_2(f)< \lambda'_2 < \lambda_1(f)^2$.
\end{proof}

\begin{proof}[Proof of Theorem~\ref{thm:eigenvector-nums}]
We first prove that $\theta^*_1$ belongs to $\nums(\X)$.
Pick any sequence $t_n\to\lambda_1(f)^{-1}$
such that we have the weak convergence $\theta^*_1(t_n)\to\theta^*_1$.
As $T^*_1(t_n)\to\infty$, the preceding estimate implies 
\[\normbpfs{\theta^*_1(t_n) \cdot \theta^*_1(t_m)}\to0 \text{ as } n,m\to\infty.\]
Introduce the truncations $\widetilde{\theta^*_1}(t_n):=\frac1{T^*_1(t_n)} \sum_{m=0}^{r_n} t_n^m f^{m*} (\om)\in\cNef(\X)$, 
where $r_n$ is chosen large enough so that 
\[
\normsigma{\widetilde{\theta^*_1}(t_n) - \theta^*_1(t_n)}{\om}\le \frac1n
~.
\]
Then~\eqref{eq:bound-norms2} gives
$\normbpfs{\widetilde{\theta^*_1}(t_n) \cdot \widetilde{\theta^*_1}(t_m)}\to0$ as $n,m\to\infty$.
Since $\theta^*_1(t_n)$ is nef and satisfies $(\theta^*_1(t_n)\cdot\om^{d-1})=1$ for all $n$, we have
$\sup_n \normsigma{{\theta^*_1}(t_n)}{\om}<\infty$ by~\eqref{eq:666}, hence 
$\sup_n \normsigma{\widetilde{\theta^*_1}(t_n)}{\om} < \infty$.

From the compactness
result stated in Theorem~\ref{thm_compactness}, a subsequence of $\widetilde{\theta^*_1}(t_m)$ converges
in the Banach space $\nums(\X)$ (with respect to the norm $\normsigma{\cdot}{\om}$), and the
limit is necessarily
$\theta^*_1\in\nums(\X)$ since it is the weak limit of $\widetilde{\theta^*_1}(t_m)$.

We obtain
\[
 \normsigma{\theta^*_1(t_n) - \theta^*_1}{\om}
 \le
\normsigma{\widetilde{\theta^*_1}(t_n) - \theta^*_1}{\om}
+\normsigma{\widetilde{\theta^*_1}(t_n) - \theta^*_1(t_n)}{\om}
\to 0
\]
 therefore
$\normbpfs{\theta^*_1 \cdot \theta^*_1}=0$.
This concludes the proof of the theorem.
\end{proof}

%%%%%%%%%%%%

\subsection{Quasi-eigenvectors}
We suppose as in the previous section that  $f \colon X\dashrightarrow X$ is a dominant rational self-map such that $\lambda_1(f)^2>\lambda_2(f)$.
Let  $\theta^*_1\in\nums(\X)$ be the eigenvector for the eigenvalue $\lambda_1(f)$ constructed above.
Write $\nums(\X)_\C = \nums(\X)\otimes_\R\C$.

\begin{prop}\label{prop:quasi-gap}
Let $(\alpha_n)\in\nums(\X)_\C$ be a quasi-eigenvector of $f^*\colon \nums(\X)_\C \to \nums(\X)_\C$ associated to an eigenvalue $\lambda$. 
If $|\lambda|>\sqrt{\lambda_2(f)}$, then we have $\alpha_n\to \frac{\theta^*_1}{\normsigma{\theta^*_1}{\om}}$ in $\nums(\X)_\C$.
\end{prop}

\begin{proof}
Since Cartier classes are dense in $\nums(\X)$, we can suppose that $\alpha_n \in \cnum^1(\X)$.
Observe that the sequence $\alpha_n \cdot \bar\alpha_m$ is bounded in $\numbpfs{2}(\X)$.
Indeed, by~\eqref{eq:bound-norms2} we have:
\begin{equation*}
\normbpfs{\alpha_n \cdot \bar \alpha_m} \le 3 \normsigma{\alpha_m}{\omega} \normsigma{\alpha_n}{\omega} = 3.
\end{equation*}
Let us show that $\normbpfs{\alpha_n\cdot \bar\alpha_m}$ converges uniformly in $m,n$ to zero. 
The computation runs as follows. We may suppose that
\[\max_{j\le n} \normsigma{f^{j*}\alpha_n - \lambda^j \alpha_n}{\om} \le \frac1n  \]
for all $n$.
For all $\gamma\in\cBPF^{d-2}(\X)$ such that $(\gamma\cdot\om^2)=1$, we then obtain using the previous inequality and \eqref{eq:bound-norms}:
\[
\left|( f^{n*}\alpha_m \cdot f^{n*} \bar\alpha_n\cdot\gamma) - |\lambda|^{2n} (\alpha_m \cdot \bar\alpha_n\cdot\gamma) \right|
\le 6 \frac{|\lambda|^n}n
\]
for all $m\ge n$,
so that 
\begin{align*}
\varlimsup_{n\to\infty}
|(\alpha_m \cdot \bar\alpha_n\cdot\gamma) |
&\le
\varlimsup_{n\to\infty}
\frac1{\lambda^{2n}}|(\alpha_m \cdot \bar\alpha_n\cdot f^n_* \gamma)|
\\
&\le 
\varlimsup_{n\to\infty}
\normbpfs{(\alpha_m \cdot \bar\alpha_n)}\, \frac{(f^n_* \gamma\cdot\om^2)}{\lambda^{2n}} 
\\
&\le
\varlimsup_{n\to\infty}
C' \, \frac{\deg_2(f^n)}{\lambda^{2n}} \longrightarrow 0
\end{align*}
where the last inequality follows from Siu's  inequalities~\eqref{eq:siu_bpf}.
Up to taking a subsequence, we can suppose that $\alpha_n$ converges weakly to $\alpha \in \wnum^1(\X)$. 
Since $\normbpfs{\alpha_n \cdot \omega}$ is bounded, we have by Theorem~\ref{thm_compactness} that the sequence $\alpha_n$ converges to $\alpha$ in $\nums(\X)_\C$ .
Since $f^*\alpha=\lambda\alpha$
we have $f^*(\alpha\cdot\bar{\alpha}) = |\lambda|^2 (\alpha\cdot\bar{\alpha})$.

But the spectral radius of $f^*\colon\numbpfs{2}(\X)\to\numbpfs{2}(\X)$ is equal to $\lambda_2(f) < |\lambda^2|$
hence $(\alpha\cdot\bar{\alpha})=0$ in $\numbpfs{2}(\X)$. In a similar way, we have $(\theta^*_1\cdot\bar{\alpha})=0$ in $\numbpfs{2}(\X)$
hence $\alpha$ and $\theta^*_1$ are proportional by Corollary~\ref{cor_hodge_index}.
\end{proof}

%%%%%%%%%%%%%%%%%%%%%%%%%%%%%%%%%%%%%%%%%%%%%%%%%%%%%%%%%%%

\subsection{Spectral analysis of $f^*\colon\nums(\X)\to\nums(\X)$}

\begin{thm}\label{thm:spectrum-nums}
Let  $f \colon X\dashrightarrow X$ be any dominant rational self-map such that $\lambda_1(f)^2>\lambda_2(f)$. 
Denote by $\theta^*_1\in\nums(\X)$ any non-zero nef class satisfying $f^*\theta^*_1=\lambda_1(f)\theta^*_1$.
  
Then there exists a closed hyperplane $E\subset \nums(\X)$ which is $f^*$-invariant, such that 
$\nums(\X)= \R\cdot\theta^*_1 \oplus E$, and $\Spec(f^*|_E)\subset \bar{\D}(0,\sqrt{\lambda_2(f)})$.
\end{thm}

\begin{rem}
The decomposition above is canonical since $E$ consists of those classes $\alpha$ such that $\normsigma{f^{n*}(\alpha)}{\om}\le C (\lambda'_2(f))^n$
for some $C>0$, some $\lambda'_2(f)<\lambda_1(f)$, and all $n\ge0$.
\end{rem}

\begin{proof}
Since $\theta^*_1$ is an eigenvector in $\nums(\X)$ of eigenvalue $\lambda_1(f)$ the latter point belongs to  $\Spec(f^*)$.
We claim that $\Spec(f^*)\setminus \{\lambda_1(f)\}\subset\bar{\D}(0,\sqrt{\lambda_2(f)})$. Indeed, if it were not the case, then we could
find $\lambda\neq\lambda_1(f)$ lying in the boundary of  $\Spec(f^*)$, and of norm $|\lambda|>\sqrt{\lambda_2(f)}$.
This would contradict Proposition~\ref{prop:quasi-gap}.

Let $B=\Spec(f^*)\cap\bar{\D}(0,\sqrt{\lambda_2(f)})$, and
apply the spectral theorem (Theorem \ref{thm_spectral_decomp}): we get an $f^*$-invariant decomposition $\nums(\X) = F \oplus E$ with $\Spec(f^*|_F) = \{\lambda_1\}$, and  $\Spec(f^*|_E)= B$.

Suppose by contradiction that $\R \cdot \theta^*_1\subsetneq F$, and
consider the restriction operator $f^*|_{F}$. By Hahn-Banach's theorem, we may find a closed hyperplane $F'\subset F$
such that $F = \R\cdot \theta^*_1 \oplus F'$, and write $g= f^*- \lambda_1 \Id$.

We claim that $g|_{F'}$ has closed range. To see this it is sufficient to prove that 
 $\| g(\alpha)\| \ge c \|\alpha\|$ for some $c>0$ and all $\alpha\in F'$. If it were not the case, we would find a sequence of vectors $\alpha_n $ in $F'$
 such that $\normsigma{f^*\alpha_n- \lambda_1 \alpha_n}{\om}\to0$ and $\normsigma{\alpha_n}{\om}=1$, contradicting
Proposition~\ref{prop:quasi-gap}. 

Since $g(F')$ is $f^*$-invariant and closed, it admits a quasi-eigenvector with eigenvalue $\lambda_1(f)$ which implies
$\theta^*_1 \in g(F')$ again by Proposition~\ref{prop:quasi-gap}. 
Pick $\alpha\in F'$ such that $f^*\alpha = \lambda_1(f) \alpha + \theta^*_1$.
Then we get
\[
(f^n)^* (\theta^*_1\cdot\bar{\alpha}) = 
\lambda_1(f)^{2n} (\theta^*_1\cdot\bar{\alpha}),
\text{ and } 
(f^n)^* (\alpha \cdot \bar{\alpha}) =  \lambda_1(f)^{2n} (\alpha \cdot \bar{\alpha})
\]
so that  $(\theta^*_1\cdot\bar{\alpha})=0$, and $(\alpha \cdot \bar{\alpha}) = 0$
in $\numbpfs{2}(\X)$ since 
\[\left(\normbpfs{f^{n*}|_{\numbpfs{2}(\X)}}\right)^{1/n} \to \lambda_2(f) < \lambda_1(f)^2.\]
Since $\alpha$ and $\theta^*_1$ are not proportional, 
we get a contradiction using Theorem~\ref{thm_hodge_index}.
\end{proof}

%%%%%%%%%%%%

\begin{proof}[Proof of Theorem~\ref{thm_int_gap}]

Pick any big and nef class $\om\in\num^1(X)$ and identify it with its Cartier $b$-divisor class determined in $X$.
By the previous theorem, we may write $\om = t \theta^*_1 + v$ with $t\in \R$ and $v\in E$. 
Since the spectral radius of $f^*|_E$ is bounded by $\sqrt{\lambda_2(f)}$ we get 
that $\normsigma{f^{n*}v}{\om} \le C\, \lambda^n$ for any $\lambda_1(f)>\lambda >\sqrt{\lambda_2(f)}$,
hence
\[
\deg_{1,\om}(f^n) = (f^{n*}\om\cdot\om^{d-1}) = t \lambda_1(f)^n (\theta^*_1\cdot\om^{d-1}) + O(\lambda^n).
\]
Since $\deg_{1,\om}(f^n)^{1/n} \to \lambda_1(f) > \lambda$, we have $t>0$. This concludes the proof of our main theorem.
\end{proof}

\begin{cor}\label{cor:existence dual form}
There exists $\ell\in\nums(\X)^*$ such that for any $\alpha\in\nums(\X)$, we have
\[
\lambda_1(f)^{-n}\, f^{n*} \alpha \to  \ell(\alpha)\, \theta^*_1
\text{ in } \nums(\X).\]
Moreover, we have $\ell(\alpha)\ge 0$ whenever $\alpha \ge 0$.
\end{cor}
\begin{rem}
This corollary is the higher dimensional generalization of~\cite[Corollary 3.6]{boucksom_favre_jonsson_deggrowth}
where the linear form $\alpha \mapsto (\alpha\cdot \theta_*)$ in op. cit. is replaced by $\ell$.
\end{rem}
\begin{proof}\label{rem:existence dual form}
Since the hyperplane $E$ given by Theorem~\ref{thm:spectrum-nums} is closed, we may find $\ell\in\nums(\X)^*$ whose kernel is equal to $E$ that we may normalize so that $\ell(\theta^*_1)=+1$.
Then for any $\alpha\in\nums(\X)$, we have
\[
\lambda_1(f)^{-n}\, f^{n*} \alpha \to  \ell(\alpha) \theta^*_1
\text{ in } \nums(\X).\] The last statement follows by density from the fact that $f^*$ preserves 
psef classes in $\cnum^{1}(\X)$.
\end{proof}

%%%%%%%%%%%%%%%%%%%%%%%%%%%%%%%%%%%%%%%%%%%%%%%%%%%%%%

\section{Polynomial maps of $\A^d$}

In this section, we prove Theorem~\ref{thm_poly}.

We fix some affine coordinates $x_1,\cdots, x_d$ on the affine space $\A^d_\K$, and let $\Pp^d_\K=\proj(\K[x_0,x_1,\cdots,x_d])$
be the canonical projective compactification of $\A^d_\K$. Denote by $H_\infty =\Pp^d_\K\setminus \A^d_\K$ the hyperplane at infinity, 
and let $\cP^d$ be the Riemann-Zariski space of $\Pp^d_\K$.

A pseudo-valuation on the ring $K[x_1, \cdots, x_d]$ is a function \[v\colon K[x_1, \cdots, x_d] \to  \R \cup \{ +\infty\}\]
such that $v(0)=+\infty$, $v(PQ)= v(P)+ v(Q)$, and $v(P+Q) \ge \min \{v(P),v(Q)\}$ for any pair of polynomials, 
and $v(c) =0$ for any $c\in K^*$. When $v(P)$ is finite for all non-zero polynomials, then $v$ is a (real-valued) valuation.
We refer to~\cite{vaquie} for generalities on these notions.

For any polynomial $P(x)= \sum_I a_I x^I\in K[x_1, \cdots, x_d]$, we set:
\[
-\deg(P) := - \max \{ |I|, a_I\neq0\}~.
\]
Observe that $-\deg$ is a valuation which takes negative values on any non-constant polynomial. 

We shall prove the following analog of~\cite[Theorem A']{favre_jonsson_eigenvaluations}.
\begin{thm}\label{thm:eigenvaluation}
Let $f\colon \A^d_\K\to \A^d_\K$ be any proper polynomial map such that $\lambda_1(f)^2>\lambda_2(f)$.
There exists a valuation $v$ on the ring $\K[x_1,\cdots, x_d]$ which is trivial on $\K$, takes non-positive values on non-zero polynomials,
satisfies $\min_{1\le i\le d}\{ v(x_i)\}=-1$ and $v(P\circ f) = \lambda_1(f)\, v(P)$
for all $P\in\K[x_1,\cdots,x_d]$.
\end{thm}

Let us explain how to infer Theorem~\ref{thm_poly} from the previous result. 
Consider the value group $\Gamma= v(K(x_1,\cdots,x_d)^*)\subset (\R,+)$. 
By Abhyankar's inequalities, see e.g.~\cite[Theorem~9.2]{vaquie}, we have $1\le \dim_\Q (\Gamma\otimes_\Z \Q) \le d$ so that we may find
$d$ polynomials $P_1, \cdots, P_d \in K[x_1, \cdots , x_d]$ 
 such that $\Gamma\otimes_\Z \Q = \Q \cdot v(P_1) + \ldots + \Q  \cdot v(P_d)$. Note that $v(P_i\circ f)\in \Gamma$ for all $i$ so that there exists rational numbers $q_{ij}$ such that
\[
\lambda_1(f) v(P_i) = \sum_{j=1}^d q_{ij} v(P_j)
\]
for all $i=1, \cdots, d$.
It follows that $\lambda_1(f)$ is an eigenvalue of a  $d \times d$ square matrix with rational coefficients, hence
is an algebraic number of degree $\le d$ over $\Q$.

\begin{rem}
As observed by M. Jonsson, when $v$ is an Abhyankar valuation, then its value group is isomorphic to $\Z^k$ for some $1\le k\le d$, therefore $\lambda_1(f)$ is an algebraic integer of degree $\le d$, see~\cite[Lemma~8.7]{MR3330767}.
\end{rem}

Our strategy is to prove that the opposite of the linear functional $\ell$ given by Corollary~\ref{cor:existence dual form} induces a valuation which is invariant by $f$. 
To better exploit our understanding of the action of $f$ on the space of Cartier $b$-divisors, it will be convenient to interpret valuations as functionals
on the set of fractional ideal sheaves. A similar presentation of valuations over a closed point is given in~\cite[\S 2.1]{MR3060755}. 

\subsection{Fractional ideal sheaves and valuations}
Recall that a fractional ideal sheaf $\fA$ on an algebraic variety $X$ is a coherent $\cO_{X}$-submodule of the sheaf of
meromorphic functions $\mathcal{K}_X$, so that locally near any point $g\cdot \fA \subset \cO_{X}$ for some $g\in\cO_{X}$. 

We shall say that $\fA$ is co-supported in a subvariety $Z$ when $\fA_x = \cO_{X,x}$ for any $x \notin Z$.
Any finite family of rational functions $\phi_1, \cdots, \phi_n$ generates a fractional ideal sheaf on $X$ that we denote 
$\langle \phi_1, \cdots, \phi_n\rangle$ so that $\fA_x = \phi_1 \cdot \cO_{X,x} + \cdots + \phi_n \cdot \cO_{X,x}\subset \mathcal{K}_{X,x}$ for all $x$.
For instance, we have  $ \cO_{\Pp^d_\K}(H_\infty)=\langle 1, x_1, \cdots, x_d\rangle$.

\smallskip 

Since $\cO_{\Pp^d_\K}(1)$ is very ample, for any fractional ideal sheaf $\fA$ on $\Pp^d_\K$ which is co-supported in $H_\infty$, 
one can find a family of polynomials $P_1, \cdots, P_n\in K[x_1, \cdots, x_d]$, and $\mu\in\N$ such that $\bigcap_1^n P_i^{-1}(0) \cap \A^d_\K = \emptyset$, and
\begin{equation}\label{eq:fraco}
\fA\cdot \cO_{\Pp^d_\K}(\mu H_\infty) = \langle P_1,\cdots, P_n\rangle
~.\end{equation}

Let $v\colon \K[x_1,\cdots, x_d] \to \R \cup \{ +\infty \}$ be any valuation as defined above.
We attach to $v$ a real-valued function $L_v$ on the set of fractional ideals sheaves
co-supported in $H_\infty$ as follows.
Take any fractional ideal sheaf $\fA$
satisfying~\eqref{eq:fraco}. Then we set
\begin{equation}
L_v(\fA)
 = \min_i\{v(P_i)\} - \mu\, \min_{1\le j\le d}\{v(x_j)\}. 
\end{equation}
This definition does not depend on the choice of generators, and
we have $L_v(\fA\cdot\fB)=L_v(\fA)+L_v(\fB)$, for any fractional ideal sheaves $\fA,\fB$, and
$L_v(\fA)\ge L_v(\fB)$ if $\fA\subset\fB$.
We also have 
\begin{lem}
$L_v(\fA+\fB)=\min\{L_v(\fA),L_v(\fB)\}$ for any fractional ideal sheaves $\fA,\fB$.
\end{lem}
\begin{proof}
The result follows immediately from the following observation. 
Let $P_1, \cdots, P_n$ be any collection of polynomials such that  $\bigcap_1^n P_i^{-1}(0) \cap \A^d_\K = \emptyset$. 
Then we have
\[
\min_i\{v(P_i)\}
= v \left (\sum_{i=1}^n a_i P_i \right )\]
for a generic choice of coefficients $(a_1, \ldots, a_n)\in \K^n$.

To see this, we may suppose that $v(P_1) =\min_i\{v(P_i)\}$, and consider the map $\Phi$
sending $a \in \K^n$ to the class of $(\sum a_i P_i)/P_1$ in the residue field of the valuation
$v$. This map is $\K$-linear and  non zero since $\Phi(1,0,\cdots,0) \neq 0$.
Any point $a$ for which $\Phi(a)\neq 0$ then satisfies $v(P_1)= v \left (\sum_{i=1}^n a_i P_i \right )$.
\end{proof}

Conversely, let $L$ be any real-valued function on the set of non-zero fractional ideal sheaves co-supported  in $H_\infty$
and such that the following conditions hold:
\begin{itemize}
\item[(V1)]
$L(\fA\cdot\fB)=L(\fA)+L(\fB)$, and $L(\fA+\fB)=\min\{L(\fA),L(\fB)\}$ for any fractional ideal sheaves $\fA,\fB$;
\item[(V2)]
$L(\fA)\ge L(\fB)$ if $\fA\subset\fB$;
\item[(V3)]
$L(\cO_{\Pp^d_\K})=0$ and
$L(\langle 1, x_1, \cdots, x_d\rangle )<0$.
\end{itemize}
For any polynomial $P$, we set 
\begin{equation}\label{eq:defvalo}
v_L(P) =\lim_{n\to \infty} L\left( \fA_n(P)\right),\end{equation}
where $\fA_n(P) =   (P) + \cO_{\Pp^d_\K}(-n H_\infty)$ (that is locally $\fA_n(P)$ is generated by $P$ viewed as a rational function on $\Pp^d_\K$ and an equation of the divisor $nH_\infty$). 
Note that $\fA_{n+1}(P)\subset \fA_n(P)$ so that
$L\left( \fA_n(P)\right)$ is increasing by (V2) and the limit in~\eqref{eq:defvalo} exists in $\R\cup \{+\infty \}$.
\begin{prop}\label{prop:end-proof} 
For any function $L$ as above, 
the function $v_L \colon \K[x_1,\cdots,x_d] \to \R \cup \{+\infty \}$ 
defined by \eqref{eq:defvalo} is a 
pseudo-valuation on $K[x_1,\cdots, x_d]$ which satisfies
$\min_j\{v(x_j)\}<0$.
\end{prop}
\begin{proof}
Up to multiplying $L$ by a suitable positive real constant, we may suppose $L(\langle 1, x_1, \cdots, x_d\rangle ) = -1$.

Recall  our convention on homogeneous coordinates
$(x_1, \cdots , x_d) = [1:x_1: \cdots : x_d]$.
Choose any point $p$ on the hyperplane at infinity. We may suppose
$p=[0:1: \cdots:0]$,
and choose local coordinates $[z_1:1:z_2: \cdots: z_d]$ near that point so that $H_\infty = \{ z_1 =0\}$,
and $x_1= 1/z_1$, $x_i = z_i/z_1$ for $d \ge i\ge 2$.

Fix any non-zero polynomial $P$. In the above coordinates $P$ can be written as a quotient $\bar{P}(z)/z_1^{\deg(P)}$
with $\bar{P}(z) = P(1/z_1, z_2/z_1, \cdots, z_d/z_1)z_1^{\deg(P)}$
so that  
\[
\fA_n(P)= \hat{\fA}_n(P) \cdot \cO_{\Pp^d_\K}(\deg(P) H_\infty)
\text{ with }
\hat{\fA}_n(P) = \left( \bar{P}, z_1^{n+\deg(P)}\right)~.\]
When $P=c$ is a non-zero constant, then $\fA_n(c) = \cO_{\Pp^d_\K}$ hence
$v_L(c) = 0$ by (V3), hence $v_L$ is the trivial valuation on $K$.

Pick any two non-zero polynomials $P$ and $Q$. We may suppose that 
$l:= \deg(P)\ge \deg(Q) =k$.  Consider as above the polynomials 
$\bar{P}$ and $\bar{Q}$.
The key observation is that 
\[
\hat{\fA}_n(P+Q) = \left( \bar{P}+ \bar{Q}z_1^{l-k}, z_1^{n+l}\right)
\subset
\hat{\fA}_n(P) +
\hat{\fA}_n(Q)\cdot  \cO_{\Pp^d_\K}(-(l-k) H_\infty)
\]
hence $\fA_n(P+Q) \subset \cO_{\Pp^d_\K}(l H_\infty)\cdot \hat{\fA}_n(P) + \cO_{\Pp^d_\K}(l - (l-k) H_\infty) \cdot \hat{\fA}_n(Q) = \fA_n(P) + \fA_n(Q)  $ and we obtain:
\begin{align*}
L(\fA_n(P+Q)) 
&\ge 
L(\fA_n(P) + \fA_n(Q))
\\
&=
 \min\{L(\fA_n(P)), L(\fA_n(Q))\}
\end{align*}
which implies $v_L(P+Q) \ge \min \{v_L(P), v_L(Q)\}$.

Observe now that 
\[
\hat{\fA}_{2n}(PQ) \subset\hat{\fA}_n(P)\cdot \hat{\fA}_n(Q)
\subset 
\hat{\fA}_{n-\deg(P)-\deg(Q)}(PQ) 
\]
for all $n$ sufficiently large so that (V1) implies $v_L(PQ)= v_L(P)+v_L(Q)$.

Finally we have 
$\fA_n(x_1) + \cdots + \fA_n(x_d) = \cO_{\Pp^d_\K}(H_\infty)$ 
hence
$
\min_j \{v_L(x_j)\}
= L(\langle 1, x_1, \cdots, x_d\rangle ) 
<0$.
This concludes the proof.
\end{proof}

\begin{rem}
Although we shall not need it, the maps $v \mapsto L_v$ and $L \mapsto v_L$ are inverse one to the other.
\end{rem}

\subsection{Fractional ideal sheaves and $b$-divisor classes}
Pick any proper birational map $\pi\colon X\to \Pp^d_\K$ which is an isomorphism over $\A^d_\K$, and 
let $\fA$ be any fractional ideal sheaf on $X$ co-supported in $\pi^{-1}(H_\infty)$. We  attach to $\fA$ a Cartier $b$-divisor $Z(\fA)\in\cnum^1(\cP^d)$ as follows. 

Choose any log-resolution of $\fA$, i.e. any proper birational map $\pi\colon Y\to X$ such that the fractional ideal sheaf $\fA\cdot \cO_Y$ is locally principal.
Since $\fA$ is assumed to be co-supported in $\pi^{-1}(H_\infty)$, we may suppose $\pi$ to be an isomorphism over $\A_\K^d$, and write $\fA\cdot \cO_Y= \cO_Y(W(\fA)_Y)$ for
a unique divisor $W(\fA)_Y$ supported at infinity, and let $Z(\fA)$ be the Cartier $b$-class determined by $W(\fA)_Y$ in $Y$. 

Pick any proper birational morphism $\pi'\colon X' \to \Pp^d_\K$. Then $Z(\fA)_{X'}$ is represented by the unique divisor $W(\fA)_{X'}$ supported in $(\pi')^{-1}(H_\infty)$
such that $\fA\cdot \cO_{X'}= \cO_{X'}(W(\fA)_{X'})\cdot \fB$ where $\fB$ is an ideal sheaf whose co-support has codimension at least $2$.

Note that the incarnation of the Cartier $b$-class $Z(\fA)$ in $\Pp^d_\K$
is determined by the divisor $W(\fA)_{\Pp} := -\deg(\fA) [H_\infty]$, and that 
the following properties hold:
\begin{align*}
Z(\fA\cdot\fB) &= Z(\fA)+ Z(\fB),\\
Z(\fA) &\le  Z(\fB) \text{ when } \fA \subset \fB~,
\end{align*}
and  
\[\ord_{H_\infty}(W(\fA+ \fB)_{\Pp}) = \max \{\ord_{H_\infty}(W(\fA)_{\Pp}), \ord_{H_\infty}(W(\fB)_{\Pp})\}
~.\]
We may naturally pull-back fractional ideal sheaves co-supported in $H_\infty$ by our given proper polynomial map $f\colon \A^d_\K\to\A^d_\K$. 
Fix any birational proper morphism $\pi\colon X \to \Pp^d_\K$ that is an isomorphism over $\A^d_\K$, and such that 
the map $F\colon X\to \Pp^d_\K$ induced by $f$ becomes regular.
If $\fA$ is co-supported in $H_\infty$, then $F^*\fA$ is a fractional ideal sheaf on $X$ co-supported in $\pi^{-1}(H_\infty)$, and 
we have $f^{*}Z(\fA)= Z(F^*\fA)$.

Introduce the Cartier $b$-class $\cL_\infty\in\cBPF^{d-1}(\cP^d)$ determined by a line in the hyperplane $H_\infty$.
Observe that $\cL_\infty = c_1(\cO_{\Pp^d_\K}(1))^{d-1}$ 
so that for any fractional ideal sheaf we have
\[
-\deg(\fA) = - (\cL_\infty\cdot Z(\fA)) =- \ord_{H_\infty}(W(\fA)_{\Pp})
~.\]
Set \[
L(\fA) := - (\cL_\infty\cdot f^{*}Z(\fA))~.
\]
\begin{lem}\label{lem:pull-val}
The function $L$ satisfies all conditions (V1)--(V3), and for any polynomial $P$, we have
\[ v_{L}(P) = -\deg(P\circ f)~.\]
\end{lem}
\begin{proof}
Observe that we have
\begin{align*}
(\cL_\infty\cdot f^{*}Z(\fA)) &= (c_1(\cO_{\Pp^d_\K}(1))^{d-1} \cdot Z(F^*\fA)_{\Pp})
\\
&
= - \ord_{H_\infty}(W(F^*\fA)_{\Pp})
~,\end{align*}
so that $L$ satisfies the three properties (V1)--(V3).

Note that $F^* \cO_{\Pp^d_\K}(-NH_\infty)$ and $\pi^* \cO_{\Pp^d_\K}(-H_\infty)$ are locally principal ideal sheaves in $X$.
Since $F$ is proper, they have the same co-support hence we may find a positive integer $N\in \N^*$ such that 
$F^* \cO_{\Pp^d_\K}(-NH_\infty) \subset \pi^* \cO_{\Pp^d_\K}(-H_\infty) 
$ and 
$ \pi^* \cO_{\Pp^d_\K}(-NH_\infty) \subset F^* \cO_{\Pp^d_\K}(-H_\infty)$. 

This yields the inclusion
$\fA_{nN}(P\circ f)\cdot \cO_X
\subset
F^* \fA_{n}(P)$. Translating this in terms of Cartier $b$-divisors, we obtain 
\begin{equation}\label{eq:Fproper1}
Z(\fA_{nN}(P\circ f)) \le f^* Z(\fA_{n}(P))
\end{equation}
 hence 
$(\cL_\infty\cdot Z(\fA_{nN}(P\circ f))) \le (\cL_\infty\cdot f^* Z(\fA_{n}(P)))$
 since $\cL_\infty$ is a BPF class. This implies
 $v_{L}(P) \le  -  \deg(P\circ f)$. 
We obtain in an analogous way 
\begin{equation}\label{eq:Fproper2}
Z(\fA_{n}(P\circ f)) \ge f^* Z(\fA_{nN}(P))
\end{equation}
which proves  $v_{L}(P) \ge  - \deg(P\circ f)$ as desired. 
\end{proof}
\subsection{Proof of Theorem~\ref{thm:eigenvaluation}}
By Corollary~\ref{cor:existence dual form},
there exists a continuous linear function $\ell\colon\nums(\cP^d)\to\R$ such that 
\[
\lim_{n\to\infty} \lambda_1(f)^{-n}\, f^{n*}Z = \ell(Z) \, \theta^*_1
~,\]
for any Cartier class $Z\in\cnum^1(\cP^d)$.
Normalizing the class $\theta^*_1$ so that $(\cL_\infty\cdot\theta^*_1)=1$, and setting  $L_\star(\fA) = - \ell(Z(\fA))$
for any fractional ideal sheaf co-supported in $H_\infty$,
we obtain $L_\star(\fA) = \lim_n L_n(\fA)$ for all $\fA$ where
\[
L_n(\fA) := \frac{- (\cL_\infty\cdot f^{n*}Z(\fA))}{\lambda_1(f)^{n}}~.
\]
By Lemma~\ref{lem:pull-val} and by continuity the functional $L_\star$  
again satisfies (V1),  (V2), $L_\star(\cO_{\Pp^d_\K})=0$ and it follows from Theorem~\ref{thm_int_gap}
that
\[
L_\star(\langle 1, x_1, \cdots, x_d\rangle )
= 
- \lim_{n\to\infty} \lambda_1(f)^{-n}\, (f^{n*}H_\infty\cdot \cL_\infty)
=- \lim_{n\to\infty} \frac{\deg(f^n)}{\lambda_1(f)^{n}} <0~
\]
so that (V3) holds.

We may thus apply Proposition~\ref{prop:end-proof}. To simplify notation, write $v_\star = v_{L_\star}$, and take any non-zero polynomial $P$.
Since $L_m(\fA_n(P))$ is increasing in $n$ to $v_{L_m}(P) \le 0$, we
get 
\[v_\star(P) = \lim_n L_\star( \fA_n(P)) = \lim_n \lim_m  L_m( \fA_n(P))\le0
\] 
hence $v_\star$ is a valuation taking non-positive values. 

Finally by Corollary~\ref{cor:existence dual form}, and applying~\eqref{eq:Fproper1} and~\eqref{eq:Fproper2}, we obtain
\[
v_\star( P\circ f) = 
- \lim_n \ell (Z(\fA_n(P\circ f)))
= \lambda_1 v_\star(P)~.
\]
This concludes the proof.
 
%%%%%%%%%%%%%%%%%%%%%%%%%%%%%%%%%%%%%%%%%%%%%%%%%%%%%%%%%%%%%%%%%%

\bibliographystyle{alpha}
\bibliography{ref_degre_dynamique}
\addcontentsline{toc}{section}{References}

\end{document}